\documentclass[11pt]{amsart}
\usepackage{hyperref}

\usepackage{amsmath,amsfonts,amssymb}
\usepackage{mathrsfs}
\usepackage{verbatim}
\usepackage{amsthm}

\newtheorem{theorem}{Theorem}[section]
 \newtheorem{corollary}[theorem]{Corollary}
 \newtheorem{lemma}[theorem]{Lemma}
 \newtheorem{proposition}[theorem]{Proposition}
 \theoremstyle{definition}
 \newtheorem{definition}[theorem]{Definition}
 \theoremstyle{remark}
 \newtheorem{remark}[theorem]{Remark}
 \newtheorem{ex}[theorem]{Example}
 \numberwithin{equation}{section}

\def \bC {\mathbb C}

\def \bH {\mathbb H}

\def \bN {\mathbb N}

\def \bQ {\mathbb Q}
\def \bR {\mathbb R}
\def \bS {\mathbb S}

\def \bZ {\mathbb Z}

\def \cC {\mathcal C}
\def \cD {\mathcal D}

\def \cF {\mathcal F}

\def \cH {\mathcal H}

\def \cK {\mathcal K}
\def \cL {\mathcal L}
\def \cM {\mathcal M}

\def \cR {\mathcal R}
\def \cS {\mathcal S}
\def \cR {\mathcal R}
\def \cR {\mathcal R}

\def \cX {\mathcal X}
\def \cY {\mathcal Y}

\def \fg {\mathfrak g}

\def \fz {\mathfrak z}

\def \fS {\mathfrak S}

\def \fU {\mathfrak U}

\def \tr {\text{\rm tr}\,}
\def \trN {\text{\rm tr}_N}
\def \trK {\text{\rm tr}_K}
\def \diag {\text{\rm Diag}}

\def \supp {\text{\rm supp}\, }

\def \princ {\text{\rm princ} }
\def\Op{\text{\rm Op}}
\def\id{\text{\rm I}}
\def\Gh{\widehat G}
\def\dom {\text {\rm Dom}}

\def\eps{\varepsilon}

\def\div{\text{\rm div}}

\def\Tend#1#2{\mathop{\longrightarrow}\limits_{#1\rightarrow#2}}
\def\TendWeak#1#2{\mathop{\rightharpoonup}\limits_{#1\rightarrow#2}}

\def \sL{\mathscr L}

\def\and{\quad\mbox{and}\quad}

\begin{document}

\title[Defect measures on graded Lie groups]
{Defect measures on graded Lie groups}
\author{V. Fischer and C. Fermanian-Kammerer}
\date{June 2017}

\maketitle

\makeatletter
\renewcommand\l@subsection{\@tocline{2}{0pt}{3pc}{5pc}{}}
\makeatother

\begin{abstract}
In this article, we define a generalisation of microlocal defect measures (also known as H-measures) to the setting of graded nilpotent Lie groups. 
This requires to develop the notions of homogeneous symbols and classical pseudo-differential calculus adapted to this setting and defined via the representations of the groups.
Our method relies on the study of  the $C^*$-algebra of 0-homogeneous  symbols. Then, we compute microlocal defect measures for concentrating and oscillating sequences, which also requires to investigate the notion of oscillating sequences in graded Lie groups.  
Finally, we discuss compacity compactness approaches in the context of graded nilpotent Lie groups.  
\end{abstract}

$ $


$ $


$ $

\tableofcontents

\section{Introduction}

The aim of this article is to develop a new approach for analysing the lack of compactness of bounded square integrable families on nilpotent  Lie groups. The idea is to generalise the notions of  microlocal defect measures (MDM)  which were originally defined and studied  in the Euclidean setting
 by Luc Tartar and Patrick G\'erard independently  in the 90's, see 
\cite{tartar} and
\cite{gerard_91} respectively;
the original definition of~\cite{gerard_91} is recalled  in the next paragraph.
These notions have given a new insight on compensated compactness theorems as developed by Di Perna and Lions~\cite{di_perna_lions} for example. Such theorems allow one to pass to the limit on quadratic quantities appearing in mechanics for example and of the form $(Au_k,v_k)$ for weakly converging subsequences $(u_k)$ and $(v_k)$ provided their MDM's and the operator $A$ satisfy convenient assumptions. Such descriptions were already possible in some cases thanks to the {\it Div-Curl } Lemma \cite{francfort_murat}. The analysis of MDM's extends the range of applications of the ideas which are behind this lemma. 
It happens that the  Div-Curl lemma has recently been studied in the context of Lie groups: see the article~\cite{baldi_franchi} in the context of the Heisenberg group and~\cite{baldi_franchi_tchou_tesi} for Carnot groups. This motivates the investigation of MDM's and of compensated compactness questions on Lie groups.
Notice also that the MDM and their semi-classical counterpart (also called semi-classical or Wigner measures~\cite{GMMP}) have also proved useful for the analysis of pde-s in different context, from quantum chemistry  to theory of chaos and analysis of quantum ergodicity, including control theory
\cite{lebeau1,lebeau2}. And questions not so far to those of the latter references are now addressed in the context of sub-laplacians (see~\cite{colin_de_verdiere_trelat}).

\smallskip

Before discussing the setting of nilpotent Lie groups in more details, let us recall briefly 
the definition of an MDM in the Euclidian case.
On an open subset $\Omega$ of $\bR^n$, an MDM of a sequence 
$(u_k)_{k\in \bN}$ of functions  converging weakly to a distribution $u$ in $L^2(\Omega,loc)$ is a positive measure 
$\gamma$ on $ \Omega\times \bS^{n-1}$
such that, up to extraction of a subsequence, we have the convergence
\begin{equation}
\label{eq_limit_intro}
(A(u_{k_j}-u),u_{k_j}-u)_{L^2}
  \longrightarrow_{j\to\infty}
  \int_{\Omega\times \bS^{n-1}} a_0(x,\xi) \, d\gamma(x,\xi),
\end{equation}
for any test pseudodifferential operator $A$ of order 0
with principal symbol $a_0$;
by test pseudo-differential operators we mean for instance operators in the classical H\"ormander calculus, properly supported, defined through inverse Fourier transform by 
\begin{equation}\label{eq:Rnpseudo}
Au(x)=\int_{ \bR^n} a_0(x,\xi) {\rm e}^{ix\cdot\xi} \widehat u(\xi) d\xi,\;\; u\in{\mathcal S}(\bR^n),\;\; x\in\bR^n,
\end{equation}
(for simplicity, in the formula above, we have assumed that the symbol of $A$  exactly $a_0$ and have no term of lower order)
This extends easily to closed manifolds by replacing $\bS^{n-1}\times \Omega$
with the spherical co-tangent bundle, 
and also to vector-valued functions by taking suitable traces. Note that in the approach of Luc Tartar~\cite{tartar} test operators are the ones which are tensor products of  multiplication operators with Fourier multipliers, which is enough to construct the measure $\gamma$. 

\smallskip 

As we see from the paragraph above, the notion of (Euclidean) MDM relies
on microlocal analysis and the theory of pseudodifferential operators 
which has been developed since the  70's  in the Euclidean setting  (see~\cite{hormander,U}, or the review books~\cite{DimassiSjostrand,EvansZworski}).
The development of a pseudo-differential theory on nilpotent Lie groups has been the purpose of works by several authors, see e.g. \cite{geller,gellercore,rbeals2,bealsgreiner,christ_etc,taylor_84}. 
The recent contribution of the second author with her collaborators in~\cite{BFG} for the Heisenberg group has been followed by the monograph~\cite{R+F_monograph} of the first author and her collaborator, where they have defined a pseudodifferential calculus on graded nilpotent Lie groups. As in the Euclidean context, they are defined thanks to inverse Fourier transform with the major difference that  the Fourier transform of a function at a (unitary irreducible) representation  is an operator on the space of the representation.
Consequently
the  symbols of pseudo-differential operators introduced in 
\cite{R+F_monograph}  are measurable fields of operators on $G\times \Gh$ where $\Gh$ is the unitary dual, i.e. the set of unitary irreducible representations of $G$ modulo equivalence, that we shall denote by $\pi(x)$, $x\in G$.  Then, the operator $A$ whose symbol is the field of operator $\sigma(x,\pi)$ satisfies 
$$Au(x)=\int_{ \Gh} {\rm tr} \left( \pi(x)  \sigma (x,\pi)  \widehat u(\pi)\right) d\mu(\pi),\;\; u\in{\mathcal S}(G),\;\; x\in G,$$
which is the analogue of~(\ref{eq:Rnpseudo}) (precise definitions are given in sections~\ref{sec_preliminary} and~\ref{sec_pseudo}). 
It is on this latter result that relies the construction of MDM's  developed hereafter. However, we shall need to extend the theory and we develop in section~\ref{sec_hom} the classes of homogeneous symbols and of classical symbols, together with the notion of principal symbol. 

\medskip

We will see that the MDM's on a  graded nilpotent Lie group $G$ defined in this paper are non commutative objects, and this is not surprising since the Fourier transform is operator-valued.
More precisely,
 a MDM on $G$ consists of 
 a positive measure $\gamma$ on $G\times (\Gh/\bR^+)$
 and a field $\gamma$-integrable field $\Gamma$ of trace-class operators on $G\times (\Gh/\bR^+)$ (the quotient set $\Gh/\bR^+$ is defined by use of dilations and the class of $\pi\in \Gh$ will be denote by $\dot\pi$, see section~\ref{subsec_polar_dec_Gh}). Then, the analogue of formula~(\ref{eq_limit_intro}) which is proved in section~\ref{sec_defect_meas} writes 
\begin{equation}\label{eqniwprem}
(A(u_{k_j}-u),u_{k_j}-u)_{L^2}
  \longrightarrow_{j\to\infty}
  \int_{G\times (\Gh/\bR^+)} {\rm tr}\left(\sigma_0(x,\dot\pi) \Gamma(x,\dot\pi)\right) d\gamma(x,\dot\pi),
\end{equation}
where $\sigma_0$ is the principal symbol of the operator $A$, as defined in section~\ref{sec_hom}.
 Note that operator-valued measures have been introduced in semi-classical settings since the 90's \cite{nier,miller,Fermanian2micro,FermanianShocks} and, more recently, in the context of quantum ergodicity~\cite{AnantharamanMacia,AnantharamanMaciaSurv,AnRiv,MaciaTorus}. 
 As in the Euclidean case, one can develop applications to compensated compactness as discussed in section~\ref{sec_app}. 

\medskip
 
An important difference with the Euclidean context is the lack of Garding inequality, and this prevents us to adapt the main steps of the proof of existence of Euclidean MDM's given in \cite{gerard_91}.
This is overcome by the use of $C^*$-algebra formalism and the notion of state. 
More precisely, we prove that convergent limits of quantities similar to the left hand side of \eqref{eqniwprem} 
in the nilpotent setting define positive linear functionals on the algebra of homogeneous symbols of order 0, 
and that these linear functionals  extend to states on certain $C^*$ algebras. 
Understanding the spectrum of these $C^*$ algebras and decomposing these states yield the main result of the paper.
Reformulating our proof in the case of the abelian  group $\bR^n$
(which is trivially a graded Lie group) gives a new proof of the result in the Euclidean case, albeit too sophisticated.
Note also that,
although this paper belongs to the fields of
 micro-local analysis and non-commutative analysis,  
many of its tools and techniques relies on 
the progress of the last four decades in harmonic analysis on Lie groups:
for instance, in understanding the properties 
of spectral multipliers in sub-laplacians on nilpotent Lie groups (or more generally positive Rockland operators), 
or in describing homogeneous convolution operators in terms of their kernels.

\medskip 

Finally, we want to emphasize that the nilpotent Lie groups considered in this paper and in \cite{R+F_monograph} are graded, but this is a natural restriction.
Indeed, the class of graded nilpotent Lie groups contains the class of stratified Lie groups
 (also called Carnot groups in more geometric contexts), the prime example being the Heisenberg group.
Graded or even stratified groups are  the groups usually appearing in applications of analysis on nilpotent Lie groups, for instance in the study of operators sums of squares of vector fields, as in~\cite{colin_de_verdiere_trelat}.

\medskip

Our article is organised as follows. Section~\ref{sec_preliminary} is devoted to definitions on graded Lie groups and to analysis results that we shall use. Then we recall in section~\ref{sec_pseudo} the definition of pseudodifferential operators on graded Lie groups and we introduce in section~\ref{sec_hom} the notion of homogeneous symbols and of principal symbols. 
In Section \ref{sec_0homsymb}, we analyse the $C^*$-algebras formed by 0-homogeneous symbols.
The core of the paper consists in sections~\ref{sec_defect_meas} where we prove the existence of  MDM and analyse the fundamental examples of concentrating and oscillating sequences. Then, in section~\ref{sec_app}, we link our results with compensated compactness theory and definition of {\it Curl} operators on Lie groups. 

\medskip

\noindent {\it Convention:}
In the paper, if $\cX$ and $\cY$ are Banach spaces, 
$\sL(\cX,\cY)$ denotes the Banach space of bounded linear applications form $\cX$ to $\cY$. 
If a linear operator $A$ is densely defined in a Banach space $\cX$ and valued in a Banach space $\cY$, 
then writing $A\in \sL(\cX,\cY)$ means that $A$ extends to a bounded operator $\cX\to\cY$
and that we identify the operator $A$ with its bounded extension on $\cX$
 (which is unique). 
 If $\cX=\cY$, we write $\sL(\cX,\cX)=\sL(\cX)$.
 If $\cX$ is a Hilbert space, we define by $\sL^1(\cX)$ the trace-class operators on $\cX$ and we set $\|A\|_{\sL^1(\cX)}=\tr |A|$.
 
 \medskip
 
 \noindent{\it Acknowledgements:} The authors give deep thanks to Vladimir Georgescu, the exchange they had with him proved to be determinant for the orientation of their work. They also warmly thank Philippe Biane  and Patrick G\'erard for fruitful discussions and they are indebted to the {\it Centre International de Rencontres Math\'ematiques} for its ``{\it Recherche En Bin\^ome}'' program that has helped them to close the writing of this article.


\section{Preliminaries: graded Lie groups }
\label{sec_preliminary}

In this section, after defining graded Lie groups, 
we recall their homogeneous structure, the definition of the Fourier tranfrom and results on the dual.
A complete description of the notions of graded and homogeneous nilpotent Lie groups may be found in \cite[ch1]{folland+stein_82} and \cite[ch3]{R+F_monograph}.

\subsection{Graded Lie groups}
\label{subsec_G}

We will be concerned with graded Lie groups $G$
which means that $G$ is a connected and simply connected 
Lie group 
whose Lie algebra $\mathfrak g$ 
admits an $\bN$-gradation
$\mathfrak g= \oplus_{\ell=1}^\infty \mathfrak g_{\ell}$
where the $\mathfrak g_{\ell}$, $\ell=1,2,\ldots$, 
are vector subspaces of $\mathfrak g$,
almost all equal to $\{0\}$,
and satisfying 
$[\mathfrak g_{\ell},\mathfrak g_{\ell'}]\subset\mathfrak g_{\ell+\ell'}$
for any $\ell,\ell'\in \bN$.
This implies that the group $G$ is nilpotent.
Examples of such groups are the Heisenberg group
 and, more generally,
all stratified groups (which by definition correspond to the case $\fg_1$ generating the full Lie algebra $\fg$).

We construct a basis $X_1,\ldots, X_n$  of $\fg$ adapted to the gradation,
by choosing a basis $\{X_1,\ldots X_{n_1}\}$ of $\mathfrak g_1$ (this basis is possibly reduced to $\emptyset$), then 
$\{X_{n_1+1},\ldots,  X_{n_1+n_2}\}$ a basis of $\mathfrak g_2$
(possibly $\{0\}$ as well as the others)
and so on.
Via the exponential mapping $\exp_G : \mathfrak g \to G$, we   identify 
the points $(x_{1},\ldots,x_n)\in \bR^n$ 
 with the points  $x=\exp_G(x_{1}X_1+\cdots+x_n X_n)$ in~ $G$.
Consequently we allow ourselves to denote by $C(G)$, $\cD(G)$ and $\cS(G)$ etc,
the spaces of continuous functions, of smooth and compactly supported functions or 
of Schwartz functions on $G$ identified with $\bR^n$,
and similarly for distributions with the duality notation 
$\langle \cdot,\cdot\rangle$.

This basis also leads to a corresponding Lebesgue measure on $\mathfrak g$ and the Haar measure $dx$ on the group $G$,
hence $L^p(G)\cong L^p(\bR^n)$.
The group convolution of two functions $f_1$ and $f_2$, 
for instance square integrable, 
is defined via 
$$
 (f_1*f_2)(x):=\int_G f_1(y) f_2(y^{-1}x) dy.
$$
The convolution is not commutative: in general, $f_1*f_2\not=f_2*f_1$.

The coordinate function $x=(x_1,\ldots,x_n)\in G\mapsto x_j \in \bR$
is denoted by $x_j$.
More generally we define for every multi-index $\alpha\in \bN_0^n$,
$x^\alpha:=x_1^{\alpha_1} x_2 ^{\alpha_2}\ldots x_{n}^{\alpha_n}$, 
as a function on $G$.
Similarly we set
$X^{\alpha}=X_1^{\alpha_1}X_2^{\alpha_2}\cdots
X_{n}^{\alpha_n}$ in the (complex) universal enveloping Lie algebra $\fU(\fg)$ of $\mathfrak g$.
Let us recall 
that a vector of $\mathfrak g$ defines a left-invariant vector field on $G$ 
and, more generally, 
that the universal enveloping Lie algebra $\fU(\fg)$ of $\mathfrak g$ 
is isomorphic with the left-invariant differential operators; 
we keep the same notation for the vectors and the corresponding operators. 
However if $X\in \fg$, then $\tilde X$ denotes the corresponding right invariant vector field. More generally, if $T\in \fU(\fg)$, we denote by $\tilde T$ the right-invariant differential operator.

For any $r>0$, 
we define the  linear mapping $D_r:\mathfrak g\to \mathfrak g$ by
$D_r X=r^\ell X$ for every $X\in \mathfrak g_\ell$, $\ell\in \bN$.
Then  the Lie algebra $\mathfrak g$ is endowed 
with the family of dilations  $\{D_r, r>0\}$
and becomes a homogeneous Lie algebra in the sense of 
\cite{folland+stein_82}.
We re-write the set of integers $\ell\in \bN$ such that $\fg_\ell\not=\{0\}$
into the increasing sequence of positive integers
 $\upsilon_1,\ldots,\upsilon_n$ counted with multiplicity,
 the multiplicity of $\fg_\ell$ being its dimension.
 In this way, the integers $\upsilon_1,\ldots, \upsilon_n$ become 
 the weights of the dilations and we have $D_r X_j =r^{\upsilon_j} X_j$, $j=1,\ldots, n$,
 on the chosen basis of $\fg$.
 The associated group dilations are defined by
$$
D_r(x)=
rx
:=(r^{\upsilon_1} x_{1},r^{\upsilon_2}x_{2},\ldots,r^{\upsilon_n}x_{n}),
\quad x=(x_{1},\ldots,x_n)\in G, \ r>0.
$$
In a canonical way,  this leads to the notions of homogeneity for functions and operators.
For instance
the degree of homogeneity of $x^\alpha$ and $X^\alpha$,
viewed respectively as a function and a differential operator on $G$, is 
$[\alpha]=\sum_j \upsilon_j\alpha_{j}$.
This also leads to the notion of homogeneous distribution.

\begin{ex}
The Haar measure is $Q$-homogeneous:
\begin{equation}
\label{eq_int_Gf(rx)dx}
r^{Q}  \int_G f(rx) dx =\int_G f(y) dy,
\end{equation}
where 
$$
Q:=\sum_{\ell\in \bN}\ell \dim \fg_\ell=\upsilon_1+\ldots+\upsilon_n,
$$
 is called the \emph{homogeneous dimension} of $G$.
\end{ex}

Recall that a \emph{homogeneous quasi-norm} on $G$ is a continuous function $|\cdot| : G\rightarrow [0,+\infty)$ homogeneous of degree 1
on $G$ which vanishes only at 0. This often replaces the Euclidean norm in the analysis on homogeneous Lie groups.
Any homogeneous quasi-norm $|\cdot|$ on $G$ satisfies a triangle inequality up to a constant:
$$
\exists C\geq 1, \quad \forall x,y\in G,\quad
|xy|\leq C (|x|+|y|).
$$
Any two homogeneous quasi-norms $|\cdot|_1$ and $|\cdot|_2$ are equivalent in the sense that
$$
\exists C>0, \quad \forall x\in G,\quad
C^{-1} |x|_2\leq |x|_1\leq C |x|_2.
$$
There is an analogue of polar coordinates on $G$:
\begin{proposition}
\label{prop_polar_coord}
Let $|\cdot|$ be a fixed homogeneous quasi-norm on $G$.
Then there is a (unique) positive Borel measure $\sigma$ on 
the unit sphere 
$\fS:=\{x\in G\, : \, |x|=1\}$,
such that for all $f\in L^1(G)$, we have
\begin{equation}
\label{formula_polar_coord}
\int_G f(x)dx
=
\int_0^\infty \int_\fS f(ry) r^{Q-1} d\sigma(y) dr
.
\end{equation}
\end{proposition}

There is also an analogue of the mean value theorem:
\begin{lemma}
\label{lem_mvt}
We fix $|\cdot|$ a homogeneous quasi-norm on $G$.
Then there exists a constant $C>0$ such that
for any $f\in C^1(G)$, $x\in G$,
we have
$$
|f(x) -f(0)|
\leq 
C
\sum_{j=0}^n |x|^{\nu_j} 
\sup_{y\in G} 
| X_j  f(y)|.
$$
\end{lemma}

\subsection{The dual of $G$ and the Plancherel theorem}
\label{subsec_Gh+plancherel}

Here we set some notations and recall some properties 
regarding the representations of the group $G$ 
(especially the Plancherel theorem)
and its enveloping Lie algebra $\fU(\fg)$.

In this paper,  we always assume that the representations of the group $G$ 
are strongly continuous and acting on separable Hilbert spaces.
Unless otherwise stated, the representations of $G$ will also be assumed unitary.
For a representation $\pi$ of $G$, 
we keep the same notation for the corresponding infinitesimal representation
which acts on the universal enveloping algebra $\fU(\fg)$ of the Lie algebra of the group.
It is characterised by its action on $\fg$:
\begin{equation}
\label{eq_def_pi(X)}
\pi(X)=\partial_{t=0}\pi(e^{tX}),
\quad  X\in \fg.
\end{equation}

The infinitesimal action acts on the space $\cH_\pi^\infty$
of smooth vectors, that is, the space of vectors $v\in \cH_\pi$ such that 
 the mapping $G\ni x\mapsto \pi(x)v\in \cH_\pi$ is smooth.

\begin{ex}
\label{ex_smooth_vectors}
Vectors of the form 
 $\pi(\phi)v$ where $\phi\in \cD(G)$ or $\cS(G)$ and $v\in \cH_\pi$
 are smooth.
\end{ex}
Here we have used the usual notation for 
the \emph{group Fourier transform} of a function  $f\in L^1(G)$
  at $\pi$: 
$$
 \pi(f) \equiv \widehat f(\pi) \equiv \cF_G(f)(\pi)=\int_G f(x) \pi(x)^*dx.
$$ 
We denote by $\Gh$ the unitary dual of $G$,
that is, the unitary irreducible representations of $G$ modulo equivalence and identify a unitary irreducible representation 
with its class in $\Gh$. The set $\Gh$ is naturally equipped with a structure of standard Borel space.

The Plancherel measure is the unique positive Borel measure $\mu$ 
on $\Gh$ such that 
for any $f\in C_c(G)$, we have:
\begin{equation}
\label{eq_plancherel_formula}
\int_G |f(x)|^2 dx = \int_{\Gh} \|\cF_G(f)(\pi)\|_{HS(\cH_\pi)}^2 d\mu(\pi).
\end{equation}
Here $\|\cdot\|_{HS(\cH_\pi)}$ denotes the Hilbert-Schmidt norm on $\cH_\pi$.
This implies that the group Fourier transform extends unitarily from 
$L^1(G)\cap L^2(G)$ to $L^2(G)$ onto 
$L^2(\Gh):=\int_{\Gh} \cH_\pi \otimes\cH_\pi^* d\mu(\pi)$
which we identify with the space of $\mu$-square integrable fields on $\Gh$.
Consequently \eqref{eq_plancherel_formula} holds for any $f\in L^2(G)$;
this formula is called the Plancherel formula.
Consequently, for any $\phi_1,\phi_2\in L^2(G)$, 
 the  quantity 
$\int_{\Gh} \tr |\widehat \phi_1(\pi)\widehat \phi_2(\pi)| d\mu(\pi)$ is finite and 
we have the Parseval formula
\begin{equation}
\label{eq_parseval_formula}
\int_G \phi_1 (x)\bar \phi_2(x) \ dx =\int_{\Gh}\tr\left(\widehat \phi_1(\pi)\widehat {\phi}_2(\pi)^*\right) d\mu(\pi).
\end{equation}
The orbit method furnishes an expression for the Plancherel measure $\mu$,
see \cite[Section 4.3]{corwingreenleaf}.
However we will not need this here.

The general theory on locally compact unimodular group of type I applies 
(see \cite{Dixmier_C*}):
let $\sL(L^2(G))$ be the space of bounded linear operators on $L^2(G)$
and 
let  $ \sL_L(L^2(G))$ be the subspace of those operators $T\in \sL(L^2(G))$ which are left-invariant, that is, commute with the left translation:
$$
T(f(g\cdot))(g_1) =(Tf)(gg_1), \quad f\in L^2(G), \ g,g_1\in G.
$$ 
Then there exists a field of bounded operators $\widehat T (\pi)\in \sL(\cH_\pi)$, $\pi\in \Gh$, such that
$$
\forall f\in L^2(G),\quad
\cF_G(Tf)(\pi) = \widehat T(\pi) \ \widehat f(\pi)
\quad \mbox{for} \ \mu-\mbox{almost all}\ \pi\in \Gh.
$$
Moreover the operator norm of $T$ is equal to 
$$
\|T\|_{\sL(L^2(G))}=\sup_{\pi\in \Gh} \|\widehat T(\pi)\|_{\sL(\cH_\pi)}.
$$
The supremum here has to be understood as the essential supremum with respect to the Plancherel measure $\mu$.
We denote by $L^\infty(\Gh)$ 
the space of fields of  operators $\sigma_\pi\in \sL(\cH_\pi)$, $\pi\in \Gh$, with 
$$
\|\sigma\|_{L^\infty(\Gh)}:=
\sup_{\pi\in \Gh} \|\sigma_\pi\|_{\sL(\cH_\pi)}<\infty,
$$
modulo equivalence under the Plancherel measure $\mu$.
Conversely, any  field in $L^\infty(\Gh)$ naturally yields a left-invariant bounded operator on $L^2(G)$.

By the Schwartz kernel theorem, 
any operator $T\in \sL_L(L^2(G))$ is a convolution operator 
and we denote by $T\delta_0\in \cS'(G)$ its convolution kernel:
$Tf = f * (T\delta_0)$, $f\in \cS(G)$.
We may call $T\delta_0$ the kernel of $T$ or of $\widehat T$.

We denote by $\cK(G)$ the space of convolution kernels of operators in $\sL_L(L^2(G))$
and we define the group Fourier transform of $T\delta_0$  as 
$$
\cF_G(T\delta_0)\equiv \widehat T.
$$
This extends the previous definition of the group Fourier transforms from $L^1(G)\cap \cK(G)$ or $L^2(G)\cap \cK(G)$ to $\cK(G)$ onto $L^\infty(\Gh)$.
The  group Fourier transform also extends for instance to 
 the space of convolution kernels $\cK_{a,b}(G)$
 of operators in $\sL_L(L^2_a(G), L^2_b(G))$.

\subsection{Dilations on $\Gh$}
\label{subsec_polar_dec_Gh}

Since the group $G$ is a (connected, simply connected) nilpotent Lie group, one can use the orbit method to construct unitary irreducible representations of $G$ (see e.g. \cite{corwingreenleaf}):
with this method,
to any  linear functional  $\varphi\in \fg^*$, 
one associates a class $\pi_\varphi\in \Gh$
of equivalent unitary irreducible representation.
Any element of $\Gh$ may be realised in this way
and two such classes in~$\Gh$ coincide when the linear functionals are on the same orbit for the co-adjoint action of~$G$ on~$\fg^*$.
In other words, one obtains a bijection
$\fg^*/G \longleftrightarrow \Gh$,
known as Kirillov's map.

The dilations of $G$ 
provide an action of $\bR^+=(0,\infty)$
on the Lie algebra $\fg$,
hence  on $\fg^*$ by duality, 
and one checks easily that 
quotienting by the co-adjoint action of $G$ and by the  $\bR^+$-action 
commutes.
Hence one obtains an action of $\bR^+$ on $\fg^*/G$.
The dilations also provide an action of~$\bR^+$ on the group $G$
thus on its dual via
\begin{equation}
\label{eq_def_rpi}
r\cdot \pi (x)
=
\pi(rx),\quad x\in G,\
\pi\in \Gh, \ r>0.
\end{equation}
One checks that 
Kirillov's map 
$\fg^*/G \longleftrightarrow \Gh$
is $\bR^{+}$-equivariant.

As usual, $\Gh$ is equipped with the hull-kernel topology 
and $\fg^*/G$ with the quotient Euclidean topology.
It is known \cite{brown} that  Kirillov's map is a homeomorphism.
One checks easily that, quotienting by the $\bR^+$-actions, 
the map 
$(\fg^*/G)/\bR^+ \longleftrightarrow \Gh/\bR^+$ 
is a homeomorphism.

\medskip

If $X\in \fg$ is of degree $d$, then it follows from
\eqref{eq_def_pi(X)} that we have
$$
(r\cdot \pi)(X)
=
\partial_{t=0}(r\cdot \pi)(e^{tX})
=
\partial_{t=0}\pi(re^{tX})
=
\partial_{t=0}\pi(e^{t D_r( X)})
=
\partial_{t=0}\pi(e^{t r^d X})
=r^d \pi(X).
$$
More generally for any $\alpha\in \bN_0^n$, we have:
\begin{equation}
\label{eq_rpiXalpha}
(r\cdot \pi)(X^\alpha) =r^{[\alpha]} \pi(X^\alpha),
\quad r>0.
\end{equation}
If $f\in L^1(G)$, 
then so does $f\circ D_{r^{-1}}$ 
and using \eqref{eq_int_Gf(rx)dx}, we have
\begin{eqnarray*}
(r\cdot \pi)(f) 
&=&
\int_G f(x) r\cdot\pi(x) ^* dx
=
\int_G f(x) \pi(rx) ^* dx
=
\int_G f\circ D_{r^{-1}} (x) \pi(x) ^* r^{-Q}dx
\\
&=&
r^{-Q} \pi\left(f\circ D_{r^{-1}}\right).
\end{eqnarray*}
More generally, 
using the properties of the group Fourier transforms, 
we obtain 
\begin{equation}
\label{eq_rcdotpif}
(r\cdot \pi)(f) =\pi (f_{(r)})
\quad\mbox{where}\ f_{(r)}=r^{-Q} f\circ D_{r^{-1}},
\quad r>0.
\end{equation}
for any $f$ in $L^1(G)$, $L^2(G)$ or $\cK_{a,b}(G)$.

Formula \eqref{eq_rcdotpif} and the Plancherel measure being unique, 
easily imply that for any positive measurable or integrable function $F$ on $\Gh$ and any $r>0$,
we have
\begin{equation}
\label{eq_dilation_mu}
 \int_{\Gh} F(r\cdot \pi)d\mu(\pi) 
 =
 r^{-Q}  \int_{\Gh} F(\pi)d\mu(\pi) .
\end{equation}

\bigskip

Let us fix a quasi-norm  $|\cdot|$ on $G$.
This yields a map on $\fg^*$ for which we keep the same notation. 
We set
$$
|[\varphi]| 
:=\inf\{|\varphi'|, \varphi'\in [\varphi] \}
=\min\{|\varphi'|, \varphi\in  [\varphi] \}, 
$$
where $[\varphi]$ denotes the co-adjoint class of $\varphi\in \fg^*$.
Naturally, the map $[\varphi]\mapsto |[\varphi]|$ is continuous  $\fg^*/\bR^+\to [0,\infty)$.
We set for each $\pi\in \Gh$, 
\begin{equation}
\label{eq_def|pi|}
|\pi|
:=\inf\{|\varphi|, \varphi\in \fg^*\ \mbox{s.t.} \ \pi\equiv \pi_\varphi\}
=\min\{|\varphi|, \varphi\in \fg^*\ \mbox{s.t.} \ \pi\equiv \pi_\varphi\},
\end{equation}
where $\pi_\varphi$ is the class of unitary irreducible representations of $G$ 
corresponding to the co-adjoint orbit containing $\varphi$.
This mapping is nothing else than the map $[\varphi]\mapsto |[\varphi]|$
transported by the Kirillov mapping.
There the function $\pi\mapsto |\pi|$ is continuous $\Gh\to [0,\infty)$.

One checks easily that the map $[\varphi]\mapsto |[\varphi]|$,
 and therefore the map $\pi\mapsto |\pi|$
 respect the dilations in the following way:
\begin{equation}
\label{eq_normGh_dilation}
|[\varphi(r^{-1}\cdot)]| = r |[\varphi]| 
\quad\mbox{and}\quad
|r\cdot\pi|=r|\pi| \ ,
\quad r>0,\ \pi\in \Gh, \ \varphi\in \fg^*.
\end{equation}
Furthermore
$$
|[\varphi]|=0 \ \Longrightarrow\ \varphi=0 \ , 
\qquad
|\pi|=0 \ \Longrightarrow \ \pi =1 \ , 
$$
where
$1=\pi_0$ denotes the trivial representation of $G$.

This induces a continuous surjection from  the `sphere' in $\Gh$
\begin{equation}
\label{eq_S||}
\Sigma_1=\Sigma_{1,|\cdot|}:=\{\pi\in \Gh, \ |\pi|=1\},
\end{equation}
onto  $(\Gh/\bR^+)\backslash\{1\}$.
This shows the following property:
\begin{lemma}
\label{lem_Kmap_quotient}
The $\bR^+$-quotient of Kirillov's map 
is a homoemorphism between the compact spaces
$(\fg^*/G)/\bR^+$ and $\Gh/\bR^+$.
\end{lemma}

\begin{remark}
In the case of the Heisenberg group, 
this sphere is reduced to two points.
\end{remark}

Having defined a unit sphere on $\Gh$, 
we can state a polar decomposition:

\begin{lemma}
\label{lem_pol_dec_Gh}
Let $|\cdot|$ be a quasi-norm on $G$
and let $|\cdot|$ be the associated mapping on $\Gh$ 
and the sphere $\Sigma_1=\Sigma_{1,|\cdot|}$ defined in \eqref{eq_S||} above.
\begin{enumerate}
\item 
The linear mapping 
$\displaystyle{
f\longmapsto \int_{1\leq|\pi|\leq e} 
f (|\pi|^{-1}\cdot\pi)  |\pi|^{-Q}
d\mu(\pi),
}$
defines a continuous positivity-preserving linear mapping on
the Banach space $\cC(\Sigma_1)$ of continuous function on the compact space $\Sigma_1$.
We denote  by $\varsigma =\varsigma_{|\cdot|}$ the corresponding Radon measure.
\item 
For any measurable function $F:G\to [0,\infty)$,
we have:
$$
\int_G F(\pi) d\mu=
\int_{\Sigma_1\times (0,\infty)} F(r\cdot \pi) d\varsigma(\pi) r^{Q-1} dr,
$$ 
\item 
In particular, 
if $u\in L^2(G)$, then 
$$
\int_{\Sigma_1 \times (0,\infty)}
\|\hat u(r\cdot\pi)\|_{HS}^2 
d \varsigma(\pi) r^{Q-1} dr
=\|u\|_{L^2(G)}^2.
$$
\end{enumerate}
\end{lemma}

\begin{proof}
One checks easily Part (1).
As the Plancherel measure is the unique measure such that the Plancherel formula holds, it suffices to show Part (3)
which follows from simple manipulations and  \eqref{eq_dilation_mu}.
\end{proof}

\begin{remark}
\label{rem_lem_pol_dec_Gh}
Adapting the ideas of the proof of the polar decomposition on a homogeneous Lie group
(see e.g. \cite[3.1.7]{R+F_monograph}), one can show the following property:
having fixed $|\cdot|$, 
if $F \in L^1(\Gh\backslash\{1\}, loc)$ is $(-Q)$-homogeneous, 
that is,
$F(r\cdot\pi)=r^{-Q}F(\pi)$, 
then 
we have
$$
\forall r>0,\qquad
\int_{1\leq |\pi|\leq r}
F(\mu)
   \ d\mu(\pi)
=
|\ln r|\int_{1\leq |\pi|\leq e}
F (\pi)  \ d\mu(\pi).
$$
Furthermore
the quantity $\int_{1\leq |\pi|\leq e}
F (\pi)  \ d\mu(\pi)$ is independent  of $|\cdot|$.
This may shed some light on  our choice of definition for $\varsigma$.
\end{remark}

\subsection{Rockland operators}

Here we recall the definition of Rockland operators
and their main properties.
See \cite[Ch.4]{R+F_monograph} for proofs and references.

\begin{definition}
A \emph{Rockland operator}
\index{Rockland!operator}
 $\cR$ on $G$ is 
a left-invariant differential operator 
which is homogeneous of positive degree and satisfies the Rockland condition:
\begin{center}
(R) 
for each unitary irreducible representation $\pi$ on $G$,
except for the trivial representation, \\
the operator $\pi(\cR)$ is injective on $\cH_\pi^\infty$, 
that is,
\end{center}
$$
 \forall v\in \cH_\pi^\infty,\qquad
\pi(\cR)v = 0 \ \Longrightarrow \ v=0
.
$$
\end{definition}

\begin{ex}
In the stratified case, one can check easily that
any (left-invariant negative)  \emph{sub-Laplacian}, 
that is
\begin{equation}
\label{eq_def_lih_subLapl}
\cL=Z_1^2+\ldots+Z_{n'}^2
\quad\mbox{with} \ Z_1,\ldots,Z_{n'} \
\mbox{forming any basis of the first stratum} \ \fg_1,
\end{equation}
 is a Rockland operator.
\end{ex}

\begin{ex}
On any graded group $G$, it is not difficult to see that 
 the operator
\begin{equation}
\label{eq_cR_ex}
\sum_{j=1}^n
(-1)^{\frac{\nu_o}{\upsilon_j}}
 c_j X_j^{2\frac{\nu_o}{\upsilon_j}}
\quad\mbox{with}\quad c_j>0
,
\end{equation}
is a Rockland operator of homogeneous degree $2\nu_o$
if $\nu_o$ is any common multiple of $\upsilon_1,\ldots, \upsilon_n$.
\end{ex}

Hence Rockland operators do exist on any graded Lie group 
(not necessarily stratified).

If the Rockland operator $\cR$  is formally self-adjoint,
that is, $\cR^*=\cR$
as elements of the universal enveloping algebra $\fU(\fg)$,
then $\cR$ and $\pi(\cR)$ admit self-adjoint extensions
on $L^2(G)$ and $\cH_\pi$ respectively.
We keep the same notation for their self-adjoint extension.
We denote by    $E$ and $E_\pi$ their spectral measure:
$$
\cR = \int_\bR \lambda dE (\lambda)
\quad\mbox{and}\quad
\pi(\cR) = \int_\bR \lambda dE_\pi(\lambda).
$$
Example of formally self-adjoint Rockland operators are the 
positive Rockland operators, 
that is,  Rockland operators $\cR$ that satisfy 
$$
\forall f \in \cS(G),\qquad
\int_G \cR f(x) \ f(x) dx\geq 0.
$$
One checks easily that the operator  in \eqref{eq_cR_ex} is positive.
This shows that positive Rockland operators always exist 
on any graded Lie group.
Note that if $G$ is stratified and $\cL$ is a (left-invariant negative) sub-Laplacian, 
then it is customary to privilege   $-\cL$ as a positive Rockland operator.

 The next lemma says that the point 0 
 can be neglected in the spectrum of a positive Rockland operator
 and its group Fourier transform.
 \begin{lemma}
\label{lem_normL2_E(epsilon,infty)}
 Let $\cR$ be a positive Rockland operator with spectral measure $E$.
\begin{enumerate} 
\item Then for any $f\in L^2(G)$,
$$
\|E[0,\epsilon)f\|_2\searrow 0
\quad\mbox{and}\quad
 \|E(\epsilon,+\infty) f\|_{L^2(G)}\searrow \|f\|_{L^2(G)}
\quad\mbox{as}\quad \epsilon\searrow 0.
$$
\item If $\pi$ is a non-trivial unitary irreducible representation of $G$, 
then the spectrum of $\pi(\cR)$ is a discrete subset of $(0,\infty)$.

\item 
Let $\psi\in C^\infty(\bR)$  be a  scalar valued function 
satisfying $\psi\equiv 1$ on $(\Lambda,\infty)$ 
for some $\Lambda\in \bR$.
Then 
$\psi(t\cR)$
and $\psi(t\pi(\cR))$ converges to the identity mapping 
of $L^2(G)$ and $\cH_\pi$ for the strong operator topology (SOT).
Moreover we have
\begin{equation}
\label{eq_psi(rpiR)=I}
\forall \pi\in \Gh,\quad\exists r=r_\pi>0,\quad
\forall r>r_\pi,\qquad
\psi(r\cdot\pi(\cR))=\psi(r^\nu\pi(\cR))=\id_{\cH_\pi}.
\end{equation}

\end{enumerate}
 \end{lemma}
 \begin{proof}[Sketch of the proof]
Let us recall that the heat kernel $h_t$ of $\cR$ is by definition 
the right convolution kernel of $e^{-t\cR}$ and that it satisfies
$h_t=t^{-\frac Q\nu}h_1\circ D_{t^{-\frac 1\nu}}$ with $h_1\in \cS(G)$.
This has the two following consequences.
Firstly, it yields classically
$$
\|e^{-t\cR}f\|_2=\|f*h_t\| \Tend{t}{\infty}0,
\quad f\in L^2(G),
$$
which implies Part (1).
Secondly, it implies that the operators $\pi(h_t)$, $t>0$, 
are compact and form a continuous semi-group.
One checks easily that $\pi(\cR)$ is its infinitesimal generator,
and this yields Part (2).
Part (3) follows easily from spectral properties and Parts (1) and (2).
 \end{proof}

\begin{remark}
\label{rem_rpi}
We can give a value for $r_\pi$ in \eqref{eq_psi(rpiR)=I}:
$$
r_\pi = \frac{\lambda_{min}(\pi)}{\Lambda} \in (0,\infty),
$$
where $\lambda_{min}(\pi)$ is the minimum eigenvalue of $\pi(\cR)$, 
see Part 2 of Lemma \ref{lem_normL2_E(epsilon,infty)}.
In this case,  $r_\pi$ is $\nu$-homogeneous in $\pi$:
$$
\forall t>0, \ \pi\in \Gh,\qquad
r_{t\cdot\pi}
= \frac{\lambda_{min}(t\cdot\pi)}{\Lambda}
= \frac{t^\nu\lambda_{min}(\pi)}{\Lambda}.
$$
Hence the range of $r_\pi$ as $\pi$ runs over $\Gh$ is $(0,\infty)$.
\end{remark}

The properties of the functional calculus of $\cR$ and of the group 
Fourier transform imply the following lemma.

\begin{lemma}
\label{lem_homogeneity_multipliers}
Let $\cR$ be a positive Rockland operator of homogeneous degree $\nu$
and  $f:\bR^+\to \bC$ be a measurable function.
We assume that the domain of the operator $ f(\cR)= \int_\bR f(\lambda) dE (\lambda) $ 
contains $\cS(G)$.
Then for any $x\in G$, 
$$
 \left( f(r^\nu \cR) \phi\right)\circ D_r =f(\cR) \left(\phi \circ D_r\right),
\quad\phi\in \cS(G),
$$
where $\nu$ denotes the homogeneous degree of $\cR$,
and
\begin{equation}
\label{eq_homogeneity_kernel_multipliers}
 f(r^\nu \cR) \delta_0 (x) =r^{-Q} f(\cR)\delta_0(r^{-1}x),
 \quad x\in G, 
\end{equation}
where  $f(\cR)\delta_0$ denotes the right convolution kernel of $f(\cR)$.

Let $\pi$ be an irreducible unitary representation. 
Then the domain of the operator
$ f(\pi(\cR))= \int_\bR f(\lambda) dE_\pi (\lambda) $
contains $\cH_\pi^\infty$ and we have
\begin{equation}
\label{eq_pi_mulitpliers1}
\cF \{f(\cR)\varphi\}(\pi)=
 f(\pi(\cR)) \ \widehat \varphi(\pi), 
 \quad \phi\in \cS(G).
\end{equation}
\end{lemma}

\subsection{Sobolev spaces}
The (inhomogeneous) \emph{Sobolev spaces} $L^2_a(G)$, 
resp. the \emph{homogeneous Sobolev spaces} $\dot L^2_a(G)$,
$a\in \bR$, as 
the completion of the domain $\dom (\id+\cR)^\frac a\nu$
of $(\id+\cR)^\frac a\nu$, 
resp. the domain $\dom (\cR^\frac a\nu)$
of $\cR^\frac a\nu$, 
for the Sobolev norm 
$$
\|f\|_{L^2_a(G),\cR}:= \|(\id+\cR)^\frac a\nu f\|_{L^2(G)},
\quad\mbox{resp.}\quad
\|f\|_{\dot L^2_a(G),\cR}:= \|\cR^\frac a\nu f\|_{L^2(G)}.
$$
We realise the elements of $L^2_a(G)$ as tempered distributions and we have
$$
\cS(G)\subset \dom (\id+\cR)^\frac a\nu \subset L^2_a(G) \subset \cS'(G).
$$
We realise the elements of $\dot L^2_a(G)$ as the linear functionals $f$ on 
$\dom(\bar\cR^{-\frac a\nu})$ satisfying
$$
\exists C>0,\qquad
\forall \phi\in \dom (\bar\cR^{-\frac a\nu}),
\qquad |f(\phi)|\leq C \|\bar\cR^{-\frac a\nu}\phi\|_{L^2(G)}.
$$
Each $f\in \dot L^2_a(G)$ defines a unique function $\cR^{\frac a\nu} f\in L^2(G)$ via the continuous linear functional $\psi\mapsto f(\bar \cR^{\frac a\nu}\psi)$.

The following lemma implies that the Sobolev spaces defined above do not depend on the choice of Rockland operators:

\begin{lemma}
\label{lem_2R_fractional_power}
Let $\cR_1$ and $\cR_2$ be two positive Rockland operators
of homogeneous degree $\nu_1$ and $\nu_2$ respectively.
Then for any $a\in \bR$, 
the operators 
$(\id+\cR_1)^{\frac a {\nu_1}}(\id+\cR_2)^{-\frac a {\nu_2}}$
and
$(\cR_1)^{\frac a {\nu_1}}(\cR_2)^{-\frac a {\nu_2}}$
extends to  bounded operators on $L^2(G)$.
\end{lemma}

The Sobolev spaces defined above satisfy the following natural properties:
\begin{theorem}
\label{thm_sobolev_spaces}
\begin{enumerate}
\item 
The  spaces $L^2_a(G)$ and $\dot L^2_a(G)$ are Banach spaces.
 Different choices of positive Rockland operators yields equivalent (homogeneous) Sobolev norms.
 \item
 (Sobolev embeddings) 
 We have the continuous inclusions
  $$
 L^2_a(G)\subset C(G), \ a>Q/2,
$$
where $C(G)$ denotes the Banach space of continuous and bounded functions on $G$.
\item 
\label{item_thm_sobolev_spaces_a>0}
If $a=0$ then $L^2_0(G)=L^2(G)$.
If $a>0$ then 
$\dom (\cR^{\frac a\nu})=\dom(\id+\cR)^{\frac a\nu}\supset \cS(G)$,
and
$L^2_a(G)=L^2(G)\cap \dot L^2_a(G)$ with
$$
\|f\|_{L^2_a(G)} \asymp \|f\|_{L^2(G)}+\|f\|_{\dot L^2_a(G)},
$$
after a choice of positive Rockland operators to realise the Sobolev norms.

\item 
\label{item_thm_sobolev_spaces_continuity_Xalpha}
For any $\alpha\in \bN_0^n$, $X^\alpha$ maps continuously 
$L^2_s(G)$ to $L^2_{s-[\alpha]}(G)$
and 
$\dot L^2_s(G)$ to $\dot L^2_{s-[\alpha]}(G)$, for any $s\in \bR$.

\item Let $\cR$ be a positive Rockland operator of degree $\nu$.
Let also $a,s\in\bR$.
Then the operator $(\id+\cR)^{\frac a \nu}$ maps continuously 
$L^2_s(G)$ to $L^2_{s-a}(G)$
and 
the operator $(\cR)^{\frac a \nu}$ maps continuously 
$\dot L^2_s(G)$ to $\dot L^2_{s-a}(G)$.

\item
\label{item_thm_sobolev_spaces_duality}
For any $s\in \bR$, 
 the Banach spaces $L^2_{-s}(G)$ and  $\dot L^2_{-s}(G)$ 
 are  the duals of $L^2_s(G)$ and $\dot L^2_s(G)$ respectively 
 via the dualities 
$$
\langle f,g\rangle_{L^2_s\times L^2_{-s} }
=
\langle (\id+\cR)^{\frac s\nu}f,(\id+\cR)^{-\frac s\nu}g\rangle_{L^2\times L^2 }\ , \qquad
\langle f,g\rangle_{\dot L^2_s\times \dot L^2_{-s} }
=
\langle \cR^{\frac s\nu}f,\cR^{-\frac s\nu}g\rangle_{L^2\times L^2 }\ .
$$
\item
\label{item_thm_sobolev_spaces_interpolation}
 The Banach spaces $L^2_a(G)$ and $\dot L^2_a(G)$
satisfy the properties of interpolation 
(in the sense of 
Theorem 4.4.9 and Proposition 4.4.15 in \cite{R+F_monograph}).
\end{enumerate}
\end{theorem}

\bigskip

In order to distinguish the Sobolev spaces $L^2_s(G)$ on the graded group $G$ and the usual Sobolev spaces on the underlying $\bR^n$, 
we  denote by $H^s$ the Eulidean Sobolev spaces on $\bR^n$.
The spaces $H^s$ and $L^2_s(G)$ are not comparable globally
(we assume that $G$ is not abelian), but they are locally:

\begin{proposition}
\label{prop_sob_loc}
For any $s\in \bR$ and any $\chi\in \cD(\bR^n)$, 
the mapping $\cS(G)\ni f \mapsto \chi f$
extends (uniquely) 
to a continuous operator
of $H^s\to L^2_{s\upsilon_1}(G)$ 
and 
to a continuous operator
of $L^2_s(G)\to H^{s/\upsilon_n}$
(where $\upsilon_1\leq \ldots\leq \upsilon_n$
are the dilation's weights in increasing order).
\end{proposition}

\subsection{Bessel potential and Fourier Inversion Formula}
\label{subsec_Bessel+FIF}
Classical considerations on Bessel potentials in this context imply that 
the convolution kernel 
of $(\id+\cR)^{\frac {s_1}\nu}$ is square integrable 
when $s_1<-Q/2$,
i.e.
 $(\id+\cR)^{\frac {s_1}\nu}\delta_0\in L^2(G)$, 
 see \cite[\S 4.3.3]{R+F_monograph}.
 The Plancherel formula \eqref{eq_plancherel_formula}
 then yields
\begin{equation}
\label{eq_int(I+R)HS}
\int_{\Gh} \|\pi(\id+\cR)^{\frac {s_1}\nu}\|_{HS(\cH_\pi)}^2 d\mu(\pi)
<\infty
\quad \mbox{for}\ s_1<-Q/2,
\end{equation}
and consequently,
\begin{equation}
\label{eq_inttr(I+R)}
\int_{\Gh} \tr\left|\pi(\id+\cR)^{\frac {s_1}\nu}\right|  d\mu(\pi)
<\infty
\quad \mbox{for}\ s_1<-Q.
\end{equation}

\medskip

Naturally, the Plancherel theorem (cf. Section \ref{subsec_Gh+plancherel})
implies a Fourier Inverse Formula, at least formally. 
The next proposition gives hypotheses which imply the FIF. 

\begin{proposition}[Fourier Inversion Formula]
\label{prop_FIF}
Let $\cR$ be a positive Rockland operator of homogeneous degree $\nu$.
Let $\sigma=\{\sigma(\pi):\cH_\pi^\infty\to\cH_\pi, \pi\in\Gh\}$ 
be a field of operators on $\Gh$ defined (at least) on $\int_{\Gh}\cH_\pi^\infty d\mu(\pi)$.
Let $s>Q$.
\begin{enumerate}
\item 
The field of operators  $\{\pi(\id+\cR)^{\frac s\nu} \sigma(\pi),\pi\in \Gh\}$ is also defined on $\int_{\Gh}\cH_\pi^\infty d\mu(\pi)$. 
We assume that $\sup_{\pi\in \Gh}\|\pi(\id+\cR)^{\frac s\nu} \sigma(\pi)\|_{\sL(\cH_\pi)}$ is finite.
\item 
For each $\pi\in \Gh$, the field of operator 
$\{\pi(\id+\cR)^{\frac s\nu}:\cH_\pi^\infty\to\cH_\pi^\infty, \pi\in \Gh\}$ acts on $\int_{\Gh}\cH_\pi^\infty d\mu(\pi)$. 
The field of operators  $\{\pi(\id+\cR)^{\frac s\nu} \sigma(\pi),\pi\in \Gh\}$ is defined on $\int_{\Gh}\cH_\pi^\infty d\mu(\pi)$. 
We assume that $\sup_{\pi\in \Gh}\|\pi(\id+\cR)^{\frac s\nu} \sigma(\pi)\|_{\sL(\cH_\pi)}$ is finite.
\end{enumerate}
Then $\sigma\in L^2(\Gh)$ and $\kappa=\cF^{-1} \sigma$ is in $L^2_{s/2}$, in particular it coincides with a continuous and bounded function on $G$. 
Moreover
$$
\int_{\Gh} \tr |\sigma(\pi)| d\mu(\pi)<\infty
\qquad\mbox{and}\qquad
\kappa(0)=\int_{\Gh}\tr \sigma(\pi) \ d\mu(\pi).
$$

\end{proposition}

We will need the following classical properties for approximations of $\delta_0$:
\begin{lemma}
\label{lem_cv_psiepsilon}
Let $\psi_1\in \cS(G)$.
For $\epsilon>0$, we set
$\psi_\epsilon=(\psi_1)_{(\epsilon)}$, that is, $\psi_\epsilon(x)=\epsilon^{-Q} \psi_1(\epsilon^{-1}x)$.
We also denote $c:=\int_G \psi_1 =\int_G \psi_\epsilon$.

\begin{enumerate}
\item 
As $\epsilon\to 0$, 
we have
$\displaystyle{
\psi_\epsilon \longrightarrow   c\delta_0}$
in $\cS'(G)$,
and  if $\kappa\in \cS'(G)$ is continuous and bounded then 
$\displaystyle{
\int_G \kappa \psi_\epsilon \longrightarrow  c\kappa(0) \, .
}$
\item 
If $\pi$ is a continuous unitary representation of $G$, then 
$(\widehat \psi_\epsilon(\pi))_{\epsilon>0}$
converges to $ c \id_{\cH_\pi}$
in the strong operator topology (SOT) on $\cH_\pi$.
\end{enumerate}
\end{lemma}

\begin{proof}[Proof of Lemma \ref{lem_cv_psiepsilon}]
Part 1 is classical, see e.g. \cite[\S 3.1.10]{R+F_monograph}.
For Part 2, we write
$$
\widehat \psi_\epsilon (\pi)^* - c\id_{\cH_\pi}
=
\int_G \psi_\epsilon(x) \pi(x)  dx- c\id_{\cH_\pi}
=
\int_G \psi_1(x) \left(\pi(\epsilon x) - \id_{\cH_\pi}\right) dx.
$$
Thus, applying a vector $v\in \cH_\pi$ and for any $R>0$ decomposing the integral as $\int_G=\int_{|x|<R}+\int_{|x|\geq R}$, 
we obtain:
$$
\|(\widehat \psi_\epsilon (\pi)^* - c\id_{\cH_\pi})v\|_{\cH_\pi}
\leq 
\sup_{|x'|\leq \epsilon R} \|\pi(x')v -v\|_{\cH_\pi} 
\int_G |\psi_1|
\ + \
2\|v\|_{\cH_\pi}\int_{|x|> R} |\psi_1(x)|.
$$
And the conclusion follows easily from the continuity of $x'\mapsto \pi(x')v$ at 0 and the integrability of $\psi_1\in \cS(G)$.
\end{proof}

\begin{proof}[Proof of Proposition \ref{prop_FIF}]
Let us assume the hypothesis of Part 1.
 The membership of
 $\sigma$ in $L^2(\Gh)$ follows from
$$
\|\sigma\|_{L^2(\Gh)} 
\leq
\|\pi(\id+\cR)^{-\frac s\nu} \|_{L^2(\Gh)} 
\sup_{\pi\in \Gh}\|\pi(\id+\cR)^{\frac s\nu} \sigma(\pi)\|_{\sL(\cH_\pi)},
$$
since the first term of the right-hand side is square integrable by \eqref{eq_int(I+R)HS}.
We also have
$$
\int_{\Gh} \tr |\sigma(\pi)| d\mu(\pi)
\leq
\sup_{\pi\in \Gh}\|\pi(\id+\cR)^{\frac s\nu} \sigma(\pi)\|_{\sL(\cH_\pi)}
\int_{\Gh} \tr |\pi(\id+\cR)^{-\frac s\nu}| d\mu(\pi)
$$ 
and the last integral is finite by \eqref{eq_inttr(I+R)}.
 Hence $\int_{\Gh} \tr |\sigma(\pi)| d\mu(\pi)$ is finite.

Let $\kappa=\cF^{-1} \sigma$. 
As $\sigma\in L^2(\Gh)$,  $\kappa\in L^2(G)$.
Moreover
\begin{eqnarray*}
\|(\id+\cR)^{\frac s {2\nu}} \kappa\|_{L^2(G)}
&=&
\|\pi(\id+\cR)^{\frac s{2\nu}} \sigma\|_{L^2(\Gh)} 
\\&\leq&
\|\pi(\id+\cR)^{-\frac s{2\nu}} \|_{L^2(\Gh)} 
\sup_{\pi\in \Gh}\|\pi(\id+\cR)^{\frac s\nu} \sigma(\pi)\|_{\sL(\cH_\pi)}. 
\end{eqnarray*}
Hence $\kappa\in L^2_{s/2}$.
The Sobolev embedding, see Theorem \ref{thm_sobolev_spaces}, 
implies that
$\kappa$ is continuous and bounded on $G$.

Let $\psi_1\in \cS(G)$ with $\int_G\psi_1=1$.
We construct the $\delta_0$-approximate
$(\psi_\epsilon)_{\epsilon>0}\subset \cS(G)$ as in Lemma~\ref{lem_cv_psiepsilon}.
By the Parseval formula, see \eqref{eq_parseval_formula}, 
we have:
\begin{equation}
\label{eq_pf_prop_FIF}
\int_G \kappa(x) \bar \psi_\epsilon(x) dx
=
\int_{\Gh} \tr \left(\sigma(\pi) \widehat \psi_\epsilon (\pi)^* \right)d\mu(\pi). 
\end{equation}
By Lemma \ref{lem_cv_psiepsilon},  the left-hand side of \eqref{eq_pf_prop_FIF} tends to $\kappa(0)$ as $\epsilon\to 0$.
Note that the right-hand side of \eqref{eq_pf_prop_FIF}  is integrable since:
$$
| \tr \left(\sigma(\pi) \widehat \psi_\epsilon (\pi)^* \right)|
\leq \|\widehat \psi_\epsilon (\pi)\|_{\sL(\cH_\pi)}
\tr |\sigma|
\leq \| \psi_1\|_{L^1(G)}
\tr |\sigma|.
$$
Lemma \ref{lem_cv_psiepsilon} and  the Lebesgue Dominated Convergence Theorem imply that the right-hand side of \eqref{eq_pf_prop_FIF} converges 
to $\int_{\Gh} \tr \sigma(\pi) d\mu(\pi)$ as $\epsilon\to0$.
Taking the limit in both sides of \eqref{eq_pf_prop_FIF} as $\epsilon\to 0$ concludes the proof of Proposition \ref{prop_FIF} under the hypothesis of Part 1.

For Part 2, 
the fact that $\{\pi(\id+\cR)^{\frac s\nu}:\cH_\pi^\infty\to\cH_\pi^\infty, \pi\in \Gh\}$ acts on $\int_{\Gh}\cH_\pi^\infty d\mu(\pi)$
follows from \cite[Lemma 5.1.2]{R+F_monograph}.
We then proceed as above.
Another proof is by taking the adjoint.
\end{proof}

\begin{remark}
\label{rem_pf_FIF}
In fact, the proof above shows that if $\psi_1\in \cS(G)$ and $\psi_\epsilon=(\psi_1)_{(\epsilon)}$ as in Lemma \ref{lem_cv_psiepsilon}
and if $\sigma$ and $\kappa$ are as in Proposition \ref{prop_FIF}, 
then 
$$
\int_{\Gh}\int_{\Gh} \tr\left| \sigma (\pi) \widehat \psi_\epsilon^*(\pi)\right|d\mu(\pi)<\infty,\;\;
{\rm and}\;\;
\lim_{\epsilon\to 0} \int_{\Gh} \tr\left( \sigma (\pi) \widehat \psi_\epsilon^*(\pi)\right)d\mu(\pi)
=
c \, \kappa(0),
$$
where $c=\int_G \psi_1$.
\end{remark}

\begin{corollary}
\label{cor_prop_FIF}
Let $\sigma\in L^\infty(\Gh)$.
Then for any $\phi\in \cS(G)$, 
we have
$\displaystyle{
\int_{\Gh}\tr \left|\sigma (\pi) \cF_G(\check \phi) (\pi) \right|
d\mu(\pi)<\infty
}$
and denoting by $\kappa\in \cS'(G)$ the kernel of $\sigma$, 
i.e. $\sigma=\widehat \kappa$, 
we have:
\begin{equation}
\label{eq_cor_prop_FIF}
\langle \kappa,\phi\rangle
=
\int_{\Gh}\tr \left(\sigma (\pi) \cF_G(\check \phi) (\pi) \right)
d\mu(\pi),
\end{equation}
where $\check \phi (x)=\phi(x^{-1})$.
\end{corollary}

\begin{proof}
If $\phi\in \cS(G)$, 
then  $\widehat \phi$ satisfies the hypotheses of Proposition \ref{prop_FIF} since $(\id+\cR)^N\phi \in \cS(G)$ is integrable for any $N\in \bN$.
For the same reason $\sigma\cF_G (\check\phi)$ satisfies the hypotheses of Proposition~\ref{prop_FIF}.
We conclude with $\langle \kappa,\phi\rangle = \check \phi *\kappa(0)$
and $\cF_G(\check \phi *\kappa)=\sigma\cF_G (\check\phi)$.
\end{proof}

\begin{remark}
\label{rem_cor_prop_FIF}
Corollary \ref{cor_prop_FIF} implies 
that if $\sigma\in L^\infty(\Gh)$
and $\kappa\in \cS'(G)$ are such that 
\eqref{eq_cor_prop_FIF} holds for any $\phi\in \cS(G)$
or $\cD(G)$ then $\kappa$ is the kernel of $\sigma$, 
i.e. $\sigma=\widehat \kappa$.
\end{remark}

We will also need the following inversion formula:
\begin{proposition}
\label{prop_FIF2}
Let $\kappa$ be a compactly supported distribution on $G$.
Then for each unitary representation $\pi$ of $G$ and $v,w\in \cH_\pi$,
 we can define 
$$
(\widehat \kappa(\pi)v,w)_{\cH_\pi} =\int_G \kappa(x) (\pi(x)^*v,w)_{\cH_\pi}  dx,
$$
since $x\mapsto (\pi(x)^*v,w)_{\cH_\pi} $ is smooth and bounded on $G$.

For any smooth and bounded function $\phi$ on $G$, we have
$$
\int_{\Gh} \tr \left(\widehat \kappa (\pi)\ \widehat \phi(\pi)\right) d\mu(\pi)
=\langle \kappa,\phi\rangle,
$$
 interpreting the left hand side as the limits (in this order) of the absolutely convergent double integral: 
$$
\lim_{R\to \infty} \lim_{N\to +\infty}\int_{N\cdot \cC} \int_G \trN \left(\widehat \kappa (\pi)\pi(x)\right) \phi(x) \chi_R(x) \ dx  d\mu(\pi),
$$
where $\chi\in \cD(G)$ with $\chi\equiv 1$ on a neighbourhood of 0
and $\chi_R(x):=\chi (R^{-1}x)$, $\cC$ a compact neighbourhood of $1\in \Gh$ such that $\cup_{N\in \bN} N\cdot \cC=\Gh$, 
and $\trN$ denotes the trace of the operators projected on the subspace spanned by the first $N$ vectors, having fixed a fundamental sequence of vector fields.
\end{proposition}

For instance, having fixed a quasinorm, we can choose $\cC :=\{|\pi|\leq 1\}$, see Section \ref{subsec_polar_dec_Gh}.
The definition of a fundamental sequence of vector fields may be found in 
\cite[A93]{Dixmier_C*}.

\begin{proof}
Corollary \ref{cor_prop_FIF} implies the result when $\kappa\in \cD(G)$.

Let $\kappa$ be a compactly supported distribution.
We consider $\psi_1\in \cD(G)$ satisfying $\psi(0)=1$, 
and $\psi_\epsilon(x):=\epsilon^{-Q} \psi(\epsilon^{-1} x)$.
Then $\kappa_\epsilon=\kappa*\psi_\epsilon$ is in $\cD(G)$
and we conclude the proof by passing carefully to the limit 
using Lemma \ref{lem_cv_psiepsilon}.
\end{proof}

\subsection{Operators of type $\nu$}
\label{subsec_optypenu}

The properties of kernels or operators  of type $\nu$
extends from the Euclidean setting to the case of homogeneous Lie groups, 
so in particular to graded Lie groups
(see e.g.  \cite[Chapter 6 A]{folland+stein_82} or
\cite[\S 3.2]{R+F_monograph}):
\begin{definition}
\label{def_kernel_type_nu}
A distribution $\kappa\in \cD'(G)$ 
which is smooth away from the origin and  homogeneous of degree $\nu-Q$
is called a \emph{kernel of type} $\nu\in \bC$ on $G$.
The corresponding convolution operator
$f\in \cD(G)\mapsto f*\kappa$ is called an \emph{operator of type} $\nu$.
\end{definition}

\begin{ex}
\label{ex_powercR_optypenu}
Let $\cR$ be a positive Rockland operator of homogeneous degree $\nu$.
For any $a\in \bC$, $\Re a\in [0,Q)$, 
the operator $\cR^{-\frac a\nu}$ is of type $a$.
See \cite[\S 4.3]{R+F_monograph}. 
\end{ex}

The next statement summarise the properties of the operators of type $\nu$
used in the paper:
\begin{proposition}
\label{prop_op_type}
Let $G$ be a graded group.
\begin{enumerate}
\item 
\label{item_prop_op_type_Lp_bdd}
An operator of type $\nu$ with $\nu\in [0,Q)$
is $(-\nu)$-homogeneous  
and
extends to a bounded operator from $L^p(G)$ to $L^q(G)$
whenever $p,q\in (1,\infty)$ satisfy
$\frac 1p -\frac 1q = \frac {\Re \nu}Q$.
\item 
\label{item_prop_op_type_-Qhomo_distrib}
Let $\kappa$ be a smooth function away from the origin homogeneous of degree $\nu$ with $\Re \nu=-Q$.
Then $\kappa$ is a kernel of type $\nu$,
if and only if
 its mean value is zero, 
that is, 
when 
$\int_{\fS} \kappa \, d\sigma=0$ where $\sigma$ is the measure on the unit sphere of a homogeneous quasi-norm given 
by the polar change of coordinates, see Proposition \ref{prop_polar_coord}.
(This condition is independent of the choice of a homogeneous quasi-norm.)

\item 
\label{item_prop_op_type_diff}
Let $\kappa$ be a kernel of type $s\in [0,Q)$. 
Let $T$ be a homogeneous left differential operator of degree $\nu_T$.
If $s-\nu_T\in [0,Q)$, then $T\kappa$ is a kernel of type $s-\nu_T$.

\item 
\label{item_prop_op_type_conv}
Suppose $\kappa_{1}$ is a kernel of type ${\nu_1}\in \bC$
with $\Re {\nu_1}>0$ and $\kappa_{2}$ is a kernel of type ${\nu_2}\in \bC$ 
with $\Re{\nu_2}\geq 0$.
We assume $\Re({\nu_1}+{\nu_2})<Q$.
Then $\kappa_{1}*\kappa_{2}$ is well defined as a kernel of type 
${\nu_1}+{\nu_2}$.
Moreover if $f\in L^p(G)$ where $1<p<Q/(\Re({\nu_1}+{\nu_2}))$
then $(f*\kappa_{1})*\kappa_{2}$ and $f*(\kappa_{1}*\kappa_{2})$
belong to $ L^q(G)$, $\frac 1q=\frac 1p -\frac{\Re({\nu_1}+{\nu_2})}Q$,
 and they are equal.
\end{enumerate}
\end{proposition}

The $L^2$-boundedness of operators of type 0 
(see Part \eqref{item_prop_op_type_Lp_bdd} in the case $\nu=0$)
and the characterisation of  Part \eqref{item_prop_op_type_-Qhomo_distrib}
are proved 
using the classical construction of a principal value distribution
and quasi-orthogonality.
The next lemma summarises the result in more detail
with the vocabulary of this paper:

\begin{lemma}
\label{lem_pv}
\begin{enumerate}
\item Let $\kappa\in C^1(G\backslash \{0\})$ be $(-Q)$-homogeneous 
and with vanishing mean value. 
Then $\kappa$ extended to a distribution on $G$ which is the kernel of an
convolution operator bounded on $L^2(G)$.
We fix a homogeneous quasi-norm $|\cdot|$.
For each $j\in \bZ$, we define the integrable function $\kappa_j$ 
via
$\kappa_j(x):=\kappa(x) 1_{2^j\leq |x| \leq 2^{j+1}}$, $x\in G$.
Then for each $\pi\in \Gh$, and each $v\in \cH_\pi$, 
the limit  $\sum_{j=-M_1}^{M_2} \widehat \kappa_j(\pi) v$ 
converges in $\cH_\pi$ as $M_1,M_2\to \infty$.
This defines a field of operators  $\sum_{j\in \bZ} \widehat \kappa_j$ 
wich is 0-homogeneous and satisfies:
 \begin{equation}
\label{eq_cv1_lem_pv}
 \sup_{\pi\in \Gh} \|  \sum_{j\in \bZ} \widehat \kappa_j(\pi)\|_{\sL(\cH_\pi)} 
 \leq 
 C \sup_{|z|=1,|\alpha|\leq 1} |X^\alpha \kappa(z)|,
\end{equation}
where $C$ is a constant which depends on the structural constants 
 of the group $G$ 
and of the choice of a homogeneous  quasi-norm $|\cdot|$, but not on $\kappa$.
\item 
 Let $\sigma=\{\sigma(\pi)\in \sL(\cH_\pi), \pi\in \Gh\}$ be a measurable field of operators  such that 
\begin{itemize}
\item $\sigma$ is 0-homogeneous, 
i.e. $\sigma(r\pi)=\sigma(\pi)$ for (almost) all $\pi\in \Gh$ and all $r>0$,
\item $\sigma$ is bounded, 
i.e. $\sup_{\pi\in \Gh} \|\sigma(\pi)\|_{\sL(\cH_\pi)}<\infty$,
\item the kernel associated with $\sigma$, i.e. $\kappa\in \cS'(G)$ such that $\widehat \kappa=\sigma$, coincides with a $C^1$ function on  $G\backslash\{0\}$.
\end{itemize}
 Then the mean value of $\kappa$ vanishes. Using the notation of Part 1, 
the sum $\sum_j \kappa_j$ converges in $\cS'(G)$
and defines a tempered distribution
which coincides with $\kappa$ on $G\backslash\{0\}$.
We have
$$
\kappa=\sum_j \kappa_j  + c_\sigma \delta_0, 
$$
where
$\displaystyle{
c_\sigma =\int_G \kappa(z) \chi (z) dz
}$
where $\chi\in \cD(\bR)$ is such that $\chi(0)=1$
and $\chi(z)=\chi_1(|z|)$ for some $\chi_1\in \cC(\bR)$.
The constant $c_\sigma$ does not depend on  $\chi$ or $|\cdot|$.

As a representative of the measurable field $\sigma$, 
we may choose the one given via:
$$
\sigma(\pi)u = \sum_{j=-\infty}^{+\infty} \widehat\kappa_j(\pi)u 
+c_\sigma u,
\qquad \pi\in \Gh, u\in \cH_\pi.
$$
In this case, 
$
c_\sigma=\sigma(1).
$
\end{enumerate}
\end{lemma}

\begin{proof}[Sketch of the proof of Lemma \ref{lem_pv}]
See e.g. \cite[\S 3.2.5]{R+F_monograph} for the proof of Part (1).
For Part (2), 
let $\tilde \kappa$ be the kernel associated with 
the symbol $\sum_{j\in \bZ} \widehat\kappa_j$.
Then  $\tilde \kappa=\sum_j \kappa_j$ is a $(-Q)$-homogeneous tempered distribution.
For any $\phi \in \cD(G)$, 
the sum $\sum_j \langle \kappa_j, \phi\rangle$ is absolutely convergent and its sum is $\langle \tilde \kappa ,\phi\rangle$.
Hence $\tilde \kappa$ coincides with $\kappa$ on $G\backslash\{0\}$
so the distribution  $\kappa -\tilde \kappa$
being $-Q$-homogeneous and supported at the origin
must be a multiple of $ \delta_0$.
If there exists $\phi_1\in \cC(\bR)$ such that $\phi(z)=\phi_1(|z|)$
then 
$$
\langle \kappa_j, \phi\rangle=\int_{2^j \leq |z|\leq 2^{j+1}} \kappa(z) \phi_1(|z|)dz=0
$$
as the mean value of $\kappa$ is zero and $\langle \tilde \kappa ,\phi\rangle=0$.
This together with $\kappa=\tilde \kappa +c_\sigma \delta_0$ with $c_\sigma\in \bC$ implies the rest of the statement.
\end{proof}


\section{Pseudo-differential calculus}\label{sec_pseudo}

Here we outline the pseudo-differentical calculus developed in \cite{R+F_monograph}.

\subsection{Quantisation}
\label{subsec_quantisation}
A \emph{symbol} is a measurable field of operators $\sigma(x,\pi):\cH_\pi^\infty \to \cH_\pi^\infty$, parametrised by $x\in G$ and $\pi\in \Gh$.
We formally associate to $\sigma$ the operator $\Op(\sigma)$
as follows
$$
\Op(\sigma) f (x) := \int_G 
\tr \left(\pi(x) \sigma(x,\pi) \widehat f (\pi) \right)
d\mu(\pi),
$$
where $f\in \cS(G)$ and $x\in G$.

Regarding symbols, when no confusion is possible,
we will allow ourselves some notational shortcuts, 
for instance writing $\sigma(x,\pi)$ 
when considering the field of operators $\{\sigma(x,\pi) :\cH_\pi^\infty \to \cH_\pi^\infty, (x,\pi)\in G\times\Gh\}$ with the usual identifications
for $\pi\in \Gh$ and $\mu$-measurability.

This quantisation has already been observed in \cite{taylor_84,BFG,R+F_monograph} for instance.
It can be viewed as an analogue of the Kohn-Nirenberg quantisation
since the inverse formula can be written as 
$$
 f (x) := \int_G 
\tr \left(\pi(x) \widehat f (\pi) \right)
d\mu(\pi),
\quad f\in \cS(G), \ x\in G.
$$
This also shows that the operator 
associated with the symbol $\id=\{\id_{\cH_\pi} , (x,\pi)\in G\times\Gh\} $
is the identity operator $\Op(\id)=\id$.

Note that (formally or whenever it makes sense),
if we denote the (right convolution) kernel of $\Op(\sigma)$ by $\kappa_x$,
that is, 
$$
\Op(\sigma)\phi(x)=\phi*\kappa_x,
\quad x\in G, \ \phi\in \cS(G),
$$
then it is given by
$$
\pi(\kappa_x)=\sigma(x,\pi).
$$
Moreover  the integral kernel of $\Op(\sigma)$ is 
$$
K(x,y)=\kappa_x(y^{-1}x),\quad\mbox{where}\quad
\Op(\sigma)\phi(x)=\int_G K(x,y) \phi(y)dy.
$$
We shall abuse the vocabulary and call $\kappa_x$ 
the kernel of $\sigma$, and $K$ its integral kernel.

\subsection{Difference operators}
The difference operators are aimed at replacing the derivatives with respect to the Fourier variable in the Euclidean case.
For each $\alpha\in \bN_0^n$, 
the difference operator $\Delta^\alpha$ is defined via
$$
\Delta^\alpha \widehat f (\pi) = \cF_G(x^\alpha f) (\pi),
\quad \pi\in \Gh.
$$
Here $f$ is in a distributional space on which the group Fourier transform has been defined, i.e. $L^1(G)$, $L^2(G)$ or $\cK_{a,b}(G)$ etc...

The difference operators satisfy the Leibniz rule:
\begin{equation}
\label{eq_Leibniz_rule_sigma}
\Delta^\alpha (\sigma_1\sigma_2)=
\sum_{[\alpha_1]+[\alpha_2]= [\alpha]} 
c_{\alpha_1,\alpha_2}\Delta^{\alpha_1}\sigma_1 \ 
\Delta^{\alpha_2}\sigma_2,
\end{equation}
where $c_{\alpha_1,\alpha_2}$ are universal constants.
By `universal constants', 
we mean that they depend only on~$G$ and the choice of the basis $\{X_j\}_{j=1}^n$.
This comes from the  fact  that for any $\alpha\in \bN_0^n$, 
with the same constants $c_{\alpha_1,\alpha_2}$ as above,
we have
\begin{equation}
\label{eq_xy_alpha}
(xy)^\alpha =
\sum_{[\alpha_1]+[\alpha_2]= [\alpha]} 
c_{\alpha_1,\alpha_2} x^{\alpha_1}y^{\alpha_2}.
\end{equation}
Note that $c_{0,\alpha'_2}=\delta_{\alpha=\alpha'_2}$
and $c_{\alpha'_1,0}=\delta_{\alpha'_1=\alpha}$.

From \eqref{eq_xy_alpha},
 one also deduces that if $\phi\in \cS(G)$
 and $\kappa\in \cS'(G)$ 
$$
(x^\alpha \phi)  *  \kappa
=
\sum_{[\alpha_1]+[\alpha_2]= [\alpha]} 
c_{\alpha_1,\alpha_2}  (-1) ^{\alpha_2}x^{\alpha_1}
\phi* (x^{\alpha_2}\kappa).
$$
Taking the Fourier transform, the following property holds for any $\phi\in \cS(G)$ and $\sigma\in L^\infty(\Gh)$ satisfying $\Delta^{\alpha'}\sigma\in L^\infty(\Gh)$ for any $\alpha'\in \bN_0^n$ with $[\alpha']\leq[\alpha]$:
\begin{equation}
\label{eq_sigmaDeltaphi}
\sigma 
\Delta^\alpha\phi
= 
\sum_{[\alpha_1]+[\alpha_2]= [\alpha]} 
c_{\alpha_1,\alpha_2}  (-1) ^{\alpha_2}\Delta^{\alpha_1}
\{\Delta^{\alpha_2} \sigma \ \widehat \phi\}.
\end{equation}

\begin{ex}
\label{ex_Delta_piX}
One can prove easily
$$
\Delta^\alpha \pi(X)^\beta = 
\left\{\begin{array}{ll}
0
&\mbox{if}\ [\alpha]>[\beta],\\
\sum_{[\alpha']= [\alpha]-[\beta]} 
c'_{\alpha',\alpha,\beta} \pi(X)^{\alpha'}
&\mbox{if}\ [\alpha]\leq[\beta],
\end{array}\right.
$$
where $c'_{\alpha',\alpha,\beta} $ are universal constants.
\end{ex}

Using \eqref{eq_rcdotpif},
if $f$ and $x^\alpha f$ are integrable, 
then 
$$
(r\cdot \pi)(x^\alpha f) 
=
\pi(x^\alpha f)_{(r)}
=
r^{-[\alpha]} 
\pi(x^\alpha f_{(r)})
=
r^{-[\alpha]} 
\Delta^\alpha \widehat f_{(r)}.
$$
Hence denoting $\sigma=\widehat f$ 
and $\sigma_{r\,\cdot}=\{\sigma_{r\cdot\pi},\pi\in \Gh\}$, 
we have
$\sigma_{r\,\cdot}=\widehat f_{(r)} $ by \eqref{eq_rcdotpif} and
we have obtained:
\begin{equation}
\label{eq_Deltarcdotpi}
\Delta^\alpha\left( \sigma_{r\,\cdot}\right) (\pi) = 
r^{[\alpha]}
\left(\Delta^\alpha \sigma\right) (r\cdot \pi),
\quad r>0,\ \pi\in \Gh.
\end{equation}
One checks easily that \eqref{eq_Deltarcdotpi}
holds as long as it makes sense.

We also have the following integration by parts:
$$
\int_{\Gh} \tr \left(\Delta^\alpha \sigma_1  \  \ \sigma_2 \right)d\mu 
=
(-1)^{|\alpha|} 
\int_{\Gh} \tr \left(\sigma_1  \   \Delta^\alpha \sigma_2 \right)d\mu, 
$$
if $\sigma_1,\sigma_2\in \cF_G \cS(G)$ and $\alpha\in \bN_0^n$.
Indeed in this case, using the FIF, see Proposition \ref{prop_FIF}, 
both sides are equal to $\int_G \cF^{-1}_G\sigma_1(x) x^\alpha \cF^{-1}_G\sigma_2(x^{-1})dx$.
Along the same idea, we have:
\begin{lemma}
\label{lem_FIF+IBP}
Let $\sigma$ be a symbol such that at least one of the two quantities
$$
\sup_{\pi\in \Gh}\|\pi(\id+\cR)^{\frac s\nu} \sigma(\pi)\|_{\sL(\cH_\pi)}
\quad,\quad
\sup_{\pi\in \Gh}\|\sigma(\pi)\pi(\id+\cR)^{\frac s\nu} \|_{\sL(\cH_\pi)},
$$
 is finite for some $s>Q$.
 Then for any $\alpha\in \bN_0^n\backslash\{0\}$, we have:
$$
\int_G \tr\left( \Delta^\alpha \sigma (\pi)\right) d\mu(\pi) =0,
$$
in the sense that if $(\psi_\epsilon)_{\epsilon>0}$ is any $\delta_0$-approximate as in Lemma \ref{lem_cv_psiepsilon}, 
then each quantity $$\int_G \tr\left| \sigma (\pi) \ \Delta^\alpha \widehat \psi_\epsilon(\pi)\right| d\mu(\pi), \;\;\epsilon>0,$$ is finite and
the following limit exists and is zero:
$$
\lim_{\epsilon\to0}
\int_G \tr\left(  \sigma (\pi) \ \Delta^\alpha \widehat \psi_\epsilon(\pi)\right) d\mu(\pi) =0.
$$
\end{lemma}

By Remark \ref{rem_pf_FIF}, in the case $\alpha=0$, 
the limit above is $c \kappa(0)$ where $c=\int_G \psi_1$ and $\kappa$ is the kernel of $\sigma$.

\begin{proof}[Proof of Lemma \ref{lem_FIF+IBP}]
We set $\phi_1 (x) = x^\alpha\psi_1(x)$
and 
$\phi_\epsilon (x) = \epsilon^{-Q} \phi_1(\epsilon^{-1}x)$.
The statement follows from 
$\widehat \phi_\epsilon= 
\epsilon^{-[\alpha]}\Delta^\alpha \widehat \psi_\epsilon $ 
by \eqref{eq_Deltarcdotpi}
and by Remark \ref{rem_pf_FIF}, 
$$
\int_{\Gh} \tr\left(  \sigma (\pi) \  \widehat \phi_\epsilon(\pi)\right) d\mu(\pi) 
\longrightarrow_{\epsilon\to0}
c\,\cF^{-1} \sigma (0), 
\quad\mbox{with}\quad c=\int_G \phi_1.
$$
\end{proof}

\subsection{The symbol classes $S^m(G)$ and the calculus}

In this section, we recall the definition and properties of the symbolic pseudo-differential calculus defined on graded Lie groups in 
\cite[\S 5]{R+F_monograph}.

\begin{definition}
A symbol 
$\displaystyle{
\sigma=\{\sigma(x,\pi):\cH_\pi^\infty\to\cH_\pi^\infty, (x,\pi)\in G\times\Gh\}
}$
 is called a {\em symbol of order $m$} 
whenever,
for each $\alpha,\beta\in \bN_0^n$ and $\gamma\in \bR$
we have
 \begin{equation}
\label{eq_def_Smrhodelta}
\sup_{x\in G, \pi\in \Gh}
\|\pi(\id+\cR)^{\frac{ [\alpha]-m +\gamma}\nu }
X_x^\beta\Delta^\alpha \sigma(x,\pi) 
\pi(\id+\cR)^{-\frac{\gamma}\nu }\|_{\sL(\cH_\pi)}<\infty,
\end{equation}
where  $\cR$ is  a (fixed) positive Rockland operator of homogeneous degree $\nu$.
The \emph{symbol class}  $S^m=S^m(G)$ is the set of symbols 
of order $m$.
\end{definition}

By Lemma \ref{lem_2R_fractional_power},
each symbol class $S^m$ is independent of $\cR$.
The property of interpolation of Sobolev spaces 
(cf. Theorem \ref{thm_sobolev_spaces} \eqref{item_thm_sobolev_spaces_interpolation})
also implies that it suffices to have \eqref{eq_def_Smrhodelta}
only for a sequence $(\gamma_\ell)_{\ell\in \bZ}$
with $\gamma_\ell\rightarrow_{\ell\to\pm\infty}\pm \infty$.

For a chosen positive Rockland operator $\cR$ of homogeneous degree $\nu$,
the seminorms
$$
\|\sigma\|_{S^m,a,b,c}:=
\max_{\substack {[\alpha]\leq a \\
[\beta]\leq b, |\gamma|\leq c}}
\sup_{\substack {x\in G\\ \pi\in \Gh}}
\|\pi(\id+\cR)^{\frac{ [\alpha]-m +\gamma}\nu }
X_x^\beta\Delta^\alpha \sigma(x,\pi) 
\pi(\id+\cR)^{-\frac{\gamma}\nu }\|_{\sL(\cH_\pi)}
,
$$
 yield a structure of Frechet spaces on each $S^m$, $m\in \bR$.
One checks that $S^{m_1} \subset S^{m_2}$ if $m_1\leq m_2$.
We also define the space of smoothing symbols 
$$
S^{-\infty}:=\bigcap_{m\in \bR} S^m,
$$
which is endowed with the topology of projective limit.

The corresponding spaces of operators
$$
\Psi^m\equiv \Psi^m(G):=\Op(S^m),\quad m\in \bR\cup\{-\infty\},
$$
yield a  calculus:

\begin{theorem}
\label{thm_calculus}
\begin{enumerate}
\item 
The set of operators $\cup_{m\in \bR} S^m$ is an algebra of symbols
in the sense that product, taking the adjoint and applying spacial and dual derivatives
$$
\left\{\begin{array}{rcl}
S^{m_1}\times S^{m_2} &\longrightarrow& S^{m_1+m_2}\\
(\sigma_1,\sigma_2)&\longmapsto& \sigma_1\sigma_2
\end{array}
\right.,\quad
\left\{\begin{array}{rcl}
S^{m} &\longrightarrow& S^{m}\\
\sigma&\longmapsto& \sigma^*
\end{array}\right.,
\quad\mbox{and}\quad
\left\{\begin{array}{rcl}
S^{m} &\longrightarrow& S^{m-[\alpha]}\\
\sigma&\longmapsto&X^\beta \Delta^\alpha \sigma
\end{array}\right.,
$$
(for any $m,m_1,m_2\in \bR$ and $\alpha,\beta\in \bN_0^n$)
 are continuous operations.
\item 
Furthermore $\sigma\in S^m$ if and only if 
it satisfies \eqref{eq_def_Smrhodelta} with $\gamma=0$.
The seminorms $\|\cdot\|_{S^m,a,b,0}$ yield an equivalent 
family of seminorms for the Fr\'echet topology of $S^m$. 
\item 
The set of operators $\cup_{m\in \bR} \Psi^m$ is a calculus in the sense that  product, taking the adjoint and applying spacial and dual derivatives
$$
\left\{\begin{array}{rcl}
\Psi^{m_1}\times \Psi^{m_2} &\longrightarrow& \Psi^{m_1+m_2}\\
(T_1,T_2)&\longmapsto& T_1T_2
\end{array}
\right.,
\quad\mbox{and}\quad
\left\{\begin{array}{rcl}
\Psi^{m} &\longrightarrow& \Psi^{m}\\
T&\longmapsto& T^*
\end{array}\right.,
$$
(for any $m,m_1,m_2\in \bR$)
 are continuous operations, 
 and that any operator in $\Psi^{m}$
maps continuously $L^2_s(G)$ to $L^2_{s-m}(G)$.
\item 
We have the asymptotic expansions:
$$
\displaylines{
\Op^{-1}\left( \Op(\sigma_1)\Op(\sigma_2)\right)  \sim \sum_{[\alpha]} c_\alpha \Delta^\alpha \sigma_1 X_x^\alpha \sigma_2,\cr
\Op^{-1}\left(\Op(\sigma)^{*}\right) 
\sim \sum_{[\alpha]} c'_\alpha X_x^\alpha \Delta^\alpha  \sigma^*.\cr}
$$
where the constants $c_\alpha$ and $c'_\alpha$, $\alpha\in \bN_0^n$, 
are universal with  $c_0=c'_0=1$.
\end{enumerate}
 \end{theorem}

In the statement above we use asymptotic expansions of the form 
\begin{equation}
\label{eq_def_asymptotic}
\sigma \sim \sum_{\ell=0}^\infty \sigma_\ell,
\quad \sigma_\ell \in S^{m_\ell}, 
\quad\mbox{with}\ m_\ell \ \mbox{strictly decreasing to} \ -\infty.
\end{equation}
This means that for any $M\in \bN$,
$$
\sigma- \sum_{\ell \leq M} \sigma_\ell \in S^{m_M+1}.
$$
More precisely in Theorem \ref{thm_calculus}, 
this was used with
$$
\sigma_\ell=\sum_{[\alpha]=w_\ell}  
c_\alpha \Delta^\alpha \sigma_1 X_x^\alpha \sigma_2
\in S^{m_1+m_2 -w_\ell}
\quad\mbox{and}\quad
\sigma_\ell=\sum_{[\alpha]=w_\ell}  
c'_\alpha X_x^\alpha \Delta^\alpha  \sigma^*
\in S^{m-w_\ell}.
$$
Note that any formal asymptotic yields a symbol modulo a smoothing operator:

\begin{theorem}
\label{thm_asymptotic}
Let $\{\sigma_\ell\}_{\ell\in \bN_0}$ be a sequence of symbols such that
 $\sigma_\ell\in S^{m_\ell}$ with $m_\ell$ strictly decreasing to $-\infty$.
 Then there exists $\sigma\in S^{m_0}$, 
 unique modulo $S^{-\infty}$,
 such that $\sigma\sim \sum_\ell \sigma_\ell$.
 \end{theorem}

\medskip

Naturally the calculus $\cup_{m\in \bR} \Psi^m$ contains the left-invariant calculus since we have:

\begin{ex}\begin{enumerate}
\item For any $\alpha\in \bN_0^n$, 
$X^\alpha= \Op(\pi(X^\alpha)) \in \Psi^{[\alpha]}$.
\item 
The set $\Psi^0$ contains any smooth function $f:G\to \bC$ with bounded left-derivatives, 
that is, 
\begin{equation}
\label{eq_bdd_left_derivatives}
\forall\beta\in \bN_0^n,\qquad
\sup_{x\in G} |X^\beta f(x)|<\infty.
\end{equation}
\end{enumerate}
\end{ex}

Another important class of symbols in the calculus is given by multipliers in Rockland operators. The precise class of multipliers that we consider is the following. Let $\cM_{m}$ be the space of functions $f\in C^\infty(\bR^+)$
such that the following quantities for all $\ell\in \bN_0$ are finite:
$$
\|f\|_{\cM_{m,\ell}} := \sup_{\lambda>0,\, \ell'=0,\ldots,\ell}
(1+\lambda)^{-m +\ell'} |\partial_\lambda^{\ell'} f (\lambda)|
.
$$
In other words, the class of functions $f$ that appears in the definition above
are the functions which are smooth on $\bR^+$ and have the symbolic 
behaviour at infinity of the H\''ormander class $S^m_{1,0}(\bR)$ on the real line.
For instance, for any $m\in \bR$,
the function $\lambda\mapsto (1+\lambda)^m$ 
is in $\cM_{m}$.

\begin{proposition}
\label{prop_multipliers_symbol}
Let $m\in \bR$ and let $\cR$ be a positive Rockland operator of homogeneous degree $\nu$.
If $f\in \cM_{\frac m \nu}$,
then $f(\cR)$ is in $\Psi^m$ and its symbol $\{f(\pi (\cR)), \pi\in \widehat G\}$ satisfies
$$
\forall a,b,c \in \bN_0, \qquad
\exists \ell\in \bN, \qquad \exists C>0, \qquad
\|f(\pi (\cR))\|_{a,b,c} \leq C \| f\|_{\cM_{\frac m \nu} , \ell}
,
$$
with $\ell$, $a,b,c\in \bN_0$ and $C$ independent of $f$.
\end{proposition}

\begin{corollary}
\label{cor_prop_multipliers_symbol}
If $\chi\in C^\infty(\bR)$ is such that $(\supp \chi) \cap [0,+\infty)$
is compact, then the symbol
 $\{\chi(\pi(\cR)),\pi\in \Gh\}$ is in $S^{-\infty}$.
Moreover its kernel $\chi(\cR)\delta_0$ is Schwartz, 
$\{\chi(\pi(\cR)),\pi\in \Gh\} \in L^2(\Gh)$ and 
$$
\int_{\Gh} \tr |\chi(\pi(\cR))|d\mu(\pi)<\infty.
$$
 \end{corollary}

Note that the membership of $\chi(\cR)\delta_0$ in $\cS(G)$
was already proved in 
\cite{hulanicki},
and is used to show Proposition \ref{prop_multipliers_symbol}.

\medskip

We conclude this section with recalling the following properties 
of the kernel associated with a symbol in the calculus:
\begin{proposition}
\label{prop_Sm_kernel}
Let $\sigma\in S^m$ and let $\kappa$ its associated kernel. 
This means that for each $x\in G$, $\kappa_x\in \cS'(G)$
and $\sigma(x,\pi)=\widehat \kappa_x(\pi)$.
Furthermore $x\mapsto \kappa_x\in \cS'(G)$ is smooth on $G$,
and for each $x\in G$, $\kappa_x \in C^\infty (G\backslash\{0\})$.
Furthermore, for any neighbourhood $V$ of 0, any $N\in \bN_0$ 
and any $\alpha,\beta\in \bN_0^n$, 
there exist a constant $C>0$ and a seminorm 
 $\|\cdot\|_{S^m,a,b,c}$ such that
$$
\forall x,z\in G,\quad z\notin V,\qquad
  |X^\alpha_z X^\beta_x \kappa_x(z)| \leq C_{s} \|\sigma\|_{ S^m,a,b,c} |z|^{-N}.
$$
The constant $C$ and the seminorm $\|\cdot\|_{S^m,a,b,c}$
may depend on $V, N, \alpha,\beta$ but are independent of $\sigma$.
\end{proposition}


\section{Homogeneous and principal symbols, classical calculus}\label{sec_hom}

In this section, we define the notions of homogeneous symbols, classical symbol classes and principal symbol 
in a way analogous to the Euclidean case.

\subsection{Homogeneous symbol classes $\dot S^m$}
\label{subsec_dotSm}

\begin{definition}
\label{def_homogeneous_symbol}
A symbol $\sigma=\{\sigma(x,\pi):\cH_\pi^\infty \to \cH_\pi^\infty , (x,\pi)\in G\times \Gh\}$ is said to be \emph{homogeneous} 
of degree $m \in \bR$ or $m$-homogeneous 
when
$$
\sigma(x,r\cdot \pi) = r^m \sigma(x,\pi),
$$
for all $x\in G$ and $\mu$-almost all $\pi\in \Gh$ and $dr$-almost all $r>0$.

A $m$-homogeneous symbol $\sigma=\{\sigma(x,\pi)\}$ is \emph{regular} 
if  for any $\alpha,\beta\in \bN_0^n$, $\gamma\in \bR$:
\begin{equation}
\label{eq_def_dotSmrhodelta}
\sup_{\substack{\pi\in \Gh\\ x\in G}}
\| \pi(\cR)^{\frac{[\alpha]-m+\gamma}{\nu}}
X_x^\beta \Delta^\alpha \sigma(x,\pi) \pi(\cR)^{-\frac{\gamma}{\nu}}\|_{{\mathcal L}(\cH_\pi)} <\infty
\end{equation}
where $\cR$ is a positive Rockland operator of degree $\nu$

We denote by $\dot S^m $ the space of regular $m$-homogeneous symbols.
\end{definition}

\begin{remark}
\label{rem_def_dotSmrhodelta}
\begin{enumerate}
\item
As in the inhomogeneous case,
Lemma \ref{lem_2R_fractional_power} implies that
each symbol class $\dot S^m$ is independent of $\cR$.
\item
The property of interpolation of Sobolev spaces 
(cf. Theorem \ref{thm_sobolev_spaces} \eqref{item_thm_sobolev_spaces_interpolation})
also implies that it suffices to have \eqref{eq_def_dotSmrhodelta}
only for a sequence $(\gamma_\ell)_{\ell\in \bZ}$
with $\gamma_\ell\rightarrow_{\ell\to\pm\infty}\pm \infty$.
\end{enumerate}
\end{remark}

Before giving some concrete examples
and an equivalent description for symbols in $\dot S^m$, 
let us mention some routine properties regarding classes of symbols.
Each $\dot S^m$, $m\in \bR$, is a Frechet vector space 
when equipped with 
the seminorms
$$
\|\sigma\|_{\dot S^m,a,b,c}:=
\max_{\substack {[\alpha]\leq a \\
[\beta]\leq b, |\gamma|\leq c}}
\sup_{\substack {x\in G\\ \pi\in \Gh}}
\|\pi(\cR)^{\frac{ [\alpha]-m +\gamma}\nu }
X_x^\beta\Delta^\alpha \sigma(x,\pi) 
\pi(\cR)^{-\frac{\gamma}\nu }\|_{\sL(\cH_\pi)}
,\quad a,b,c\in \bN_0,
$$
where $\cR$ is a positive Rockland operator $\cR$ of homogeneous degree $\nu$. 
 This Frechet structure is independent of the chosen positive Rockland operator and we will see later in Corollary \ref{cor_prop_equiv_dotS_gamma} that we may assume $c=0$.
Furthermore taking the product and the adjoint and applying spacial and dual derivatives
$$
\left\{\begin{array}{rcl}
\dot S^{m_1}\times \dot S^{m_2} &\longrightarrow& \dot S^{m_1+m_2}\\
(\sigma_1,\sigma_2)&\longmapsto& \sigma_1\sigma_2
\end{array}
\right.,\quad
\left\{\begin{array}{rcl}
\dot S^{m} &\longrightarrow& \dot S^{m}\\
\sigma&\longmapsto& \sigma^*
\end{array}\right.,
\quad\mbox{and}\quad
\left\{\begin{array}{rcl}
\dot S^{m} &\longrightarrow& \dot S^{m-[\alpha]}\\
\sigma&\longmapsto&X^\beta \Delta^\alpha \sigma
\end{array}\right.,
$$
(for any $m,m_1,m_2\in \bR$ and $\alpha,\beta\in \bN_0^n$)
 are continuous operations for this topology.

\medskip

\begin{ex}
The symbol $\pi(X)^\alpha$ is homogeneous of degree $[\alpha]$ and regular.
(See \eqref{eq_rpiXalpha} and Example~\ref{ex_Delta_piX})
\end{ex}

\begin{ex}
The symbol given by a  function $\sigma(x)$ independent of $G$
is homogeneous of degree 0.
It is regular if the function is smooth with bounded left invariant derivatives,
see~\eqref{eq_bdd_left_derivatives}. 
\end{ex}

\begin{ex}
\label{ex_RindotSm}
If $\cR$ is a positive Rockland operator of degree $\nu$
and if $m\in \bR$,
then the symbol $\pi(\cR)^{\frac m\nu}$ (defined spectrally) 
is regular and homogeneous  of degree $m$.
\end{ex}
\begin{proof}[Proof for Example \ref{ex_RindotSm}]
The homogeneity may be obtained from the properties of the Rockland operator 
as in Lemma \ref{lem_homogeneity_multipliers}.
The regularity will be a direct consequence of Proposition
\ref{prop_equiv_dotS} below.
\end{proof}

We now give equivalent properties characterising  a symbol in $\dot S^m$.
In the abelian case, the statement boils down 
to the fact that a regular homogeneous symbol yields a 
(non-homogeneous) symbol in $S^m(\bR^n)$ 
once the low frequencies have been cut off.

\begin{proposition}
\label{prop_equiv_dotS}
Let $\sigma=\{\sigma(x,\pi):\cH_\pi^\infty\to\cH_\pi^\infty, (x,\pi)\in G\times \Gh\}$ 
be a homogeneous symbol of degree $m\geq 0$.
The following properties are equivalent.
\begin{enumerate}
\item 
$\sigma$ is  in $\dot S^m$
\item 
There exist 
a positive Rockland operator $\cR$
and a real-valued function
$\psi\in C^\infty(\bR)$ 
satisfying 
$\psi\equiv 0$ on a neighbourhood of 0 and 
$\psi\equiv 1$ on $(\Lambda,\infty)$
for some $\Lambda>0$
such that the two symbols
 $$
 \{\psi(\pi(\cR)) \sigma(x,\pi), (x,\pi)\in G\times\Gh\} 
 \quad\mbox{and}\quad
 \{ \sigma(x,\pi)\psi(\pi(\cR)), (x,\pi)\in G\times\Gh\},
 $$
 are in $S^m$. 
\item Property (2) holds for any such $\cR$ and $\psi$.
\end{enumerate}
Moreover the mapping 
$$
\left\{\begin{array}{rcl}
\dot S^{m}&\longrightarrow& S^{m}\times S^{m} \\
\sigma&\longmapsto&(\psi(\pi(\cR)) \sigma, \sigma\psi(\pi(\cR)))
\end{array}
\right.,$$
is continuous, injective and open (i.e. with continuous inverse onto its image).
\end{proposition}

The proof is given in the next subsection.
But first  let us notice that using 
$\psi(\pi(\cR)) \sigma$ or  $\sigma \psi(\pi(\cR))$
is essentially equivalent as we have:

\begin{corollary}
\label{cor_prop_equiv_dotS}
If $\sigma \in \dot S^m$, $\cR$ is a positive Rockland operator and 
$\psi\in C^\infty(\bR)$ is a scalar valued function such that 
$\psi\equiv0$ on a neighbourhood of 0
and 
$\psi\equiv1$ on $[\Lambda,\infty)$ for some $\Lambda>0$, 
then 
$\psi(\pi(\cR)) \sigma -\sigma \psi(\pi(\cR))$ is in $S^{-\infty}$.
\end{corollary}

\begin{proof}[Proof of Corollary \ref{cor_prop_equiv_dotS}]
We keep  the hypotheses of the corollary and denote 
$\Psi(\pi):=\psi(\pi(\cR))$.
As
$\Psi(\pi)=(1-\psi)(\pi(\cR))$, 
by Corollary \ref{cor_prop_multipliers_symbol}, 
$\id-\Psi$ is smoothing.
Since $\Psi \sigma$ and $\sigma \Psi$ are in $S^m$
by Proposition \ref{prop_equiv_dotS},
the symbol
$
\Psi \sigma -\sigma \Psi
=
\Psi \sigma (\id-\Psi) 
-(\id-\Psi) \sigma \Psi,
$
is smoothing.
\end{proof}

\subsection{Proof of Proposition \ref{prop_equiv_dotS}}
The underlying idea is to find 
 a replacement for the following property in the Euclidean case: in the case of $\bR^n$, 
if a cut-off function $\psi(\xi)$ on the Fourier side is constant 
for $|\xi|>\Lambda$ ($\Lambda$ large enough),
then its derivatives are $\partial_\xi^\alpha \psi (\xi)=0$ if $|\xi|>\Lambda$. In our case, we can not say anything in general about vanishing derivatives.
However, we can show that these derivatives are smoothing and behaves well enough in the following way:

\begin{lemma}
\label{lem_Delta_psiL}
Let $\psi\in C^\infty(\bR)$ be a real valued function 
satisfying 
$\psi_{|[\Lambda,+\infty)} = 1$ for some $\Lambda\in \bR$.
Let $\cR$ be a positive Rockland operator of homogeneous degree $\nu$.
Then for any $\alpha \in \bN_0^n\backslash\{0\}$,
the symbol given by 
$\Delta^{\alpha} \psi(\pi(\cR))$
is smoothing, i.e. is in $S^{-\infty}$.
Furthermore for each $a,b\in \bR$,
the fields of operators given by
$$
 \pi(\cR)^{\frac b\nu} \Delta^\alpha \psi(\pi(\cR)) \pi(\cR)^{\frac a\nu}
,\quad 
 \pi(\id+\cR)^{\frac b\nu} \Delta^\alpha \psi(\pi(\cR)) \pi(\cR)^{\frac a\nu}
 ,\quad
 \pi(\cR)^{\frac b\nu} \Delta^\alpha \psi(\pi(\cR)) \pi(\id+\cR)^{\frac a\nu},
$$
are in $L^\infty(\Gh)$.
\end{lemma}

\begin{proof}[Proof of Lemma \ref{lem_Delta_psiL}]
By Example \ref{ex_Delta_piX}, 
$\Delta^\alpha 1=0$
thus
$\Delta^{\alpha} \psi(\pi(\cR))=-\Delta^{\alpha}(1-\psi)(\pi(\cR))$.
Corollary \ref{cor_prop_multipliers_symbol} then implies
 the first part.
If $a,b\in \nu\bN_0$, 
the given fields of operators are bounded since 
 $\Delta^{\alpha} \psi(\pi(\cR))\in S^{-\infty}$
 while
 $\pi(\cR)^{\frac m\nu}$ and $\pi(\id+\cR)^{\frac m\nu}$ are in $\Psi^m$ for any $m\in \nu \bN_0$.
 Hence this is also the case for $a,b \in \nu(-\bN_0)$
by duality (see Theorem \ref{thm_sobolev_spaces}
\eqref{item_thm_sobolev_spaces_duality}), 
and then for any $a,b\in \bR$  by interpolation
(see Theorem \ref{thm_sobolev_spaces}
\eqref{item_thm_sobolev_spaces_interpolation});
indeed the adjoint of  $\left(\Op\left(\Delta^\alpha \psi(\pi(\cR))\right)\right)^*$ is a linear combination of 
$\left(\Op\left(\Delta^\beta \psi(\pi(\cR))\right)\right)$,
$[\beta]=[\alpha]$.
\end{proof}

We will also need the following technical lemma.

\begin{lemma}
\label{lem_0homogeneous_field}
 Let $\sigma=\{\sigma(\pi):\cH_\pi^\infty\to \cH_\pi^\infty,\pi\in \Gh\}$ be a measurable field of operators which is 0-homogeneous in the sense that
$\sigma(r\cdot\pi)=\sigma(\pi)$ for any  $r>0$ and  $\pi\in \Gh$.

\begin{enumerate}
\item 
If  there exist a positive Rockland operator $\cR$ and
a scalar valued function $\psi\in C^\infty(\bR)$ such that 
$\psi\equiv 1$ on $(\Lambda,\infty)$ 
for some $\Lambda\in \bR$,
and $\{\psi(\pi( \cR)) \sigma(\pi)\} \in L^\infty(\Gh)$
then  $\sigma\in L^\infty(\Gh)$.

Conversely, if $\sigma\in L^\infty(\Gh)$,
then $\{\psi(\pi( \cR)) \sigma(\pi)\} \in L^\infty(\Gh)$
for any positive Rockland operator $\cR$ and
any scalar valued function $\psi\in C^\infty(\bR)$ such that 
$\psi\equiv 1$ on $(\Lambda,\infty)$ 
for some $\Lambda\in \bR$
and we have:
$$
\|\sigma\|_{L^\infty(\Gh)}
\leq
\sup_{\pi\in \Gh} \|\psi(\pi( \cR)) \sigma(\pi)\|_{\sL(\cH_\pi)}
\leq\sup_{\lambda\geq 0} |\psi(\lambda)| \
\|\sigma\|_{L^\infty(\Gh)}
.$$

We have the same property with $\{ \sigma(\pi) \psi(\pi( \cR))\}$.

\item 
We assume that $\sigma\in L^\infty(\Gh)$
and that there exist a real-valued function
 $\psi\in C^\infty(\bR)$ satisfying
$\psi\equiv 0$ on a neighbourhood of 0 and
$\psi\equiv 1$ on $(\Lambda,+\infty)$ 
for some $\Lambda\in \bR$, a positive Rockland operator $\cR$
of homogeneous degree $\nu$
and a number $m'\in \bR$
such that one of the quantities
$$
\sup_{\pi\in \Gh} 
\|\pi(\cR)^{\frac {m'}{\nu}} \psi(\pi(\cR)) \sigma(\pi)\|_{\sL(\cH_\pi)}
\quad\mbox{or}\quad
\sup_{\pi\in \Gh} 
\|\sigma(\pi)\psi(\pi(\cR))\pi(\cR)^{\frac {m'}{\nu} } \|_{\sL(\cH_\pi)},
$$
is finite.
If $m'>0$ then $\sigma=0$.
\end{enumerate}
\end{lemma}

\begin{proof}[Proof of Lemma \ref{lem_0homogeneous_field}]
Let $\sigma=\{\sigma(\pi)\in \sL(\cH_\pi)\}$ be a 0-homogeneous symbol.
Let us assume that the quantity
$$
\sup_{\pi\in \Gh}\|\sigma(\pi)\|_{\sL(\cH_\pi)}
=
\sup_{\substack{\pi\in \Gh\\ u\in \cH_\pi, \|u\|_{\cH_\pi}=1}}\|\sigma(\pi)u\|_{\cH_\pi}
$$
is finite. Then, 
for any $\epsilon>0$, there exists $\pi_0\in \Gh$ and $u_0\in \cH_{\pi_0}$, $\|u_0\|_{\cH_{\pi_0}}=1$  such that 
$$
\sup_{\pi\in \Gh}\|\sigma(\pi)\|_{\sL(\cH_\pi)}
\leq
\|\sigma(\pi_0)u_0\|_{\cH_{\pi_0}} +\epsilon.
$$
By \eqref{eq_psi(rpiR)=I}, for any $r>r_{\pi_0}$, 
we have $\psi(r\cdot \pi_0 (\cR))=\id_{\cH_{\pi_0}}$ 
thus 
$$
\psi(r\cdot \pi_0 (\cR))\sigma(\pi_0)u_0
=
\sigma(\pi_0)u_0.
$$
As $\sigma$ is 0-homogeneous, we have 
$\sigma(\pi_0)u_0 =\sigma(r\cdot\pi_0)u_0$.
Hence 
$$
\sup_{\pi\in \Gh} \|\psi(\pi(\cR))\sigma (\pi)\|_{\sL(\cH_\pi)}
\geq 
\|\psi(r\cdot \pi_0 (\cR))\sigma(\pi_0)u_0\|_{\sL(\cH_{\pi_0})}
\geq 
\sup_{\pi\in \Gh}\|\sigma(\pi)\|_{\sL(\cH_\pi)} -\epsilon.
$$
This is true for any $\epsilon>0$ and this shows 
$$
\sup_{\pi\in \Gh}\|\sigma(\pi)\|_{\sL(\cH_\pi)}
\leq
\sup_{\pi\in \Gh}\|\psi (\pi(\cR))\sigma(\pi)\|_{\sL(\cH_\pi)}.
$$
The rest of Part (1) follows  
from similar considerations and properties of the functional calculus of $\pi(\cR)$. 

For Part (2), let us assume that the first quantity is finite, that is, 
$$
\sup_{\pi\in \Gh} 
\|\pi(\cR)^{\frac {m'}{\nu}} \psi(\pi(\cR)) \sigma(\pi)\|_{\sL(\cH_\pi)}
<\infty.
$$
For each $r>0$, we use the change of index $\pi \mapsto r\cdot \pi$
and this quantity becomes by homogeneity
\begin{eqnarray}
\sup_{\pi\in \Gh} 
\|r\cdot\pi(\cR)^{\frac {m'}{\nu}} \psi(r\cdot\pi(\cR)) \sigma(r\cdot\pi)\|_{\sL(\cH_{\pi})}
&=&
r^{m'}\sup_{\pi\in \Gh} 
\|\pi(\cR)^{\frac {m'}{\nu}} \psi(r^\nu \pi(\cR)) \sigma(\pi)\|_{\sL(\cH_{\pi})}\nonumber
\\
&\geq&
r^{m'}\sup_{\substack{\pi\in \Gh\\ r>r_\pi}} 
\|\pi(\cR)^{\frac {m'}{\nu}}  \sigma(\pi)\|_{\sL(\cH_{\pi})},
\label{eq_lem_0homogeneous_field_2_ineq}
\end{eqnarray}
having used \eqref{eq_psi(rpiR)=I}.
Note that the limit
\begin{equation}
\label{eq_lem_0homogeneous_field_2_lim}
\lim_{r\to\infty}\sup_{\substack{\pi\in \Gh\\ r>r_\pi}} 
\|\pi(\cR)^{\frac {m'}{\nu}}  \sigma(\pi)\|_{\sL(\cH_{\pi})}
=
\sup_{\pi\in \Gh} 
\|\pi(\cR)^{\frac {m'}{\nu}}  \sigma(\pi)\|_{\sL(\cH_{\pi})},
\end{equation}
is infinite unless either $m'=0$ or $\sigma=0$.
Indeed, for the same reason as above, we have for any $r_1>0$:
\begin{eqnarray*}
\sup_{\pi\in \Gh} 
\|\pi(\cR)^{\frac {m'}{\nu}}  \sigma(\pi)\|_{\sL(\cH_{\pi})}
&=&
\sup_{\pi\in \Gh} 
\|r_1\cdot\pi(\cR)^{\frac {m'}{\nu}}  \sigma(r_1\cdot\pi)\|_{\sL(\cH_{\pi})}
\\
&=&
r_1^{m'}
\sup_{\pi\in \Gh} 
\|\pi(\cR)^{\frac {m'}{\nu}}  \sigma(\pi)\|_{\sL(\cH_{\pi})}.
\end{eqnarray*}
If the limit in \eqref{eq_lem_0homogeneous_field_2_lim} is infinite,
as each side of \eqref{eq_lem_0homogeneous_field_2_ineq} must be finite as $r\to \infty$, 
we then must have $r^{m'}\to 0$ as $r\to\infty$, 
that is, $m'<0$.
This conclude the proof of Part (2).
\end{proof}

We can now prove Proposition \ref{prop_equiv_dotS}.

\begin{proof}[Proof of Proposition \ref{prop_equiv_dotS}]
Let $\cR$ be 
a positive Rockland operator of homogeneous degree $\nu$.
Let $\psi\in C^\infty(\bR)$ be a real valued function
satisfying 
$\psi\equiv 0$ on $(-\infty,\epsilon_o)$ and 
$\psi\equiv 1$ on $(\Lambda,\infty)$
for some $0<\epsilon_o<\Lambda$.

Let us assume Property (1), that is,  let $\sigma\in \dot S^m$.
We will prove that for any $\alpha,\beta\in \bN_0^n$,
 \begin{equation}
\label{eq1_prop_equiv_dotS}
\sup_{x\in G, \pi\in \Gh}
\|\pi(\id+\cR)^{\frac{ [\alpha]-m }\nu }
X_x^\beta\Delta^\alpha \{
\psi(\pi(\cR)) \sigma(x,\pi)\}\|_{\sL(\cH_\pi)}<\infty.
\end{equation}
The case $\alpha=0$ of \eqref{eq1_prop_equiv_dotS}
 follows from
\begin{eqnarray*}
&&\|\pi(\id+\cR)^{\frac{-m }{\nu}}
X^\beta\psi(\pi(\cR)) \sigma(x,\pi)\|_{\sL(\cH_\pi)}
\\&&\quad\leq
\|\pi(\id+\cR)^{\frac{-m }{\nu}}
\psi(\pi(\cR)) 
\pi(\cR)^{-\frac{-m }{\nu}}\|_{\sL(\cH_\pi)}
  \|\pi(\cR)^{\frac{-m }{\nu}}
X^\beta \sigma(x,\pi) \|_{\sL(\cH_\pi)}
\\&&\quad\leq
\sup_{\lambda>0}
|(1+\lambda)^{\frac{-m }{\nu}}
\psi(\lambda) 
\lambda^{-\frac{-m }{\nu}}|
  \|\pi(\cR)^{\frac{-m }{\nu}}
X^\beta \sigma(x,\pi) \|_{\sL(\cH_\pi)} <\infty.
\end{eqnarray*}

Let $\alpha\in \bN_0^n$ be such that $|\alpha|=1$.
Using the Leibniz rule \eqref{eq_Leibniz_rule_sigma}, we obtain
\begin{eqnarray*}
&&
\|\pi(\id+\cR)^{\frac{[\alpha]-m }{\nu}}
X^\beta\Delta^\alpha\{\psi(\pi(\cR)) \sigma(x,\pi)\} \|_{\sL(\cH_\pi)}
\\&&\quad\lesssim
\|(\id+\pi(\cR))^{\frac{[\alpha]-m }{\nu}}
\left(\Delta^\alpha\psi(\pi(\cR))\right) X^\beta \sigma(x,\pi) \|_{\sL(\cH_\pi)}
\\&&\qquad\qquad+
\|(\id+\pi(\cR))^{\frac{[\alpha]-m }{\nu}}
\psi(\pi(\cR))
\left(\Delta^\alpha X^\beta \sigma(x,\pi) \right) \|_{\sL(\cH_\pi)}
\end{eqnarray*}
For the second term on the right hand-side, 
we proceed as above.
We modify the argument for the first  term
using Lemma \ref{lem_Delta_psiL}.
Recursively we prove \eqref{eq1_prop_equiv_dotS}
 with the same arguments. The case of $\sigma(x,\pi)\psi(\pi(\cR))$ are handled in a similar way, the details are left to the reader.
 We have proved $(1)\Longrightarrow (2)$ and $(1)\Longrightarrow (3)$.

Since $(3)\Longrightarrow (2)$, it only remains to prove that $(2)\Longrightarrow (1)$.
We assume that 
$\{\psi(\pi(\cR)) \sigma(x,\pi)\}$ 
and
$ \{ \sigma(x,\pi)\psi(\pi(\cR))\}$
are in $S^m$
and we want to prove that
$$\sup_{x\in G, \pi\in \Gh}
\|\pi(\cR)^{\frac{ [\alpha]-m+\gamma }\nu }
X_x^\beta\Delta^\alpha \{
\psi(\pi(\cR)) \sigma(x,\pi)\}\pi(\cR)^{-\frac{\gamma }\nu }\|_{\sL(\cH_\pi)}<\infty,
$$
for any $\alpha,\beta\in \bN_0^n$ and 
 for a sequence $(\gamma_\ell)_{\ell\in \bZ}$
with $\gamma_\ell\rightarrow_{\ell\to\pm\infty}\pm \infty$
(see Remark \ref{rem_def_dotSmrhodelta} (2)),
and similarly for $\sigma\psi(\pi(\cR)) $.
Clearly it suffices to show it for $\beta=0$.
Using recursively the Leibniz rule (see \eqref{eq_Leibniz_rule_sigma})
and Lemma \ref{lem_Delta_psiL},
it suffices to show 
$$
\sup_{x\in G, \pi\in \Gh}
\|\pi(\cR)^{\frac{ [\alpha]-m+\gamma }\nu }
\psi(\pi(\cR)) \Delta^\alpha \{
\sigma(x,\pi)\}\pi(\cR)^{-\frac{\gamma }\nu }\|_{\sL(\cH_\pi)}<\infty,
$$
for any $\alpha\in \bN_0^n$ and
 for a sequence $(\gamma_\ell)_{\ell\in \bZ}$
with $\gamma_\ell\rightarrow_{\ell\to\pm\infty}\pm \infty$
(see Remark \ref{rem_def_dotSmrhodelta} (2)),
and similarly for $ \Delta^\alpha \{\sigma(x,\pi)\} \ \psi(\pi(\cR)) $.
By homogeneity of the operator $\Delta^\alpha$ (see \eqref{eq_Deltarcdotpi}), 
it suffices to prove the case $\alpha=0$ which we now do.

The field of operators 
$\{\pi(\cR)^{\frac{-m+\gamma}\nu}  \sigma(x,\pi)\pi(\cR)^{-\frac{\gamma}\nu},\pi\in \Gh\}$
is 0-homogeneous.
Thus by Lemma \ref{lem_0homogeneous_field} (1)
and functional analysis 
\begin{eqnarray*}
\sup_{\pi\in \Gh}
\|\pi(\cR)^{\frac{-m+\gamma}\nu} \sigma(x,\pi)
\pi(\cR)^{-\frac{\gamma}\nu}
\|_{\sL(\cH_\pi)}
\leq 
\sup_{\pi\in \Gh}
\|\psi(\pi(\cR))\pi(\cR)^{\frac{-m+\gamma}\nu} \sigma(x,\pi) \pi(\cR)^{-\frac{\gamma}\nu}\|_{\sL(\cH_\pi)}
\\
\leq C
\sup_{\pi\in \Gh}
\|\pi(\id+\cR)^{\frac{-m+\gamma}\nu} \{\psi(\pi(\cR))\sigma(x,\pi)\} \pi(\id +\cR)^{-\frac{\gamma}\nu}\|_{\sL(\cH_\pi)},
\end{eqnarray*}
with a constant 
$$
C=C_{m,\gamma,\nu}
=\sup_{\lambda_1>\epsilon_o} \left(\frac{1+\lambda_1}{\lambda_1}\right)^{\frac{m-\gamma}\nu}
\sup_{\lambda_2>0} 
\left( \frac{\lambda_2}{1+\lambda_2}\right)^{-\frac \gamma\nu},
$$
 finite for $\gamma\geq0$.
We apply the same argument to $\sigma^*$ and obtain
\begin{eqnarray*}
&&\sup_{\pi\in \Gh}
\|\pi(\cR)^{-\frac{\gamma}\nu}
X^\beta \sigma(x,\pi)
\pi(\cR)^{\frac{-m+\gamma}\nu} \|_{\sL(\cH_\pi)}
=
\sup_{\pi\in \Gh}
\|\pi(\cR)^{\frac{-m+\gamma}\nu} X^\beta \sigma(x,\pi)^*
\pi(\cR)^{-\frac{\gamma}\nu}
\|_{\sL(\cH_\pi)}
\\&&\qquad \leq C
\sup_{\pi\in \Gh}
\|\pi(\id+\cR)^{\frac{-m+\gamma}\nu} X^\beta \{\psi(\pi(\cR))\sigma(x,\pi)^*\} \pi(\id +\cR)^{-\frac{\gamma}\nu}\|_{\sL(\cH_\pi)}
\\&&\qquad = C
\sup_{\pi\in \Gh}
\|\pi(\id+\cR)^{\frac{-m+\gamma}\nu} X^\beta \{\sigma(x,\pi) \psi(\pi(\cR))\} \pi(\id +\cR)^{-\frac{\gamma}\nu}\|_{\sL(\cH_\pi)}.
\end{eqnarray*}
This concludes the proof of $(2)\Longrightarrow (1)$. 
The rest of the proof follows easily.
\end{proof}

\subsection{Consequence of Proposition \ref{prop_equiv_dotS} and of  its proof}

The proof of  Proposition \ref{prop_equiv_dotS},
especially the implication 
(1)$\Longrightarrow$(2),  
together with Proposition \ref{prop_multipliers_symbol}
imply
\begin{corollary}
\label{cor_prop_equiv_dotS_bound}
Let $\cR$ be 
a positive Rockland operator.
Then for any seminorm $\|\cdot\|_{S^m,a,b,c}$ of $S^m$, 
there exist a seminorm $\|\cdot\|_{\dot S^m,a',b',c'}$
of $\dot S^m$, $C>0$, and $k\in \bN$ such that
for any $\sigma\in \dot S^m$
and for any  real-valued function $\psi\in C^\infty(\bR)$ 
satisfying $\psi\equiv 0$ on a neighbourhood of 0 and 
$\psi\equiv 1$ on $(\Lambda,\infty)$
for some $\Lambda>0$, we have:
$$
\max(\|\psi(\pi(\cR)) \sigma\|_{S^m,a,b,c},
\| \sigma\psi(\pi(\cR))\|_{S^m,a,b,c})
\leq C \sup_{\substack{\lambda>0\\ \ell=0,\ldots,k}}
(1+\lambda) ^{\ell} \left| \partial_\lambda^\ell \psi(\lambda) \right|
\ 
\|\sigma\|_{\dot S^m,a',b',c'},
$$
\end{corollary}
Note that the constant $C$ and the integer $k$ are independent of $\sigma$ or $\psi$.

\medskip

In Section \ref{sec_0homsymb}, we will analyse more precisely the homogeneous symbols of degree 0. From Proposition \ref{prop_equiv_dotS}
and Corollary \ref{cor_prop_equiv_dotS_bound}, we can already prove the following regularity of their kernel.
If $\sigma\in \dot S^0$, 
then, for each $x\in G$,  $\sigma(x,\cdot)\in L^\infty(\Gh)$
 has a kernel $\kappa_x\in \cS'(G)$ 
such that $\sigma(x,\pi)=\widehat \kappa_x$, see
Section~\ref{subsec_Gh+plancherel}.
The regularity of the symbol implies that this distribution coincides with a smooth function:

\begin{proposition}
\label{prop_hom_kernel_smooth}
Let $\sigma\in \dot S^0$ and let $\kappa_x\in \cS'(G)$ be its kernel, 
i.e. $\pi(\kappa_x)=\sigma(x,\pi)$.

Then for each $x\in G$,  
the distribution $\kappa_x$ is $(-Q)$-homogeneous:
$$
\forall r>0, \qquad \kappa_x(ry)= r^{-Q} \kappa_x(y).
$$
For each $x\in G$, the distribution $\kappa_x$ 
coincides with a smooth function away from the origin
and the function $(x,z)\mapsto \kappa_x(z)$ is smooth on $G\times (G\backslash\{0\})$.
Furthermore, for any compact subset $S$ of  $G\backslash\{0\}$
and any $\alpha,\beta\in \bN_0^n$, 
there exist a constant $C>0$ and a seminorm 
 $\|\cdot\|_{\dot S^0,a,b,c}$ such that
$$
\sup_{\substack{x\in G\\ z\in S}} |X^\alpha_z X^\beta_x \kappa_x(z)| \leq C_{s} \|\sigma\|_{\dot S^0,a,b,c}.
$$
The constant $C$ and the seminorm $\|\cdot\|_{\dot S^0,a,b,c}$
may depend on $S$ and $\alpha,\beta$ but are independent of $\sigma$.
\end{proposition}

\begin{proof}[Proof of Proposition \ref{prop_hom_kernel_smooth}]
Let $\cR$ be a positive Rockland operator of homogeneous degree $\nu$.
Let $\psi\in C^\infty(\bR)$ be
 a real-valued function
satisfying 
$\psi\equiv 0$ on a neighbourhood of 0 and 
$\psi\equiv 1$ on $(\Lambda,\infty)$
for some $\Lambda>0$.

Let $\sigma\in \dot S^0$ and let $\kappa_x$ be its associated kernel.
For each $t>0$, we set $\sigma_{(t)}(x,\pi)=\sigma (x,\pi) \psi(t\pi(\cR))$.
By Proposition \ref{prop_equiv_dotS}, 
this defines a symbol $\sigma_{(t)}$ in $S^0$
and we denote by $\kappa_{(t)}$ its kernel.
Lemma \ref{lem_normL2_E(epsilon,infty)}~(3)
and the $L^2$-boundedness of $\Op(S^0)$ imply
 that for each $x\in G$, 
$\Op(\sigma_{(t)}(x,\cdot)) =
\Op(\sigma(x,\cdot))\psi(t\cR)$ 
converges to $\Op(\sigma(x,\cdot))$ 
as $t\to0$ for the strong operator topology of $L^2(G)$.
This implies that $\kappa_{(t),x}$ converges to $\kappa_x$ in $\cS'(G)$
for each $x\in G$ as $t\to 0$.
More generally, for each $x\in G$ and each $\beta\in \bN_0^n$, 
 $X_x^\beta \kappa_{(t),x}$ converges to $X_x^\beta\kappa_x$ in $\cS'(G)$ as $t\to 0$.
The statement now follows from the convergence in distribution, 
Proposition \ref{prop_Sm_kernel}
and Corollary \ref{cor_prop_equiv_dotS_bound}.
\end{proof}

\medskip

Another consequence of Proposition \ref{prop_equiv_dotS} and its proof is that as in the inhomogeneous case (see Theorem \ref{thm_calculus}, Part (2)), 
we can simplify the conditions on the regularity of the symbol:

\begin{corollary}
\label{cor_prop_equiv_dotS_gamma}
Let $\sigma$ be a homogeneous symbol of degree $m\in \bR$.
Then  $\sigma$ is in $\dot S^m$ if and only if 
$$
\sup_{\substack{\pi\in \Gh\\ x\in G}}
\| \pi(\cR)^{\frac{[\alpha]-m}{\nu}}
X_x^\beta \Delta^\alpha \sigma(x,\pi)\|_{\sL(\cH_\pi)} 
\quad,\quad
\sup_{\substack{\pi\in \Gh\\ x\in G}}
\| 
X_x^\beta \Delta^\alpha \sigma(x,\pi)\pi(\cR)^{\frac{[\alpha]-m}{\nu}}\|_{\sL(\cH_\pi)} ,
$$
are finite for all $\alpha,\beta\in \bN_0^n$.
Here $\cR$ is a fixed positive Rockland operator of degree $\nu$.
Furthermore, for a fixed positive Rockland operator,  
these quantities  yield an equivalent 
family of seminorms for the Fr\'echet topology of $\dot S^m$. 
\end{corollary}

Finally, we observe that  Proposition \ref{prop_equiv_dotS} implies the following property:
\begin{corollary}
\label{cor_sigmachi}
Let $\chi \in \cD(\bR)$ with support in $(0,\infty)$.
Let $\cR$ be a positive Rockland operator.
Let $\sigma\in \dot S^m$.
Then $\sigma \, \chi(\pi(\cR))$ and $\chi(\pi(\cR))\, \sigma$ are smoothing, 
i.e. in $S^{-\infty}$.
  Consequently if $\sigma$ does not depend on $x$ then its kernel is Schwartz.
 \end{corollary}
 
 \begin{proof}
The first part follows from Proposition \ref{prop_equiv_dotS}
and Corollary \ref{cor_prop_multipliers_symbol}.
The consequence follows from \cite[Theorem 5.4.9]{R+F_monograph}.
\end{proof}

\subsection{Homogeneous asymptotic and principal part}

In this subsection, we give a meaning to a homogeneous asymptotic sum 
\begin{equation}
\label{eq_def_homogeneous_asymptotic}
\sigma \sim \sum_{\ell=0}^\infty \sigma_\ell,
\quad \sigma_\ell \in \dot S^{m_\ell}, 
\quad\mbox{with}\ m_\ell \ \mbox{strictly decreasing to} \ -\infty,
\end{equation}
which is different to the (inhomogeneous) asymptotic sum in \eqref{eq_def_asymptotic}. This will enable us to define the principal part $\sigma_0$ of such an expansion.
In order to give a meaning to \eqref{eq_def_homogeneous_asymptotic},
we show:
\begin{proposition}
\label{prop_homogeneous_asymptotic}
Let $\{\sigma_\ell\}_{\ell\in \bN_0}$ be a sequence of homogeneous 
symbols such that
 $\sigma_\ell\in \dot S^{m_\ell}_{\rho,\delta}$ with $m_\ell$ strictly decreasing to $-\infty$.
If  $\cR$ is any positive Rockland operator
and 
 $\psi\in C^\infty(\bR)$ 
is any  real-valued function
satisfying 
$\psi\equiv 0$ on a neighbourhood of 0 and 
$\psi\equiv 1$ on $(\Lambda,\infty)$ 
for some $\Lambda>0$, 
then the two sums 
$$
 \sum_\ell  \sigma_\ell\psi(\pi(\cR))
\quad\mbox{and}\quad
 \sum_\ell \psi(\pi(\cR)) \sigma_\ell,
$$
define
the same symbol in $S^{m_0}$ modulo $S^{-\infty}$.

Moreover, this symbol modulo $S^{-\infty}$ does not depend on the choice of $\cR$ and $\psi$. And, if this symbol is in $S^m$ with $m<m_0$, 
then the first term in the homogeneous expansion is $\sigma_0=0$.
\end{proposition}

The proof of Proposition \ref{prop_homogeneous_asymptotic} 
relies on the following property:

\begin{lemma}
\label{lem_cq_prop_multipliers_symbol}
 If $\cR_1$ and $\cR_2$ are two positive Rockland operators
and if $\psi\in C^\infty(\bR)$ is a real-valued function such that 
$\psi\equiv 0$ on a neighbourhood of 0 and
$\psi\equiv 1$ on $(\Lambda,+\infty)$ 
for some $\Lambda\in \bR$,
then
$\{\psi(\pi(\cR_2))-\psi(\pi(\cR_1)),\pi\in \Gh\} \in S^{-\infty}.$
\end{lemma}

This Lemma is proved in~\cite{fischer16}.

\begin{proof}[Proof of Proposition \ref{prop_homogeneous_asymptotic}]
Let $\{\sigma_\ell\}_{\ell\in \bN_0}$ be as in the statement.
The two sums 
$\sum_\ell  \sigma_\ell\psi(\pi(\cR))$ and
$ \sum_\ell \psi(\pi(\cR)) \sigma_\ell$
make sense as symbols in $S^{m_0}$ modulo $S^{-\infty}$, 
see \eqref{eq_def_asymptotic} and Theorem \ref{thm_asymptotic}.
They yield the same symbol modulo $S^{-\infty}$ by Corollary \ref{cor_prop_equiv_dotS}.
 
We need to show that if $\cR_1$ and $\cR_2$ are two positive Rockland operators
and $\psi_1,\psi_2\in C^\infty(\bR)$ are two real-valued functions
satisfying 
$\psi_1\equiv 0$ and $\psi_2\equiv 0$
 on a neighbourhood of 0 and 
$\psi_1\equiv 1$ and $\psi_2\equiv 1$ on $(\Lambda,\infty)$
for some $\Lambda>0$
then 
$$
\sum_\ell\psi_1(\pi(\cR_1)) \sigma_\ell
-
\sum_\ell \psi_2(\pi(\cR_2)) \sigma_\ell
\in S^{-\infty}.
$$
It suffices to show this
in the case of $\cR_1=\cR_2$ and two (different) functions $\psi_1,\psi_2$
and in the case of $\psi_1=\psi_2$ and two (different) $\cR_1,\cR_2$.
This follows from Corollary \ref{cor_prop_multipliers_symbol} and Lemma \ref{lem_cq_prop_multipliers_symbol} respectively.

Now let us assume that the symbol defined by $\sum_\ell  \sigma_\ell\psi(\pi(\cR))$ is in $S^m$ with $m<m_0$. We may assume that $m_1<m<m_0$.
Then $\sigma_0\psi(\pi(\cR)) \in S^m$.
Denoting $\nu$ the homogeneous degree of $\cR$, 
we have $\psi(\pi(\cR))\pi(\cR)^{-\frac m \nu} \in S^m$ 
and 
$$
S^0 \ni 
\sigma_0\psi(\pi(\cR)) \psi(\pi(\cR))\pi(\cR)^{-\frac m \nu}
=
\sigma_0 \pi(\cR)^{-\frac {m _0}\nu}
\ \psi^2(\pi(\cR))  \pi(\cR)^{\frac {m'}\nu} 
$$
with $m':=-m+m_0>0$.
For each $x\in G$, 
the $0$-homogeneous field 
 $\{\sigma_0(x,\pi) \pi(\cR)^{-\frac {m _0}\nu}, \pi\in \Gh\}$
satisfies the hypotheses of  Lemma \ref{lem_0homogeneous_field} (2)
and thus must be zero.
This conclude the proof of Proposition \ref{prop_homogeneous_asymptotic}.
\end{proof}

Proposition \ref{prop_homogeneous_asymptotic}
allows us to give a meaning to a homogeneous expansion as in  \eqref{eq_def_homogeneous_asymptotic}:

\begin{definition}
\label{def_homogeneous_asymptotic}
Let $\{\sigma_\ell\}_{\ell\in \bN_0}$ be a sequence of homogeneous 
symbols such that
 $\sigma_\ell\in \dot S^{m_\ell}$ with $m_\ell$ strictly decreasing to $-\infty$.
Then $\sum_{\ell=0}^\infty \sigma_\ell$ denotes
the symbol $\sigma$ in $S^{m_0}$ modulo $S^{-\infty}$
given by 
 the asymptotic sum 
$\sum_\ell \psi(\pi(\cR)) \sigma_\ell$ 
or
$\sum_\ell \sigma_\ell \psi(\pi(\cR)) $ 
in the sense of \eqref{eq_def_asymptotic}
where $\cR$ is any positive Rockland operator
and 
 $\psi\in C^\infty(\bR)$ 
any  real-valued function
satisfying 
$\psi\equiv 0$ on a neighbourhood of 0 and 
$\psi\equiv 1$ on $(\Lambda,\infty)$
for some $\Lambda>0$.
We then write \eqref{eq_def_homogeneous_asymptotic}.

\medskip

We denote by $S^{m_0}_{asymp}$ the set of symbols $\sigma\in S^{m_0}$
which admits such an homogeneous expansion.
\end{definition}

The last part of Proposition \ref{prop_homogeneous_asymptotic}
also shows that the first term of an expansion $\sigma \sim \sum_{\ell=0}^\infty \sigma_\ell$
is unique
(hence, proceeding recursively and up to writing zero terms, 
the expansion itself is unique).
This allows us to define the principal part of a symbol:

\begin{definition}
If $\sigma \sim \sum_{\ell=0}^\infty \sigma_\ell$ is in $S^{m_0}_{asymp}$, 
then its first term $\sigma_0$ is called its principal part and we write:
$$
\princ_{m_0}(\sigma):=\sigma_0.
$$
\end{definition}

\begin{ex}
\label{ex_sumcXalpha}
If $\sigma=\sum_\alpha c_\alpha (x) \pi(X)^\alpha$
where $(c_\alpha)_{\alpha\in \bN^n}$ is a sequence of functions in $C^\infty(G)$ 
such that $c_\alpha$ and all the left derivatives $X^\beta c_\alpha$ are bounded
while all but a finite number of $c_\alpha$  are zero, 
then $\sigma\in S^m_{asymp}$  
where  $m=\max\{[\alpha], c_\alpha\not=0\}$
and
$$
\sigma = \sum_{\ell=0}^m \sigma_{m-\ell}
\quad\mbox{with}\quad
\sigma_{m-\ell}=
\sum_{[\alpha]=m-\ell} c_\alpha (x) \pi(X)^\alpha \in S^\ell \cap \dot S^\ell.
$$
Moreover the principal part coincides with the top part of the left-invariant differential operator:
$$
\princ_{m}(\sigma)=\sum_{[\alpha]=m} c_\alpha (x) \pi(X)^\alpha.
$$
\end{ex}

The asymptotic expansion and the principal part satisfy the analogue properties 
to its Euclidean counterpart:

\begin{proposition}
The set $S^m_{asymp}$ is a linear subspace of $S^m$
and the mapping
$\princ_m:S^m_{asymp}\to \dot S^m$ is linear.
Moreover if $\sigma\in S^m_{asymp}$
and $\sigma'\in S^{m'}_{asymp}$ 
with asymptotic expansion 
$\sigma\sim \sum_\ell \sigma_\ell$
and 
$\sigma'\sim \sum_{\ell'} \sigma'_{\ell'}$
then
$\sigma^*\in S^m_{asymp}$
and $\sigma\sigma'\in S^{m+m'}_{asymp}$
with asymptotic expansions
$$
\sigma^*\sim \sum_\ell \sigma_\ell^*
\qquad\mbox{and}\qquad
\sigma \sigma'\sim \sum_{\ell,\ell'} \sigma_\ell\sigma'_{\ell'}.
$$
In particular,
$$
\princ_m\left(\sigma^*\right)
=
\princ_m(\sigma)^*
\qquad\mbox{and}\qquad
\princ_{m+m'}\left(\sigma \sigma'\right)
=
\princ_m(\sigma)\princ_{m'}(\sigma').
$$
\end{proposition}

\begin{proof}
The linearity of $S^m_{asymp}$ and of $\princ_m$ is easy to check.
The property regarding the adjoint follows from Theorem~\ref{thm_calculus}~(3).
Let $\sigma$ and $\sigma'$ be as in the statement.
We fix a  positive Rockland operator $\cR$
and a real-valued function
 $\psi\in C^\infty(\bR)$ 
satisfying 
$\psi\equiv 0$ on a neighbourhood of 0 and 
$\psi\equiv 1$ on $(\Lambda,\infty)$
for some $\Lambda>0$.
Then 
we can develop
$$
\left(\sum_{\ell\leq M}  \sigma_\ell \psi(\pi(\cR))\right)
\left(\sum_{\ell'\leq M'} \psi(\pi(\cR))\sigma'_{\ell'} \right)
$$
on  one hand as (see Proposition \ref{prop_homogeneous_asymptotic} (1))
$$
\left(\sigma\ \mbox{mod} \ S^{m_M}\right)
\left(\sigma' \ \mbox{mod} \ S^{m'_M}\right)
=
\sigma\sigma' 
\ \mbox{mod} \ S^{\tilde M}
$$
where $\tilde M:= {\max(m+m'_{M'}, m'+m_M,m_M,m'_{M'})}$,
and on the other hand by Corollary \ref{cor_prop_equiv_dotS}
$$
\sum_{\substack{\ell\leq M\\\ell'\leq M'}}  
\sigma_\ell \psi^2(\pi(\cR)) \sigma'_{\ell'}
=
\sum_{\substack{\ell\leq M\\\ell'\leq M'}}  
 \psi^2(\pi(\cR))\sigma_\ell \sigma'_{\ell'}
\ \mbox{mod} \ S^{-\infty}
$$
Hence
$\displaystyle{
\sigma\sigma' 
=
\sum_{\substack{\ell\leq M,\ell'\leq M'\\ 
m_\ell + m_\ell'\geq \tilde M}}
 \psi^2(\pi(\cR))\sigma_\ell \sigma'_{\ell'}
\ \mbox{mod} \ S^{\tilde M}.
}$\\
This implies
$\displaystyle{
\sigma\sigma' 
\sim 
\sum_{\tilde \ell} \tilde\sigma_{\tilde\ell}}$
where $\displaystyle{ 
\tilde\sigma_{\tilde\ell}:=
\sum_{m_\ell+m_{\ell'}=\tilde m_{\tilde \ell}}
\sigma_\ell \sigma'_{\ell'}\in S^{\tilde m_{\tilde \ell}},
}$
and in particular 
$\tilde \sigma_0=\sigma_0\sigma'_0$.
\end{proof}

\subsection{The classical calculus $\cup_m\Psi^m_{cl}(\Omega)$}

We can now define the classical  classes of symbols and of operators.
\begin{definition}
\label{def_Psi_comp}
Let $\Omega\subset G$ be an open subset.
We denote by $S^m_{cl}(\Omega)$
the class of symbol $\sigma\in S^m_{asymp}$ 
such that the integral kernel of $\Op(\sigma)$ is compactly supported in $\Omega\times \Omega$.
The corresponding class of operators is denoted by
$$
\Psi^m_{cl}(\Omega):=\Op (S^m_{cl}(\Omega)).
$$
The operation of taking the principal part is denoted by $\princ_m$:
$$
\princ_m (\Op(\sigma))=\Op(\princ_m(\sigma)), \quad \sigma\in S^m_{cl}(\Omega).
$$
\end{definition}
If $\Omega=G$ we may allow ourselves the shortcuts
$S^m_{cl}(G)=S^m_{cl}$ and $\Psi^m_{cl}(G)=\Psi^m_{cl}$.
Naturally the differential operators in the calculus
with support in $\Omega$ are classical:

\begin{ex}
If $(c_\alpha)_{\alpha\in \bN^n}$ is a sequence of functions in $\cD(\Omega)$ 
and all but a finite number of $c_\alpha$ are zero, 
then  $\sum_\alpha c_\alpha (x) X^\alpha$ is in $\Psi^m_{cl}(\Omega)$
where $m=\max\{[\alpha], c_\alpha\not=0\}$.
Indeed the (right convolution) kernel is 
$\sum_\alpha c_\alpha (x) \delta_0^{(\alpha)}$
which is supported in $\{0\}$.
Moreover
$$
\princ_m \left(\sum_\alpha c_\alpha (x) X^\alpha\right)
=
\sum_{[\alpha]=m} c_\alpha (x) X^\alpha.
$$
\end{ex}

We will often use the following easy lemma without referring to it.
\begin{lemma}
\label{lem_easy}
Let $\Omega\subset G$ be an open subset.
If $A\in \Psi^0_{cl}(\Omega)$ then  
the operator $A$ extends uniquely into a 
continuous mapping $L^2(\Omega,loc) \to L^2(\Omega)$.
\end{lemma}

As is customary, $L^2(\Omega,loc)$ denotes the space of distribution $f\in \cD'(\Omega)$ such that for all $\chi\in \cD(\Omega)$, $f\chi\in L^2(\Omega)$. 
Later on, we will need the more general definition:
\begin{definition}
\label{def_L2sloc}
Let $\Omega\subset G$ be an open subset.
We denote by $L^2_s(\Omega,loc)$ the space of distribution $f\in \cS'(\Omega)$ such that for all $\chi\in \cD(\Omega)$, $f\chi\in L^2_s(G)$.
It is endowed with its natural structure of  Fr\'echet space.
\end{definition}

\begin{proof}[Proof of Lemma \ref{lem_easy}]
Let $A\in \Psi^0_{cl}(\Omega)$.
Its integral kernel is supported in a compact $K\subset \Omega$.
We can always find $\chi\in \cD(\Omega)$ such that $\chi\equiv 1$ on $K$.
Hence if $\phi\in \cD(\Omega)$, 
then  $A\phi=A (\chi \phi)$.

Let $f\in L^2(\Omega,loc)$.
Then $f\chi\in L^2(\Omega)\subset L^2(G)$ 
and we define $A f:= A(f\chi)$,
as $A\in \Psi^0$, it is bounded on $L^2(G)$ (see Theorem \ref{thm_calculus}).
It is easy to show that this does not depend on the choice of $\chi$
and that we have:
$$
\forall f\in L^2(\Omega,loc) \qquad 
\|Af\|_{L^2(G)} \leq \|A\|_{\sL(L^2(G))} \|f\chi\|_{L^2(G)}.
$$
Since $f\in L^2(\Omega,loc) \mapsto f\chi \in L^2(G)$ is continuous, 
the operator $A:L^2(\Omega,loc) \to L^2(\Omega)$ is continuous.
\end{proof}

We now state and prove a theorem which in the Euclidean setting
 is a consequence of Rellich's theorem.

\begin{theorem}
\label{thm_cq_rellich}
Let $\Omega\subset G$ be an open subset.
If $A\in\Psi^m_{cl}(\Omega)$ with $m<0$
then the operator
$$A:L^2(\Omega,loc) \to L^2(\Omega)$$ is compact,
i.e. if $u_k \TendWeak{k}{\infty} u$ in $L^2(\Omega,loc)$
then $A u_{k(j)} \Tend{j}{\infty } A u$ in the $L^2$-norm
 after extraction of a subsequence $k(j)$.
\end{theorem}

The notation $u_k\TendWeak{k}{\infty}u $ in $L^2(\Omega,loc)$ 
means that 
 the sequence  $(u_k)$ of distributions in $L^2(\Omega,loc)$
converge towards $u$ for the Fr\'echet topology of  $L^2(\Omega,loc)$.
Consequently, $u_k\TendWeak{k}{\infty} u $ in $L^2(\Omega,loc)$ 
if and only if for every $v\in L^2(\Omega)$ compactly supported, 
$(u_k,v)_{L^2}\Tend{k}{\infty} (u,v)_{L^2}$.

Let us recall that Rellich's Theorem states that if  $t<s$ and $K\subset \bR^n$ a compact subset, 
then the inclusion map $H^s(K) \to H^t$ is compact.

\begin{proof}[Proof of Theorem \ref{thm_cq_rellich}]
As $A\in \Psi^m_{cl}(\Omega)$, 
its integral kernel is supported in a compact $K\subset \Omega\times\Omega$ that we can assume of the form $K=K_1\times K_2$.
We can always find $\chi\in \cD(\Omega)$ such that $\chi\equiv 1$ on~$K_2$.
As the integral kernel of $A$ is supported in $K$, 
we have
$A \phi=A(\chi \phi)$, for any $\phi\in L^2(\Omega,loc)$.
Let $\cR$ be a (fixed) positive Rockland operator of homogeneous degree $\nu$. 
As $A \in \Psi^m$,
$A (\id+\cR)^{-m/\nu}\in \Psi^0$ is bounded on $L^2(G)$.
Let $(u_k)$ be a sequence in $L^2(\Omega,loc)$ with $u_k \TendWeak{k}{\infty} u$.
We have 
\begin{eqnarray*}
\|Au_k -Au\|_{L^2}
&=&
\|A(\chi u_k -\chi u)\|_{L^2}
=
\|A(\id+\cR)^{-\frac m \nu}(\id+\cR)^{\frac m \nu}
(\chi u_k -\chi u)\|_{L^2}
\\
&\leq&
\|A(\id+\cR)^{-\frac m \nu}\|_{\sL(L^2(G))}
\|(\id+\cR)^{\frac m \nu}(\chi u_k -\chi u)\|_{L^2}.
\end{eqnarray*}
By Proposition \ref{prop_sob_loc},
$$
\|(\id+\cR)^{\frac m \nu}(\chi u_k -\chi u)\|_{L^2}
=\|\chi (u_k - u)\|_{L^2_m(G)}
\leq C_{m}
\| \chi \ (u_k - u)\|_{H^{m/\upsilon_1}}
.
$$
As $u_k \TendWeak{k}{\infty} u$ in $L^2(\Omega,loc)$
and $m/\upsilon_1<0$, by Rellich's Theorem,
 we can extract a subsequence $\chi u_{k(j)} \Tend{j}{\infty} \chi u$ 
 converging  in the Sobolev norm of $H^{m/\upsilon_1}$.  
Therefore
$$
\|Au_{k(j)} -Au\|_{L^2} \leq \|A(\id+\cR)^{-\frac m \nu}\|_{\sL(L^2(G))}C_{m}
\| \chi \ (u_{k(j)} - u)\|_{H^{m/\upsilon_1}} \Tend{j}{\infty}
  0,
$$
and $A u_{k(j)} \Tend{j}{\infty} A u$ in the $L^2$-norm.
\end{proof}

\section{$C^*$-algebras generated by 0-homogeneous regular symbols} 
\label{sec_0homsymb}

In this section, 
we study the regular 0-homogeneous symbols,
that is, the symbols in $\dot S^0$, 
and the $C^*$-algebra it generates.
A particular attention will be given to the `invariant' symbols in $\dot S^0$, 
that is, those that do not depend on $x$. 

\subsection{The Fr\'echet space $\tilde S^0$}
\label{subsec_inv_dotS0}

In this section, we study the invariant regular 0-homogeneous symbols, 
or in other words the symbol  in $\dot S^0$ independent of $x$.
They form the space~$\tilde S^0$.

\begin{definition}
We denote by $\tilde S^0$ the set of symbols
$\sigma=\{\sigma(\pi):\cH_\pi^\infty \to \cH_\pi^\infty,\pi\in \Gh\}$ satisfying
\begin{enumerate}
\item $\sigma$ is 0-homogeneous, 
i.e. $\sigma(r\cdot\pi)=\sigma(\pi)$ for all $r>0$, $\pi\in \Gh$, 
\item if $\cR$ is a positive Rockland operator of degree $\nu$ and $\alpha\in \bN_0^n$ and $\gamma\in \bR$, then 
$$
\sup_{\pi\in \Gh}
\| \pi(\cR)^{\frac{[\alpha]-m+\gamma}{\nu}}
 \Delta^\alpha \sigma(\pi) \pi(\cR)^{-\frac{\gamma}{\nu}}\| <\infty
$$
\end{enumerate}
\end{definition}

Naturally, the second condition is independent of $\cR$
and it suffices to show it for a sequence $(\gamma_\ell)_{\ell\in \bZ}$ with 
$\lim_{\ell\to \pm \infty}=\pm\infty$.
This equips naturally the vector space $\tilde S^0$ with a Fr\'echet topology which is the same as the one obtained with viewing $\tilde S^0$ as a closed sub-vector space of $\dot S^0$.
Note that $\tilde S^0$ is also an algebra, in fact a sub-algebra of $\dot S^0$. 

By Corollary \ref{cor_prop_equiv_dotS_gamma}, 
a 0-homogeneous symbol $\sigma=\{\sigma(\pi)\}$  is in $\tilde S^m$ if and only if for each $\alpha\in \bN_0^n$, the following
suprema are finite
\begin{equation}
\label{eq_sup_lem_type}
 \sup_{\pi\in \Gh} \|  \pi(\cR)^{\frac{[\alpha]}\nu}
 \Delta^\alpha \sigma(\pi)\|_{\sL(\cH_\pi)},
 \quad\mbox{and}\quad
 \sup_{\pi\in \Gh} \|   \Delta^\alpha \sigma(\pi)\
 \pi(\cR)^{\frac{[\alpha]}\nu}\|_{\sL(\cH_\pi)}.
\end{equation}
Here $\cR$ is a fixed positive Rockland operator of degree $\nu$.

\medskip

Proposition \ref{prop_hom_kernel_smooth} implies that the kernel associated to a symbol $\sigma$ in $\tilde S^0$, i.e. $\kappa\in \cS'(G)$ such that $\widehat \kappa=\sigma$, is smooth on $ G\setminus\{0\}$.
Lemma \ref{lem_type} below shows the converse in the following way:

\begin{lemma}
\label{lem_type}
 Let $\sigma=\{\sigma(\pi)\in \sL(\cH_\pi), \pi\in \Gh\}$ be a measurable field of operators  such that 
\begin{itemize}
\item $\sigma$ is 0-homogeneous, 
i.e. $\sigma(r\pi)=\sigma(\pi)$ for (almost) all $\pi\in \Gh$ and all $r>0$,
\item $\sigma$ is bounded, 
i.e. $\sup_{\pi\in \Gh} \|\sigma(\pi)\|_{\sL(\cH_\pi)}<\infty$,
\item the kernel associated with $\sigma$ coincides with a smooth function on  $G\backslash\{0\}$.
\end{itemize}
Then $\sigma\in \tilde S^0$.
\end{lemma}

Note that the proof of Lemma \ref{lem_type}
given below
does not produce any bounds for the suprema in \eqref{eq_sup_lem_type} in terms of $\kappa$ or $\sigma$.
The main ingredient is  the analysis of operator of type $\nu$, 
see  Section \ref{subsec_optypenu}.

\begin{proof}[Proof of Lemma \ref{lem_type}]
If $\kappa_1$ is any tempered distribution, then $T_{\kappa_1}$ denotes the convolution operator with right-convolution kernel $\kappa_1$, i.e. $T_{\kappa_1}(\phi)=\phi*\kappa_1$, $\phi\in \cS(G)$.
Recall that $X_1,\ldots, X_n$ is a basis of $\fg$.

Let $\sigma,\kappa$ satisfying the hypotheses of the statement. 
We fix  $\alpha\in \bN_0^n$, $\alpha\not=0$.
By Lemma \ref{lem_2R_fractional_power}, 
we may replace $\cR$ by one of its power
and thus
we assume that $\nu >[\alpha]$. 
The operator $\cR$ is a linear combination of $X^\beta$ 
with $[\beta]=\nu$.
Let us write one $X^\beta$ as $Y_r\ldots Y_1$
with $Y_j\in \{X_1,\ldots,X_n\}$ for $j=1,\ldots,r$.
We also denote by $[Y_j]$ the homogeneous degree of $Y_j$,
so that $\nu=[\beta]=[Y_1]+\ldots+[Y_r]>[\alpha]$.
Let $r'\in \bN_0$, $0\leq r'\leq r$, be such that 
$[\alpha]-([Y_1]+\ldots+[Y_{r'}])>0$ 
but
$[\alpha]-([Y_1]+\ldots+[Y_{r'+1}])\leq 0$,
with the convention that $[Y_1]+\ldots+[Y_{r'}]=0$ if $r'=0$ and in this case $Y_{r'}\ldots Y_1=\id$.
By Proposition \ref{prop_op_type} \eqref{item_prop_op_type_conv},
 the operator $Y_{r'}\ldots Y_1 T_{x^\alpha \kappa}= T_{Y_1\ldots Y_r x^\alpha \kappa}$ is of type $[\alpha]-([Y_1]+\ldots+[Y_{r'}])\in (0,Q)$.
As the operator $\cR^{\frac{[\alpha]-([Y_1]+\ldots+[Y_{r'+1}])}\nu}$
is of type $[Y_1]+\ldots+[Y_{r'+1}] -[\alpha] \in (0,Q)$, see Example \ref{ex_powercR_optypenu}, 
the operator
$Y_{r'}\ldots Y_1 T_{x^\alpha \kappa} \cR^{\frac{[\alpha]-([Y_1]+\ldots+[Y_{r'+1}])}\nu}$
is of type $[Y_{r'+1}]$.
Thus
the operator
$Y_{r'+1}\ldots Y_1 T_{x^\alpha \kappa} \cR^{\frac{[\alpha]-([Y_1]+\ldots+[Y_{r'+1}])}\nu}$
is of type 0.
Then 
$Y_{r'+1}\ldots Y_1 T_{x^\alpha \kappa} \cR^{\frac{[\alpha]-([Y_1]+\ldots+[Y_{r'+2}])}\nu}$
is of type $[Y_{r'+2}]$ 
and 
$Y_{r'+2}\ldots Y_1 T_{x^\alpha \kappa} \cR^{\frac{[\alpha]-([Y_1]+\ldots+[Y_{r'+2}])}\nu}$ is of type 0.
Proceeding recursively, 
we obtain that 
$$
X^\beta T_{x^\alpha \kappa} \cR^{-\frac{\nu -[\alpha]}\nu}
=
Y_r\ldots Y_1T_{x^\alpha \kappa}\cR^{-\frac{[Y_1]+\ldots+[Y_r]-[\alpha]}\nu}
$$
is of type 0.
Thus $\cR T_{x^\alpha \kappa} \cR^{-\frac{\nu -[\alpha]}\nu}$
is bounded on $L^2(G)$.
We can apply the same reasoning to $T_{x^\alpha \kappa} ^* =(-1)^{|\alpha|} T_{x^\alpha \kappa^*}$ where $\kappa^*(x)=\bar \kappa(x^{-1})$.
This shows that $T_{x^\alpha \kappa}$ and its adjoint $T_{x^\alpha \kappa}^*$ map $\dot L^2_{\nu -[\alpha]}\to \dot L^2_\nu$ continuously. 
By duality and interpolation
(see Theorem \ref{thm_sobolev_spaces}), we obtain that $T_{x^\alpha \kappa}$ maps 
$\dot L^2\to \dot L^2_{[\alpha]}$
and 
$\dot L^2_{-[\alpha]}\to \dot L^2$ continuously.
The Plancherel theorem, see Section \ref{subsec_Gh+plancherel},
 now implies  that 
the suprema in \eqref{eq_sup_lem_type}
are finite.
This concludes the proof of Lemma \ref{lem_type}.
\end{proof}

We can now describe the symbols in $\tilde S^0$ via their kernels:

\begin{corollary}
\label{cor_pv+type}
We denote by  $\cF$
the Fr\'echet vector space of smooth functions on $G\backslash\{0\}$ which are $(-Q)$-homogeneous and with zero mean value.
If $\sigma\in \tilde S^0$ we denote by $\kappa_\sigma \in \cF$ the smooth function obtained by restriction of the associated symbol to $G\backslash\{0\}$ and $c_\sigma\in \bC$ the number defined in Lemma~\ref{lem_pv}.
The map 
$$
\Theta:
\left\{\begin{array}{rcl}
\tilde S^0&\longrightarrow & \cF \times \bC
\\
\sigma &\longmapsto & (\kappa_{\sigma},c_{\sigma})
\end{array}
\right.
$$
is an isomorphism of  Fr\'echet vector spaces. 

Consequently, 
the Fr\'echet space $\tilde S^0$ is separable.
\end{corollary}

\begin{proof}[Proof of Corollary \ref{cor_pv+type}]
The fact that the map $\Theta$ is well defined, linear, continuous, and injective
 follows easily from Proposition \ref{prop_hom_kernel_smooth} and Lemma \ref{lem_pv}.
Let us show that $\Theta$ maps $\tilde S^0$ onto $\cF\times \bC$.
Given $(\kappa,c)\in \cF \times \bC$, 
we want to construct $\sigma\in \dot S^0$ such that $\Theta(\sigma)=(\kappa,c)$.
Defining $\kappa_{j}$ as in Lemma \ref{lem_pv}, 
we then set $\sigma(\pi):=\sum_{j\in \bZ} \widehat\kappa_{j} + c_{\sigma}\id_{\cH_\pi}$.
The proof of Lemma \ref{lem_pv} shows that this defines a field of operators $\{\sigma(\pi), \pi\in \cH_\pi\}$ which is bounded by:
 $$
 \sup_{\pi\in \Gh} \|  \sigma(\pi)\|_{\sL(\cH_\pi)} \leq 
 |c_{\sigma}| +
 C \sup_{|z|=1,|\alpha|\leq 1} |X^\alpha \kappa(z)|
 $$
 One checks easily that $\sigma$ is 0-homogeneous
 and that the kernel associated with $\sigma$ coincides with $\kappa$ on $G\backslash\{0\}$ and that $c_\sigma=c$.
 By Lemma \ref{lem_type},  $\sigma\in \tilde S^0$.
 Thus $\Theta$ is surjective.
 As the map $\Theta$ is a linear and continuous bijection between Fr\'echet spaces, 
it is an isomorphism by  the open mapping theorem.

Note that  $\cF$ is a closed subspace of 
the Fr\'echet space of smooth functions on $G\backslash\{0\}$ 
which are $(-Q)$-homogeneous.
The latter is isomorphic to the Fr\'echet space of smooth function on the unit sphere given by a smooth quasi-norm, and this latter Fr\'echet space is well-known to be separable.
Note that by a smooth quasi-norm, we mean a quasi-norm which is smooth away from 0; such a quasi-norm exists.
Hence $\cF$ is separable. As $\Theta$ is an isomorphism of Fr\'echet space, $\tilde S^0$ is also separable.
\end{proof}

\subsection{An important example}

This section is devoted to a more concrete example of a symbol in~$\tilde S^0$, more precisely to the symbol $\sigma_f$ defined and studied in 
Lemma \ref{lem_sigmaf}.
It will be useful later.

\begin{lemma}
\label{lem_sigmaf}
We fix a quasi-norm $|\cdot|$ on $G$.
Let $|\cdot|$ be the associated mapping on $\Gh$, 
see Section~\ref{subsec_polar_dec_Gh}. 
 For any $f\in \cS(G)$, $\tilde S^0$ contains 
 the symbol $\sigma_f$ defined via 
$$
\sigma_f(\pi)=\widehat f(|\pi|^{-1}\cdot \pi), \quad \pi\in \Gh\backslash\{1\}.
$$
\end{lemma}

\begin{proof}[Strategy of the proof of Lemma \ref{lem_sigmaf}]
Since $f\in \cS(G)$, 
for each $\pi\in \Gh$, 
the operator $\widehat f(\pi)$ maps $\cH_\pi^\infty$ to itself
and has operator norm $\leq \|f\|_{L^1}$.
Moreover the field of operators $\widehat f$ is a measurable.
Since the map $\pi\mapsto |\pi|^{-1}\cdot \pi$ is continuous 
$\Gh\backslash\{1\}\to \Gh$, 
see  Section \ref{subsec_polar_dec_Gh}, 
we have $\sigma_f\in L^\infty(\Gh)$ with $\sup_{\pi\in \Gh} \|\sigma_f(\pi)\|_{\sL(\cH_\pi)}
\leq \|f\|_{L^1}$.
One checks easily that $\sigma_f$ is 0-homogeneous.

We denote by $\kappa\in \cS'(G)$ the kernel of $\sigma_f\in L^\infty(\Gh)$, 
i.e. $\sigma_f=\widehat\kappa$.
Let $\phi\in \cS(G)$.
By Lemma~\ref{lem_type}, it suffices to show that its kernel $\kappa \in \cS'(G)$ is smooth away from 0.
And for this, 
it suffices to show that for every $M\in \bN_0$
there exist $N\in \bN_0$ and $C=C(\sigma,M)>0$ 
such that for any  $\phi\in \cD(G)$ satisfying $|x|^{-N} \phi  \in \cD(G)$, we have:
\begin{equation}
\label{eq_pf_lem_sigmaf}
|( \kappa, \cR^{M} \phi)|
\leq C \left(\|\phi\|_{L^1(G)} +\|\phi\|_{L^2(G)} + \| |x|^{-N} \phi \|_{L^2(G)}\right).
\end{equation}
Indeed \eqref{eq_pf_lem_sigmaf} will imply that $X^\alpha\kappa$ is locally square integrable on $G\backslash\{0\}$ for any $\alpha\in \bN_0^n$, and thus that it is smooth away from 0.
\end{proof}

\begin{proof}[Proof of \eqref{eq_pf_lem_sigmaf} in the case $M=0$]
We fix  $\chi\in \cD(\bR)$ such that $0\leq \chi\leq 1$, $\chi\equiv 1$ on $[-1,1]$ and $\chi(\lambda)\equiv 0$ for $|\lambda|\geq 2$.
Let $\phi\in \cD(G)$.
Since $\chi(\cR)\delta_0 \in \cS(G)$, 
see Corollary \ref{cor_prop_multipliers_symbol}, 
 $\chi(\cR)\phi=\phi* \chi(\cR)\delta_0$ is Schwartz 
and so is $(1-\chi)(\cR)\phi$.
As $\kappa\in \cS'(G)$, we can write:
$$
\langle\kappa,\phi\rangle= 
\langle\kappa,\chi(\cR)\phi\rangle+
\langle\kappa,(1-\chi)(\cR)\phi\rangle.
$$
Note that if $\psi\in \cS(G)$, 
then $\sigma_f \widehat \psi \pi(\id+\cR)^N \in L^\infty(\Gh)$ for any $N\in \bN$, 
and $\sigma_f \widehat \psi$
satisfies the hypotheses of
Proposition \ref{prop_FIF}.
So the Fourier inversion formula yields
$$
\langle \kappa, \psi\rangle
= 
\int_{\Gh} \tr \left(\sigma_f(\pi) \cF(\check\psi)(\pi)\right)
d\mu(\pi),
$$
where $\check \psi(x)=\psi(x^{-1})$, 
and
$$
|\langle \kappa, \psi\rangle|
\leq \|\sigma_f(\pi)\|_{L^\infty(\Gh)}
\int_{\Gh} \tr  |\cF(\check\psi)(\pi)|
d\mu(\pi)
\leq \|f\|_{L^1(G)}
\int_{\Gh} \tr  |\cF(\psi)(\pi)|
d\mu(\pi).
$$
Applying this to $\psi=\chi(\cR)\phi\in \cS(G)$
implies that 
$$
|\langle\kappa,\chi(\cR)\phi\rangle|
\leq
\|f\|_{L^1(G)}
\|\phi\|_{L^1(G)}
\int_{\Gh} \tr |\chi(\pi(\cR))| d\mu(\pi),
$$
and the last integral is finite, 
see Corollary \ref{cor_prop_multipliers_symbol}.

We now turn our attention to
$\langle\kappa,(1-\chi)(\cR)\phi\rangle$.
For the same reason as above, 
$$
\int_{\pi\in \Gh}
\tr \left| \sigma_f(\pi) 
(1-\chi)(\pi(\bar\cR))
\cF \check \phi (\pi)\right| d\mu(\pi)
<\infty
$$
 \begin{eqnarray*}
 {\rm and}\quad
\langle\kappa,(1-\chi)(\cR)\phi\rangle
&=&
\int_{\pi\in \Gh}
\tr \left\{ \sigma_f(\pi) 
(1-\chi)(\pi(\bar\cR))
\cF \check \phi (\pi)\right\}d\mu(\pi)
\\
&=&
\int_{r=0}^\infty \int_{\pi\in \Sigma_1}
\tr \left\{ \sigma_f(r\cdot\pi) (1-\chi)(r\cdot \pi(\bar \cR))
\cF \check \phi (r\cdot \pi)\right\}d\varsigma(\pi) r^{Q-1} dr,
\end{eqnarray*}
having used the polar decomposition
 on $\Gh$, 
see Lemma \ref{lem_pol_dec_Gh}.
We now write:
$$
\int_{r=0}^\infty = \int_{r=0}^1 + \int_{r=1}^\infty
=I_1+I_2.
$$
We have
$r\cdot \pi(\cR)=r^\nu \pi(\cR)
\quad\mbox{and}\quad 
\sigma_f(r\cdot\pi) = \widehat f(\pi),\ \pi\in \Sigma_1,
$
so fixing $N_0\in \bN$
$$
\sigma_f(r\cdot\pi) (1-\chi)(r\cdot \pi(\bar \cR))
=
r^{\nu N_0}\widehat{ {\tilde {\bar \cR}}^{N_0} f}(\pi) 
\chi_{N_0}(r^\nu  \pi(\bar \cR))
$$
where $\chi_{N_0}(\lambda)=(1-\chi(\lambda))\lambda^{-N_0}$,
and
\begin{align*}
|I_1|
&\leq 
\int_{r=0}^1 \int_{\pi\in \Sigma_1}
\tr \left| \widehat{ {\tilde {\bar \cR}}^{N_0} f}(\pi) 
\chi_{N_0}(r^\nu  \pi(\bar \cR))
\cF \check \phi (r\cdot \pi)\right |d\varsigma(\pi) r^{N_0\nu +Q-1} dr
\\
&\leq 
\|\widehat{ {\tilde {\bar \cR}}^{N_0} f} \|_{L^\infty(\Gh)}
\int_{r=0}^\infty \int_{\pi\in \Sigma_1}
\|\chi_{N_0}(r^\nu  \pi(\bar \cR))\|_{HS(\cH_\pi)}
\|\cF \check \phi (r\cdot \pi)\|_{HS(\cH_\pi)}
d\varsigma(\pi)  r^{Q-1} dr
\\
&\leq \|\tilde {\bar \cR}^{N_0} f\|_{L^1}
\|\chi_{N_0}(\pi(\cR))\|_{L^2(\Gh)}
\|\widehat\phi\|_{L^2(\Gh)},
\end{align*}
by the Cauchy-Schwartz inequality.
The Plancherel formula yields 
$\|\widehat\phi\|_{L^2(\Gh)} = \|\phi\|_{L^2(G)}$
and, together with the functional calculus of $\cR$,
$$
\|\chi_{N_0}(\pi(\cR))\|_{L^2(\Gh)}
=
\|\chi_{N_0}(\cR)\|_{L^2(G)}
\leq \sup_{\lambda\geq 0} (1-\chi(\lambda))
\left(\frac{1+\lambda}{\lambda}\right)^{N_0} \
\|(\id+\cR)^{-N_0} \delta_0 \|_{L^2(G)}.
$$
This last quantity is finite when $N_0\nu>Q/2$
since $(\id+\cR)^{\frac {s_1}\nu}\delta_0\in L^2(G)$ for $s_1<-Q/2$,
 see \cite[\S 4.3.3]{R+F_monograph}.

For the second integral, we see
$$
I_2 
=
\int_{r=1}^\infty \int_{\pi\in \Sigma_1}
\tr \left\{ \widehat f (\pi) (1-\chi)(r^\nu  \pi(\bar \cR))
\cF \check \phi (r\cdot \pi)\right\}d\varsigma(\pi) r^{Q-1} dr.
$$
For each $r>0$, 
we define
$g_r := (1-\chi) (r^\nu\tilde \cR) f \in L^2(G)$
so that $\widehat g_r(\pi)= \widehat f(\pi)  (1-\chi)(r^\nu  \pi(\bar \cR))$
also defines an operator for each $\pi\in \Gh$ with $\widehat g_r\in L^\infty(\Gh)\cap L^2(\Gh)$.
We observe that if $\alpha_2\in \bN_0^n$, $\alpha_2\not=0$,
then
$$
\Delta^{\alpha_2}(1-\chi)(\pi(\bar \cR)) 
=
-\Delta^{\alpha_2}\chi(\pi(\bar \cR)) \in \cF_G \cS(G).
$$
This together with the Leibniz formula easily implies that for any $\alpha\in \bN_0^n$, $\Delta^\alpha \widehat g_r (\pi) $ is a well defined bounded operator on $\cH_\pi$ for each $\pi\in \Gh$ and that  $\Delta^\alpha \widehat g_r \in L^\infty(\Gh)$.

Let $N\in \bN_0$.
For any $x\in G\backslash \{0\}$, we set
 $\phi_N(x)=|x|^{-N} \phi(x)$ and assume  $\phi_N\in \cD(G)$.
We may assume that the integer $N$ and the quasi-norm $|\cdot|$ are such that $|\cdot|^N$ is a polynomial in $x$, which is then necessarily homogeneous of degree $N$. One checks readily that
$\phi = |x|^N \phi_N$ and by \eqref{eq_Deltarcdotpi},
$$
\cF \check \phi (r\cdot \pi)
=
r^{-N} \Delta_{|x|^N} \cF(\check \phi_{N})_{(r)} (\pi).
$$
From \eqref{eq_sigmaDeltaphi}, we have
\begin{align*}
&\tr \left\{ \widehat f (\pi) (1-\chi)(r^\nu  \pi(\bar \cR))
\cF \check \phi (r\cdot \pi)\right\}
=
r^{-N} \tr \left\{ \widehat g_r (\pi) 
\Delta_{|x|^N} \cF(\check \phi_{N})_{(r)} (\pi)\right\}
\\
&\qquad=r^{-N} \sum_{[\alpha_1]+[\alpha_2]=N}c_{\alpha_1,\alpha_2}
\tr \Delta^{\alpha_1}\left\{\Delta^{\alpha_2} \widehat g_r (\pi) 
\
 \cF(\check \phi_{N})_{(r)} (\pi)\right\}
\\
&\qquad=\sum_{[\alpha_1]+[\alpha_2]=N}c_{\alpha_1,\alpha_2}
\tr \Delta^{\alpha_1}\left\{
\sigma_{\alpha_2}
\
 \cF(\check \phi_{N})\right\}(r\cdot \pi)
\end{align*}
where
$\displaystyle{
\sigma_{\alpha_2}(r\cdot \pi)
:=r^{-[\alpha_2]}1_{r\geq 1}\Delta^{\alpha_2} \widehat g_r (\pi)}$ ($r>0$, $\pi\in \Sigma_1$)
is in $L^\infty(\Gh)$.
By Lemmata \ref{lem_pol_dec_Gh} and \ref{lem_FIF+IBP},
\begin{align*}
I_2 
&=
\sum_{[\alpha_1]+[\alpha_2]=N}c_{\alpha_1,\alpha_2}
\int_{r=0}^\infty\int_{\pi\in \Sigma_1}
\tr \Delta^{\alpha_1}\left\{
\sigma_{\alpha_2}
\
 \cF(\check \phi_{N})\right\}(r\cdot \pi)
d\varsigma(\pi) r^{Q-1} dr
\\
&=
\sum_{[\alpha_1]+[\alpha_2]=N}c_{\alpha_1,\alpha_2}
\int_{r=0}^\infty\int_{\pi\in \Sigma_1}
\tr \Delta^{\alpha_1}\left\{
\sigma_{\alpha_2}
\
 \cF(\check \phi_{N})\right\}(\pi)\
d\mu(\pi) 
\\
&=
\sum_{[\alpha_2]=N}c_{0,\alpha_2}
\int_{\Gh}
\tr \left\{
\sigma_{\alpha_2}
\
 \cF(\check \phi_{N})\right\}(\pi)
d\mu(\pi) .
\end{align*}
So we have obtained:
$$
|I_2 |
\lesssim \sum_{[\alpha_2]=N}
\int_{\Gh}
\tr \left|
\sigma_{\alpha_2}
\
 \cF(\check \phi_{N})\right|(\pi) \
d\mu(\pi) 
\lesssim \sum_{[\alpha_2]=N}
\|\sigma_{\alpha_2}\|_{L^2(\Gh)}
\|\cF(\check \phi_{N})\|_{L^2(\Gh)}.
$$
By the Plancherel formula, 
$\|\cF(\check \phi_{N})\|_{L^2(\Gh)}=\|\phi_N\|_{L^2(G)}$.
 We have with $\alpha_2\in \bN_0^n$, $[\alpha_2]=N$:
$$
\|\sigma_{\alpha_2}\|_{L^2(\Gh)}^2
=
\int_{r=1}^{\infty}\int_{\pi\in \Sigma_1}
\| \Delta^{\alpha_2} \widehat g_r (\pi) \|_{HS(\cH_\pi)}^2
r^{-2N+Q-1} dr d\varsigma(\pi),
$$
and by the Leibniz formula 
$$
\| \Delta^{\alpha_2} \widehat g_r (\pi) \|_{HS(\cH_\pi)}
\lesssim
\sum_{[\alpha_0]+[\alpha_1]=N}
\| \Delta^{\alpha_0} \widehat f (\pi) \
\Delta^{\alpha_1} (1-\chi)(r^\nu \pi(\cR)) 
 \|_{HS(\cH_\pi)}.
$$
In the sum above, for $\alpha_1\not=0$, 
we have
$$
\| \Delta^{\alpha_0} \widehat f (\pi) \
\Delta^{\alpha_1} (1-\chi)(r^\nu \pi(\cR)) 
 \|_{HS(\cH_\pi)} \leq
 \|  \widehat{x^{\alpha_0} f} (\pi) \|_{L^\infty(\Gh)}
\|\Delta^{\alpha_1} \chi(r^\nu \pi(\cR)) \|_{HS(\cH_\pi)}
$$
whereas for $\alpha_1=0$, we have
$$
\| \Delta^{\alpha_0} \widehat f (\pi) \
(1-\chi)(r^\nu \pi(\cR)) 
 \|_{HS(\cH_\pi)} \leq
 r^{N_0\nu} \|  \cF\{\tilde {\bar \cR}^{N_0}x^{\alpha_0} f\} (\pi) 
\chi_{N_0}(r^\nu \pi(\cR)) \|_{HS(\cH_\pi)}.
$$
When $N_0\nu\leq 2N$,  
these estimates yield
 \begin{align*}
\sum_{[\alpha_2]=N}
\|\sigma_{\alpha_2}\|_{L^2(\Gh)} &\lesssim 
\sum_{[\alpha_0]\leq N}
\|  \widehat{x^{\alpha_0} f} (\pi) \|_{L^\infty(\Gh)}
\sum_{0<[\alpha_1]\leq N}
 \|\Delta^{\alpha_1} \chi( \pi(\cR)) \|_{L^2(\Gh)}
\\ &\qquad +  \sum_{[\alpha_0]= N}
 \|  \cF\{\tilde {\bar \cR}^{N_0}x^{\alpha_0} f\}\|_{L^\infty(\Gh)}
 \|\chi_{N_0}(\pi(\cR))\|_{L^2(\Gh)}
\\
&\lesssim \|(1+|x|)^N f\|_{L^1(G)}
\|(1+|x|)^N \chi(\cR)\delta_0\|_{L^1(G)}
\\ &\qquad +   \sum_{[\alpha_0]= N}
 \| \tilde {\bar \cR}^{N_0}x^{\alpha_0} f\|_{L^1(G)}
\|\chi_{N_0}(\cR)\delta_0\|_{L^2(G)}.
\end{align*}
Recall that $\chi(\cR)\delta_0\in \cS(G)$ and we have already seen that $\chi_{N_0}(\cR)\delta_0 \in L^2(G)$ when $\nu N_0>Q/2$.
This implies that $I_2$ is bounded up to a constant of $f$ by $\|\phi_N\|_{L^2(G)}$.
Hence \eqref{eq_pf_lem_sigmaf} is proved in the case $M=0$. 
\end{proof}

\begin{proof}[Proof of \eqref{eq_pf_lem_sigmaf} for $M\in \bN$]
If $M\in \bN$, then we modify the proof above.
We write 
$$
(\kappa,\cR^M\phi)= 
(\kappa,\cR^M\chi(\cR)\phi)+
(\kappa,\cR^M(1-\chi)(\cR)\phi) .
$$
First the first term, we have
\begin{eqnarray*}
(\kappa,\cR^M\chi(\cR)\phi)
&=&
\int_{\pi\in \Gh}
\tr \left\{ \sigma_f(\pi) \pi(\bar\cR)^M
\chi(\pi(\bar\cR))
\cF \check \phi (\pi)\right\}d\mu(\pi)\\
|(\kappa,\chi(\cR)\phi)|
&\leq&
 \|f\|_{L^1}\|\phi\|_{L^1}
 \int_{\pi\in \Gh}
\tr \left|(\lambda^M\chi)( \pi(\bar\cR))\right|d\mu (\pi),
\end{eqnarray*}
and this last integral is finite since $(\lambda^M\chi)( \pi(\bar\cR))$ is the Fourier transform of a Schwartz function by Hulanicki's theorem.
For the second term,  we have:
\begin{eqnarray*}
&&(\kappa,\cR^M(1-\chi)( \cR)\phi)
\\&&\qquad=
\int_{r=0}^\infty \int_{\pi\in \Sigma_1}
\tr \left\{ \widehat f(\pi)r^{\nu M}\pi(\bar\cR)^M (1-\chi)(r^\nu \pi(\bar \cR))
\cF \check \phi (r\cdot \pi)\right\}d\varsigma(\pi) r^{Q-1} dr,
\end{eqnarray*}
We decompose again $\int_{r=0}^\infty=\int_{r=0}^1+\int_{r=1}^\infty$, 
and  a modification of the argument yields:
$$
\int_{r=0}^1\leq
\|\tilde {\bar \cR}^{N_0} f\|_{L^1}
\|\chi_{N_0+M}(\pi(\cR))\|_{L^2(\Gh)}
\|\widehat\phi\|_{L^2(\Gh)},
$$
and 
$\int_{r=1}^\infty\leq
\|\phi_N\|_{L^2(G)} I'_2$ with 
$$
(I'_2)^2
\lesssim\!\!\!\!\!
\sum_{[\alpha_0]+[\alpha_1]=N}
\int_{\Sigma_1}\int_{1}^\infty
\| \Delta^{\alpha_0} \cF_G\{ \tilde \cR^M f (\pi)\} 
\ \Delta^{\alpha_1}(1-\chi)(r^\nu \pi(\cR)) 
 \|_{HS(\cH_\pi)}^2 r^{\nu M-2N+Q-1} dr d\varsigma(\pi).
$$
This implies \eqref{eq_pf_lem_sigmaf} for $M\in \bN$ and concludes the proof of Lemma \ref{lem_sigmaf} .
\end{proof}

\subsection{The $C^*$-algebra $C^*(\tilde S^0)$ and its spectrum}
\label{subsec_C*tildeS0}

In this section, we study the closure of $\tilde S^0$ for $\sup_{\pi \in \Gh}\|\cdot\|_{\sL(\cH_\pi)}$.
It is denoted by $C^*(\tilde S^0)$. More precisely, we prove:

\begin{proposition}
\label{prop_C*tildeS0}
The closure of $\tilde S^0$ for $\sup_{\pi \in \Gh}\|\cdot\|_{\sL(\cH_\pi)}$ is a  separable $C^*$-algebra denoted by $C^*(\tilde S^0)$. 
It is of type 1.

If $\pi_0\in \Gh$, then the mapping 
$$
\left\{\begin{array}{lll}
\tilde S^0 &\longrightarrow& \sL(\cH_{\pi_0})\\
\sigma&\longmapsto & \sigma(\pi_0)
\end{array}\right. \quad,
$$
extends to a continuous mapping $\rho_{\pi_0}:C^*(\tilde S^0)\to \sL(\cH_{\pi_0})$ which is an irreducible representation of $C^*(\tilde S^0)$. For any $r>0$, we have $\rho_{\pi_0}=\rho_{r\cdot\pi_0}$.
Denoting by $\dot \pi_0$ the class of representations $\{r\cdot\pi_0, r>0\}$, 
the mapping
$$
R:\left\{\begin{array}{lll}
 \Gh/\bR^+ &\longrightarrow&\widehat{C^*(\tilde S^0)} \\
 \dot \pi_0 &\longmapsto& \rho_{\pi_0}
 \end{array}\right.
 $$
is a homeomorphism.
\end{proposition}

Consequently, we may identify the spectrum of $C^*(\tilde S^0)$ with $\Gh/\bR^+$.
Recall that $\Gh$ is the spectrum of the $C^*$-algebra $C^*(G)$ of the group.
Here we view $C^*(G)$ as the completion of $\cS(G)$ or $L^1(G)$ for the norm 
$f\mapsto \sup_{\pi\in \Gh} \|\widehat f(\pi)\|_{\sL(\cH_\pi)}$.

One checks easily that  $C^*(\tilde S^0)$ is a $C^*$-algebra.
Its separability follows easily from Corollary \ref{cor_pv+type}.
The essential point in the proof of Proposition \ref{prop_C*tildeS0} is the following lemma:
\begin{lemma}
\label{lem_prop_C*tildeS0}
Let $\rho$ be a representation of the $C^*$ algebra $C^*(\tilde S^0)$.
For any $f\in \cS(G)$, we set
$$
\pi_\rho(f) = \rho ( {\sigma_f}^*), 
$$
where the symbol  $\sigma_f \in \tilde S^0$ is defined as in Lemma \ref{lem_sigmaf}
(we assume that a quasi-norm on $G$ has been fixed).
Then 
the mapping $\pi_\rho:\cS(G)\to \sL(\cH_\rho)$ 
extends to  a continuous representation $\pi_\rho$ of $C^*(G)$. 

If $\rho$ is (non-zero) irreducible, 
then $\pi_\rho$ is (non-zero) irreducible, i.e. $\pi_\rho\in \Gh$.
Furthermore  for any symbol $\sigma\in \tilde S^0$,
we have  $\rho(\sigma)=\sigma(\pi_\rho)$.
\end{lemma}
\begin{proof}[Proof of Lemma \ref{lem_prop_C*tildeS0}]
We keep the notation of the statement.
One checks easily that for any $f,f_1,f_2\in \cS(G)$, we have:
$$
\|\sigma_f\|_{L^\infty(\Gh)}\leq \|\widehat f\|_{L^\infty(\Gh)},
\quad
\sigma_{f_1*f_2} =\sigma_{f_2}\ \sigma_{f_1},
\quad\mbox{and}\quad
{\sigma_f}^*=\sigma_{f^*}.
$$
Let $\rho$ be a representation of $C^*(\tilde S^0)$.
One checks easily that the mapping $\pi_\rho:\cS(G)\to \sL(\cH_\rho)$ defined via 
$\pi_\rho(f) = \rho ( {\sigma_f}^*)$
extends to  a continuous representation of $C^*(G)$. 

We assume that $\rho$ is irreducible and non-zero.
Let us show that $\pi_\rho$ is irreducible and non-zero.
We will need the following preliminary step.
Let $\chi_1\in \cD(\bR)\backslash\{0\}$ 
be supported in $[1/2,2]$ and valued in $[0,+\infty)$.
We set $\chi_\epsilon(\lambda)=\chi_1(\epsilon\lambda)$
for each $\epsilon>0$, $\lambda\in \bR$.
The properties of $\cR$  implies that 
$\chi_\epsilon (\cR)\delta_0
=(\chi_1 (\cR)\delta_0)_{(\epsilon)}
\in \cS(G)$.
Since 
$$
c:=\int_G \chi_\epsilon(\cR)\delta_0
=
\int_G \chi_1(\cR)\delta_0
=
\|\sqrt{\chi_1}(\cR)\delta_0\|_{L^2(G)}^2
>0, $$
we may replace $\chi_1$ with $c^{-1}\chi_1$, and assume $c=1$.
Let us show that 
\begin{equation}
\label{eq_pf_prop_C*tildeS0_prel}
\pi_\rho(\chi_\epsilon (\cR)\delta_0)
=
\rho (\sigma_{\chi_\epsilon (\cR)\delta_0})
\longrightarrow_{ \epsilon\to0}\id_{\cH_\rho}
\quad\mbox{in SOT on} \ \cH_{\rho}.
\end{equation}
By Lemma \ref{lem_cv_psiepsilon}, 
if $\pi_1$ is a continuous unitary representation of $G$, 
then 
\begin{equation}
\label{eq_pf_prop_C*tildeS0_2}
\widehat\chi_\epsilon(\pi_1(\cR)) = \pi_1\{\chi_\epsilon (\cR)\delta_0\}\longrightarrow_{ \epsilon\to0}\id_{\cH_{\pi_1}}
\quad\mbox{in SOT on} \ \cH_{\pi_1}.
\end{equation}
For any  $v\in \cH_\rho$,
the representation  $\pi_\rho$ restricted to the closure
of the subspace $\pi_\rho(C^*(G))v$
can be identified with a continuous unitary representation of $G$
when $\pi_\rho(C^*(G))v\not=\{0\}$,
see \cite[\S 13.9.3]{Dixmier_C*};
we can then apply \eqref{eq_pf_prop_C*tildeS0_2} to this representation.
By \cite[Proposition 2.2.6]{Dixmier_C*}, 
the space
$\cH_{\pi_\rho}=\cH_\rho$ of the representation $\pi_\rho$
decomposes into $ \cH_{\pi_\rho,0} \oplus^\perp  \cH_{\pi_\rho,0}^\perp$ 
where $\cH_{\pi_\rho,0}$ denotes the closure of the subspace of 
$v\in \cH_\rho$ satisfying $\pi_\rho(C^*(G))v=\{0\}$.
Hence we have obtained that, as $\epsilon\to0$,  
$\pi_\rho(\chi_\epsilon (\cR)\delta_0)
\to \id_{\cH_{\pi_\rho,0}^\perp}$ 
in SOT on $\cH_{\rho} = \cH_{\pi_\rho}$ and on $\cH_{\pi_\rho,0}^\perp$.
If $v\in \cH_{\pi_\rho,0}^\perp$ and $\sigma\in \tilde S^0$, 
then 
$\pi_\rho (\chi_\epsilon(\cR)\delta_0) 
\rho (\sigma) v$ is in $ \cH_{\pi_\rho,0}^\perp$
and converges to $\rho(\sigma)v$
which is necessarily in $ \cH_{\pi_\rho,0}^\perp$.
Thus the closed subspace $ \cH_{\pi_\rho,0}^\perp$
is invariant under $\rho$.
As $\rho$ is irreducible and non-zero, we must have
$\cH_{\pi_\rho,0}^\perp = \cH_\rho$.
Thus we have obtained \eqref{eq_pf_prop_C*tildeS0_prel}.

Let us now show that \eqref{eq_pf_prop_C*tildeS0_prel} implies the irreducibility of $\pi_\rho$.
Let $\tau$ be a symbol in $\tilde S^0$.
For every $\epsilon>0$ and $\pi\in \Gh$, 
we set $\widehat f_\epsilon(\pi):=\chi_\epsilon (\pi(\cR)) \tau(\pi)$.
By Corollary \ref{cor_sigmachi}, $f_\epsilon$ is Schwartz.
We check easily that 
$\sigma_{f_\epsilon}= 
\sigma_{\chi_\epsilon (\cR)\delta_0}\tau$ thus
\begin{equation}
\label{eq_pf_prop_C*tildeS0_1}
\pi_\rho(f_\epsilon)=
\rho(\sigma_{f_\epsilon})= 
\rho (\sigma_{\chi_\epsilon (\cR)\delta_0})
\rho(\tau )
\longrightarrow_{ \epsilon\to0}\rho(\tau )
\quad\mbox{in SOT on} \ \cH_{\rho}.
\end{equation}
This convergence implies easily that any $\pi_\rho$-invariant subspace of $\cH_\rho$
is also   invariant under $\rho$.
Thus the representation $\pi_\rho$ of $C^*(G)$ is irreducible.

We keep the same notation for the corresponding representation (class) $\pi_\rho\in \Gh$ of $G$.
We observe that
$$
\pi_\rho(f_\epsilon)=
\widehat f_\epsilon(\pi_\rho)
=\chi_\epsilon (\pi_\rho(\cR)) \tau(\pi_\rho)
$$
and, for SOT on $\cH_\rho$, 
the left-hand side converges to 
$\rho(\tau)$ by \eqref{eq_pf_prop_C*tildeS0_1}
whereas the right-hand side tends to $\tau(\pi_\rho)$
by \eqref{eq_pf_prop_C*tildeS0_2}.
Hence $\rho(\tau)=\tau(\pi_\rho)$ for any $\tau\in \tilde S^0$.
This concludes the proof of Lemma \ref{lem_prop_C*tildeS0}.
\end{proof}

We can now prove Proposition \ref{prop_C*tildeS0}.
 
  \begin{proof}[End of the proof of Proposition \ref{prop_C*tildeS0}]
We fix $\pi_0\in \Gh$.
By Lemma \ref{lem_pv}, if $\sigma\in \tilde S^0$, we can consider $\sigma(\pi_0) \in \sL(\cH_{\pi_0})$.
One checks readily that $\rho_{\pi_0}:\sigma \mapsto \sigma(\pi_0)$ is a representation of the algebra $\tilde S^0$
which extends to a continuous representation $\rho_{\pi_0}$ of $C^*(\tilde S^0)$.
This defines an injective mapping 
$R:\dot \pi_0\mapsto \rho_{\pi_0}$
which is continuous.
By Lemma \ref{lem_prop_C*tildeS0},
 $R$ is surjective.
As $\Gh/\bR^+$ is compact, see Lemma \ref{lem_Kmap_quotient}, 
$R$ is a homeomorphism. 
 
If $\rho \in \widehat{C^*(\tilde S^0)}$, 
then $\rho (C^*(\tilde S^0))$ contains $\pi_\rho(C^*(G))$ 
having used  Lemma \ref{lem_prop_C*tildeS0} and its notation.
As $C^*(G)$ is of type 1, 
by \cite[Theorem Dixmier 9.1]{Dixmier_C*},
$\pi_\rho(C^*(G))$ contains the space of compact operators in $\sL(\cH_{\pi_\rho})$
thus so does $\rho (C^*(\tilde S^0))$.
Again by \cite[Theorem Dixmier 9.1]{Dixmier_C*}, 
this shows that $C^*(\tilde S^0)$ is type 1.
This concludes the proof of Proposition \ref{prop_C*tildeS0}.
\end{proof}

\subsection{The $C^*$-algebra $C^*(\dot S^0( \Omega))$ and it spectrum}
\label{subsec_C*dotS0barOmega}

We can use the results in Section \ref{subsec_C*tildeS0} on invariant symbols to analyse $x$-dependent $0$-homogeneous regular symbols with support in an open set of $G$.   

\begin{definition}
Let $\Omega$ be an open set of $G$.
We denote by $\dot S^0(\Omega)$ 
the space of symbols 
$\sigma=\{\sigma(x,\pi):\cH_\pi^\infty\to\cH_\pi^\infty, (x,\pi)\in  \Omega\times\Gh\}$
satisfying 
$$
\forall \alpha,\beta\in \bN_0^n, \, \gamma\in \bR\qquad
\sup_{x\in  \Omega, \pi\in \Gh}
\|\pi(\id+\cR)^{\frac{ [\alpha] +\gamma}\nu }
X_x^\beta\Delta^\alpha \sigma(x,\pi) 
\pi(\id+\cR)^{-\frac{\gamma}\nu }\|_{\sL(\cH_\pi)}<\infty,
$$
where  $\cR$ is  a (fixed) positive Rockland operator of homogeneous degree $\nu$.
\end{definition}
This condition does not depend on a particular choice of $\cR$.
The space $\dot S^0(\Omega)$  has a natural topology of Fr\'echet space with semi-norms
$$
\|\sigma\|_{\dot S^0(\Omega), a,b,c}
=
\max_{\substack{[\alpha]\leq a, [\beta]\leq b\\ |\gamma|\leq c}}
\sup_{x\in  \Omega}
\sup_{ \pi\in \Gh}
\|\pi(\id+\cR)^{\frac{ [\alpha] +\gamma}\nu }
X_x^\beta\Delta^\alpha \sigma(x,\pi) 
\pi(\id+\cR)^{-\frac{\gamma}\nu }\|_{\sL(\cH_\pi)}.
$$

\begin{remark}
\label{rem_barOmega}
\begin{enumerate}
\item 
Note that one can view $\dot S^0(\Omega)$
as the set of symbols in $\dot S^0$ 
obtained by restriction to $\Omega$, 
when one identifies two symbols which have the same restriction.

\item 
Moreover if $\Omega$ is bounded, 
a symbol $\sigma$ in $\dot S^0(\Omega)$ 
is in fact defined for $x\in \bar \Omega$
and in this case $\sigma(x,\pi) \in \sL(\cH_\pi)$ for almost every $\pi\in \Gh$.
\end{enumerate}

\end{remark}

As in the invariant case (see Corollary \ref{cor_pv+type}), 
we can describe the elements of $\dot S^0(\Omega)$ 
in terms of their kernel and this implies  that $\dot S^0(\Omega)$ is separable:
\begin{proposition}
\label{prop_dotS0_separability}
Let $\Omega$ be a bounded open subset of $G$ and $\bar \Omega$ its closure.

We now denote by $\cF_0$ the space of $(\kappa,c)\in \cC^\infty(\Omega \times (G\backslash\{0\}))   \times \cC^\infty(\Omega)$
such that the function $(x,y)\mapsto \kappa(x,y)=\kappa_x(y)$
is $(-Q)$-homogeneous and has zero mean value in $y$.
It is a closed subspace of the Fr\'echet space $\cC^\infty(\Omega \times (G\backslash\{0\}))   \times \cC^\infty(\Omega)$.
Moreover $\cF_0$ is separable.

The map
$$
\Theta_0:
\left\{\begin{array}{rcl}
\dot S^0(\Omega)&\longrightarrow & \cF_0
\\
\sigma &\longmapsto & (\kappa_{\sigma},c_{\sigma})
\end{array}
\right.
$$
is an isomorphism of  Fr\'echet vector spaces and 
the Fr\'echet space $\dot S^0(\Omega)$ is separable.
\end{proposition}

Note that above, 
 the Fr\'echet structure of $\cC^\infty(\Omega)$ is given by the semi-norms
 $f\mapsto \sup_{x\in \Omega } |X^\beta_x f(x)|$;
note that $f\in \cC^\infty$ also yields a continuous function on $ \bar \Omega$.
Moreover, 
if $\cF_1$ is a Fr\'echet space,  
 $\cC^\infty (\Omega, \cF_1)$ denotes
 the Fr\'echet vector space of $\cF_1$-valued smooth functions $f:\Omega\to \cF_1$ 
 with the seminorms 
 $f\mapsto \sup_{x\in \Omega } p(X^\beta_x f_x)$, $p$ seminorm of $\cF_1$, $\beta\in \bN_0^n$.

\begin{proof}
By Corollary \ref{cor_pv+type}
(keeping its notation), 
$\Theta$ induces a continuous mapping 
 $\cC^\infty(\Omega, \tilde S^0) \to \cC^\infty (\Omega, \cF\times \bC)$, 

The Fr\'echet vector space
 $\cC^\infty(\Omega, \cF\times \bC)$
 of smooth functions from $\Omega$ 
to $\cF \times \bC$ is naturally identified with $\cF_0$
which is clearly a closed subspace of $\cC^\infty(\Omega \times (G\backslash\{0\}))   \times \cC^\infty(\Omega)$.
It is also a closed subspace of 
$\cC^\infty_{(-Q)-hom}(\Omega \times (G\backslash\{0\}))   \times \cC^\infty(\Omega)$
where $\cC^\infty_{(-Q)-hom}(\Omega \times (G\backslash\{0\}))$
denotes the space of function $\kappa\in \cC^\infty(\Omega \times (G\backslash\{0\}))$ such that 
the function $\Omega \times(G\backslash\{0\})) \ni  (x,y)\mapsto \kappa(x,y)=\kappa_x(y)$
is $(-Q)$-homogeneous  in $y$.
Adapting the proof that the Fr\'echet space of smooth function 
on the closure of a bounded open subset of a Euclidean space is separable, 
one proves easily that 
$\cC^\infty_{(-Q)-hom}(\Omega \times (G\backslash\{0\}))$
is separable.
Therefore so is $\cF_0$.

The Fr\'echet space $\cC^\infty (\Omega, \tilde S^0)$
of smooth functions from $\Omega$ to $\tilde S^0$
is naturally identified with 
the space of symbols $\dot S^0(\Omega)$.
It is therefore separable.
\end{proof}

\begin{definition}
We denote by $C^*(\dot S^0(\Omega))$ the 
closure of $\dot S^0(\Omega)$ for 
$$
\tau\mapsto \sup_{(x,\dot \pi) \in  \Omega \times (\Gh/\bR^+)}
\| \tau (x,\dot\pi)\|_{\sL(\cH_\pi)} \, .
$$
\end{definition}

Note that if $\sigma\in \dot S^0(\Omega)$, $x_0\in \bar \Omega$ 
and $\pi_0\in \Gh$, then 
$\sigma (x_0,\pi_0)$ makes sense
by Part 2 of Remark \ref{rem_barOmega} for $x_0$,
and by Lemma \ref{lem_pv} for $\pi_0$.
Moreover, one can view $C^*(\dot S^0(\Omega))$ as a space of fields on $\bar \Omega\times \Gh/\bR^+$.
Let us summarise its properties:
\begin{proposition}
\label{prop_C*dotS0}
The space $C^*(\dot S^0( \Omega))$ is a  separable $C^*$-algebra of type 1.

If $\pi_0\in \Gh$ and $x_0\in \Omega$, then the mapping 
$$
\left\{\begin{array}{lll}
\dot S^0(\Omega) &\longrightarrow& \sL(\cH_{\pi_0})\\
\sigma&\longmapsto & \sigma(x_0,\pi_0)
\end{array}\right. \quad,
$$
extends to a continuous mapping $\rho_{x_0,\pi_0}:C^*(\dot S^0(\Omega))\to \sL(\cH_{\pi_0})$ which is an irreducible representation of $C^*(\dot S^0(\Omega))$. For any $r>0$, we have $\rho_{x_0,\pi_0}=\rho_{x_0, r\cdot\pi_0}$.
Denoting by $\dot \pi_0$ the class of representations $\{r\cdot\pi_0, r>0\}$, 
the mapping
$$
R:\left\{\begin{array}{lll}
 \bar \Omega\times (\Gh/\bR^+) &\longrightarrow&\widehat{C^*(\dot S^0( \Omega))} \\
 (x_0,\dot \pi_0) &\longmapsto& \rho_{x_0,\pi_0}
 \end{array}\right.
 $$
is a homeomorphism.
\end{proposition}

Proposition \ref{prop_C*dotS0}  is the non-invariant version of Proposition \ref{prop_C*tildeS0}
whose proof we adapt.

\begin{proof}
First, let us show the analogue of Lemma \ref{lem_prop_C*tildeS0}.
Given a representation $\rho$ of the $C^*$-algebra $C^*(\dot S^0( \Omega))$, we consider 
$$
\pi_\rho ( \phi \, f) =\rho (\phi \sigma_f^*)
=\rho(\phi) \rho(\sigma_f^*)
\quad \phi \in \cC^\infty(\Omega), \ f\in \cS(G).
$$
Proceeding as in   Lemma \ref{lem_prop_C*tildeS0}, 
we see that $\pi_\rho$ extends to a representation of the $C^*$-algebra
$\cC(\bar\Omega; C^*(G))$
of continuous function valued in $C^*(G)$.
We may identify these with symbols $\{\tau (x,\dot\pi)\in \sL(\cH_\pi) : (x,\dot\pi) \in \bar \Omega \times (\Gh/\bR^+)\}$
which depends continuously on $x$.
Note that $\cC(\bar\Omega; C^*(G))$ may be 
obtained as the tensor of $\cC(\bar \Omega)$ with $C^*(G)$
and that its spectrum is $\bar\Omega \times \Gh$.

We now assume that $\rho$ is irreducible.

As  for any $\phi\in \cC(\bar \Omega)$, the operator $\rho(\phi \id)\in \sL(\cH_\rho)$
commutes with any $\rho(\tau)$, $\tau\in \cC(\bar\Omega; C^*(G))$, 
it must be scalar, i.e. $\rho(\phi\id ) \in \bC \id_{\cH_\rho}$.
This yields the one-dimensional representation
$\cC(\bar \Omega)\ni \phi\mapsto \rho(\phi\id ) \in \bC \id_{\cH_\rho}$.
It can not be zero since  $\rho(\id)=0$ would imply that $\rho$ be zero.
Hence it is given by $x_0\in \bar \Omega$, i.e. 
$$
\forall \phi\in \cC(\bar \Omega)\qquad
\rho(\phi\id )  = \phi(x_0) \id_{\cH_\rho}.
$$
This implies that the $\rho_{|1_{\bar \Omega} C^*(\tilde S^0)}$, 
that is, the restriction of $\rho$ to $1_{\bar \Omega} C^*(\tilde S^0)$,
yields an irreducible representation of $C^*(\tilde S^0)$, 
denoted by $(\rho, C^*(\tilde S^0))$.
Clearly ${\pi_\rho}_{|1_{\bar \Omega} C^*(G)}$ 
coincides with $(\pi_{(\rho, C^*(\tilde S^0))}, C^*(G))$ defined in  Lemma \ref{lem_prop_C*tildeS0}
and may be identified with an irreducible representation of $C^*(G)$, i.e. $\pi_\rho|_{1_{\bar \Omega} C^*(G)}\sim \pi_0\in \Gh$.
This easily implies 
$$
\rho (\tau) = \tau(x_0, \pi_0).
$$
And we have obtained that any irreducible representation of $\cC(\bar\Omega; C^*(G))$ is of the form $\delta_{x_0}\otimes \pi_0$. Conversely, if $\rho=\delta_{x_0}\otimes \pi_0$, then it is an irreducible representation of $\cC(\bar\Omega; C^*(G))$. 

The rest of the proof 
is obtained easily by adapting the arguments given in the proof of Proposition \ref{prop_C*tildeS0}.
\end{proof}
 
 \subsection{The states of $C^*(\tilde S^0)$ and $C^*(\dot S^0(\Omega)$}
 
 In Propositions \ref{prop_C*tildeS0}
 and \ref{prop_C*dotS0}, we described the spectra of the $C^*$-algebras
In this section, we show that this allows us to describe  
 the states (i.e. the continuous positive forms)
of  these $C^*$-algebra in terms of objects depending on $\Gh$.
We start with  $C^*(\tilde S^0)$:
 
\begin{proposition}
\label{prop_state_C*tildeS0}
\begin{enumerate}
\item 
If $\ell$ is a state of   $C^*(\tilde S^0)$, 
then there exists a positive measure $\gamma$ on  $\Gh/\bR^+$
and a measurable field of self-adjoint positive trace-class operators $\Gamma=\{\Gamma(\dot \pi)\in \sL(\cH_\pi):\dot \pi\in \Gh/\bR^+\}$ satisfying 
\begin{equation}
\label{eq_prop_trGg}
\int_{\Gh/\bR^+} \tr \left(\Gamma(\dot \pi)\right) d\gamma(\dot \pi) =1,
\end{equation}
and 
\begin{equation}
\label{eq_prop_elltrGg}
\forall \sigma\in C^*(\tilde S^0)\qquad
\ell(\sigma) = \int_{\Gh/\bR^+} \tr\left(\sigma(\dot\pi) \Gamma(\dot \pi)\right) d\gamma(\dot \pi).
\end{equation}

\item 
Conversely, given a positive measure $\gamma$ on  $\Gh/\bR^+$
and a measurable field of self-adjoint positive trace-class operators $\Gamma=\{\Gamma(\dot \pi)\in \sL(\cH_\pi):\dot \pi\in \Gh/\bR^+\}$ satisfying \eqref{eq_prop_trGg}, the linear form $\ell$ defined via \eqref{eq_prop_elltrGg} is a state of  $C^*(\tilde S^0)$.
Furthermore, 
if $\gamma'$ and $\Gamma'$ are  a
 positive measure
and a  measurable field of self-adjoint positive trace-class operators  satisfying \eqref{eq_prop_trGg} and \eqref{eq_prop_elltrGg} for the same state $\ell$, 
then there exists a measurable positive function $f$ on $\Gh/\bR^+$ such that 
$$
d\gamma'(\dot \pi)=f(\dot \pi) d\gamma(\dot \pi)
\quad\mbox{and}\quad
\Gamma'(\dot \pi) = \frac1{f(\dot \pi)} \Gamma(\dot \pi),
\ \dot\pi-a.e.
$$ 
\end{enumerate}
\end{proposition} 
 
 \begin{proof}[Proof of Part 1 of Proposition \ref{prop_state_C*tildeS0}]
 Let $\ell$ be a state of the $C^*$-algebras  $C^*(\tilde S^0)$.
  The GNS construction \cite[Proposition 2.4.4]{Dixmier_C*} yields a representation $\rho$ of $C^*(\tilde S^0)$ on the Hilbert space $\cH_\ell := C^*(\tilde S^0) / \{\sigma : \ell(\sigma \sigma^*)=0\}$ 
and 
$$
\ell(\sigma) = (\rho(\sigma) \xi ,\xi)_{\cH_\ell} \, , 
\quad\sigma\in C^*(\tilde S^0),
$$
where the unit vector $\xi$ is the image of $\id\in \tilde S^0$ via the canonical projection  $C^*(\tilde S^0)\mapsto \cH_\ell$.
We then decompose  \cite[Theorem 8.6.6]{Dixmier_C*} 
the representation $\rho$ 
(taking into account the possible multiplicities) as 
$$
(\rho,\cH_\ell) \sim (\rho_1,\cH_1) \oplus 2(\rho_2,\cH_2) \oplus \ldots \oplus \aleph_0 (\rho_\infty,\cH_\infty),
$$
and each $\rho_r$, $r\in \bN\cup\{\infty\}$, may be disintegrated as
$$
\rho_r \sim  \int_{\widehat{C^*(\tilde S^0)}}\zeta d\gamma_r(\zeta);
$$
furthermore,  the positive measures $\gamma_1, \gamma_2,\ldots, \gamma_\infty$ are mutually singular in $\widehat{C^*(\tilde S^0)}$. 
Consequently we can write $\xi \in \cH_\ell$ as
$$
\xi \sim (\xi_1,\xi_2,\ldots, \xi_\infty), \quad
\mbox{with}\ \xi_r=(\xi_{r,s})_{1\leq s\leq r} \ \mbox{for each} \ r\in \bN
\cup\{\infty\}, \ \mbox{and}\ \xi_{r,s}\in \cH_r.
$$
Note that
$$
1=|\xi|_{\cH_\ell}^2 
= \sum_{r\in \bN
\cup\{\infty\}}
\sum_{s=1}^r |\xi_{r,s}|_{\cH_r}^2
\qquad\mbox{with}\qquad
|\xi_{r,s}|_{\cH_r}^2
= \int_{\widehat{C^*(\tilde S^0)}}|\xi_{r,s}(\zeta)|_{\cH_\zeta}^2 d\gamma_r(\zeta).
$$
By Proposition \ref{prop_C*tildeS0}, 
we may identify $\widehat{C^*(\tilde S^0)}$ with $\Gh/\bR^+$:
$$
\rho_r \sim \int_{\Gh/\bR^+}\dot \pi d\gamma_r(\dot \pi),\qquad
\cH_r \sim  \int_{\Gh/\bR^+}\cH_\pi d\gamma_r(\dot \pi),\qquad
\sum_{r=1}^\infty 
\sum_{s=1}^r 
\int_{\Gh/\bR^+}
|\xi_{r,s}(\dot\pi)|_{\cH_\pi}^2  d\gamma_r(\dot \pi)=1.
$$
Hence $\Gamma_r:=\sum_{s=1}^r \xi_{r,s} \otimes {\xi_{r,s}}^*$
is a $\gamma_r$-measurable field on $\Gh/\bR^+$  of positive traceclass operators of rank $ r$. 
We have obtained:
\begin{align*}
\ell(\sigma) 
&= (\rho(\sigma) \xi ,\xi)
=\sum_{r\in \bN\cup\{\infty\}} 
\sum_{s=1}^r
\int_{\Gh/\bR^+}
(\sigma(\dot \pi)\xi_{r,s}(\dot \pi),\xi_{r,s}(\dot \pi))_{\cH_r}
d\gamma_r(\dot \pi)\\
&=\sum_{r\in \bN\cup\{\infty\}} 
\int_{\Gh/\bR^+}
\tr \left( \sigma(\dot \pi) \Gamma_r (\dot \pi)\right)
d\gamma_r(\dot \pi).
\end{align*}
We now define 
the positive measure $\gamma:=\sum_r \gamma_r$.
As the measures $\gamma_r$ are mutually singular, 
the field
$\Gamma :=\sum_r \Gamma_r$ is measurable 
and satisfies
$$
\Gamma (\dot\pi)\geq0, \quad \tr \Gamma (\dot\pi)<\infty,
\qquad
 \int_{\Gh/\bR^+} \tr \Gamma (\dot\pi) d\gamma(\dot\pi)=1 \, .
 $$
This shows Part 1.
\end{proof}

\begin{proof}[Proof of Part 2 of Proposition \ref{prop_state_C*tildeS0}]
Given a positive measure $\gamma$ on  $\Gh/\bR^+$
and a measurable field of self-adjoint positive trace-class operators $\Gamma=\{\Gamma(\dot \pi)\in \sL(\cH_\pi):\dot \pi\in \Gh/\bR^+\}$ satisfying \eqref{eq_prop_trGg}, one checks easily that the linear form $\ell$ defined via \eqref{eq_prop_elltrGg} is a state of  $C^*(\tilde S^0)$.

To prove the last part of the statement, we consider a positive measure
 $\gamma'$ and  a measurable field of self-adjoint positive trace-class operators $\Gamma'$ 
  satisfying \eqref{eq_prop_trGg} and \eqref{eq_prop_elltrGg} for the same state  $\ell$.
It suffices to consider the case of $\gamma$ and $\Gamma$ obtained as in Part 1;
in particular  $\gamma$ and $\Gamma$ have the same support in $\Gh/\bR^+$.
We may also assume that $\gamma'$ and $\Gamma'$ have the same support in $\Gh/\bR^+$.
For each $r\in \bN\cup\{\infty\}$,
let $B_r$ be the measurable subset of $\Gh/\bR^+$ where $\Gamma'(\dot \pi)$ is of rank $r$ a.e.
We may assume these subsets  disjoint.
We define the measure $\gamma'_r=1_{B_r}\gamma'$ 
and the field $\Gamma'_r:=1_{B_r}\Gamma'$
as the restrictions of $\gamma'$ 
and $\Gamma'$  to $B_r$.
As $\Gamma'_r$ is a measurable field of positive operators of rank $r$, 
there exists a measurable field of orthogonal vectors $(\xi_{r,s})_{s=1}^r$ such that 
$\Gamma'_r=\sum_{s=1}^r\xi'_{r,s}\otimes {\xi'_{r,s}}^*$.
We have $\tr \Gamma'_r= \sum_{s=1}^r|\xi'_{r,s}|^2$.

We define the representation $\rho'$ of $C^*(\tilde S^0)$
and the vector $\xi'$ of $\rho'$ via
$$
\rho':=\oplus_{r\in \bN\cup\{\infty\}} r\int_{\Gh/\bR^+} \dot \pi \ d\gamma'_r(\dot\pi),
\qquad\mbox{and}\qquad
\xi':=\oplus_{r\in \bN\cup\{\infty\}} \oplus_{s=1}^r \int_{\Gh/\bR^+}\xi'_{r,s}(\dot \pi) \ d\gamma'_r(\dot\pi).
$$
We observe that $\xi'$ is a unit vector:
$$
|\xi'|^2
=\sum_{r\in \bN\cup\{\infty\}} \sum_{s=1}^r |\tilde \xi'_{r,s}|^2
=\sum_{r\in \bN\cup\{\infty\}} \int_{\Gh/\bR^+}\tr \Gamma'_r\ d\gamma'_r
= \int_{\Gh/\bR^+}\tr \Gamma'\ d\gamma'=1.
$$
Moreover for any $\sigma\in C^*(\tilde S^0)$:
\begin{align*}
(\rho'(\sigma) \xi',\xi')
&=  \sum_{r\in \bN\cup\{\infty\}} \sum_{s=1}^r
\int_{\Gh/\bR^+} \left(\sigma\xi'_{r,s},\xi'_{r,s}\right) d\gamma'_r
=  \sum_{r\in \bN\cup\{\infty\}} 
\int_{\Gh/\bR^+} \tr \left(\sigma \Gamma'_r\right) d\gamma'_r
\\&=
\int_{\Gh/\bR^+} \tr \left(\sigma \Gamma'\right) d\gamma' = \ell(\sigma).
\end{align*}
In other words, the state associated with $\rho'$ and $\xi'$ coincides with $\ell$. 
This implies  
 that $\rho'$ and $\rho$ are equivalent \cite[Proposition 2.4.1]{Dixmier_C*}, therefore the measures $\gamma'_r$ and $\gamma_r$ are equivalent for every $r\in \bN\cup\{\infty\}$
 \cite[Theorem 8.6.6]{Dixmier_C*}.
 In other words, there exists a measurable positive function $f_r$ supported in $B_r$ such that 
 $d\gamma'_r(\dot \pi)=f_r(\dot \pi) d\gamma_r(\dot \pi)$.
As  $\xi'$ corresponds to $\xi$ via the $(\rho',\rho)$-equivalence, 
we must have $\Gamma_r(\dot \pi)=f_r(\dot \pi) \Gamma'_r(\dot \pi)$.
This concludes the proof of Part 2.
 \end{proof}

From the proof of Proposition \ref{prop_state_C*tildeS0}, 
we can determine easily the pure states, that is, the states  corresponding to the irreducible representations:
\begin{corollary}
\label{cor_prop_state_C*tildeS0}
The pure states of the $C^*$-algebra $C^*(\tilde S^0)$ are
the functionals $\ell=\ell_{\pi_0,v_0}$ of the form:
$$
\ell(\sigma) = (\sigma(\dot \pi_0)v_0,v_0)_{\cH_{\pi_0}},
\quad \sigma\in C^*(\tilde S^0),
$$
where $\pi_0\in \Gh$ and $v_0\in \cH_{\pi_0}$ is a unit vector.
The states $\ell=\ell_{\pi_0,v_0}$ where $\pi_0\in \Gh$ and $v_0\in \cH_{\pi_0}^\infty$ is a smooth unit vector, form a dense subset of the set of states of 
$C^*(\tilde S^0)$.
\end{corollary}
We observe that $\ell_{\pi_0,v_0}$ corresponds to $\gamma(\dot \pi) =\delta_{\dot\pi_0}(\dot \pi)$ and $\Gamma(\dot \pi_0) = v_0\otimes v_0^*$.

\medskip

We also have a similar description of the states  of  $C^*(\dot S^0(\Omega))$:

\begin{proposition}
\label{prop_state_C*dotS0Omega}
Let $\Omega$ be a bounded open set of $G$.
\begin{enumerate}
\item 
If $\ell$ is a state of  $C^*(\dot S^0(\Omega))$, 
then there exists a positive measure $\gamma$ on  $\bar \Omega\times \Gh/\bR^+$
and a measurable field of self-adjoint positive trace-class operators $\Gamma=\{\Gamma(x,\dot \pi)\in \sL(\cH_\pi):(x,\dot \pi)\in \bar \Omega\times \Gh/\bR^+\}$ satisfying 
\begin{equation}
\label{eq_propx_trGg}
\int_{\bar \Omega\times \Gh/\bR^+} \tr \left(\Gamma(x,\dot \pi)\right) d\gamma(x,\dot \pi) =1,
\end{equation}
and 
\begin{equation}
\label{eq_propx_elltrGg}
\forall \sigma\in C^*(\dot S^0(\Omega))\qquad
\ell(\sigma) = \int_{\bar \Omega\times\Gh/\bR^+} \tr\left(\sigma(x,\dot\pi) \Gamma(x,\dot \pi)\right) d\gamma(x,\dot \pi).
\end{equation}

\item
Conversely, given a positive measure $\gamma$ on  $\bar \Omega\times\Gh/\bR^+$
and a measurable field of self-adjoint positive trace-class operators $\Gamma=\{\Gamma(x,\dot \pi)\in \sL(\cH_\pi):(x,\dot \pi)\in \bar \Omega\times\Gh/\bR^+\}$ satisfying \eqref{eq_propx_trGg}, the linear form $\ell$ defined via \eqref{eq_propx_elltrGg} is a state of  $C^*(\dot S^0(\Omega))$.

Furthermore, 
if $\gamma'$ and $\Gamma'$ are  
a positive measure
and a measurable field of self-adjoint positive trace-class operators  satisfying \eqref{eq_propx_trGg} and \eqref{eq_propx_elltrGg} for the same state $\ell$, 
then there exists a measurable positive function $f$ on $\bar \Omega\times\Gh/\bR^+$ such that 
$$
d\gamma'(x,\dot \pi)=f(x,\dot \pi) d\gamma(x,\dot \pi)
\quad\mbox{and}\quad
\Gamma'(x,\dot \pi) = \frac1{f(x,\dot \pi)} \Gamma(x,\dot \pi).
$$ 
\item 
The pure states of the $C^*$-algebra $C^*(\dot S^0(\Omega))$ are
the functionals $\ell=\ell_{x_0,\pi_0,v_0}$ of the form:
$$
\ell(\sigma) = (\sigma(x_0,\dot \pi_0)v_0,v_0)_{\cH_{\pi_0}},
\quad \sigma\in C^*(\tilde S^0),
$$
where $x_0\in \bar\Omega$, $\pi_0\in \Gh$ and $v_0\in \cH_{\pi_0}$ is a unit vector.
The states $\ell=\ell_{x_0,\pi_0,v_0}$ where $x_0\in \bar\Omega$, $\pi_0\in \Gh$ and $v_0\in \cH_{\pi_0}^\infty$ is a smooth unit vector, form a dense subset of the set of states of 
$C^*(\dot S^0(\Omega))$.
\end{enumerate}
\end{proposition} 
\begin{proof}
The proof of Proposition \ref{prop_state_C*dotS0Omega} is a simple modification of the proof of Proposition \ref{prop_state_C*tildeS0}; indeed, it suffices to replace
the characterisation of the spectrum of $C^*(\tilde S^0)$
with the one of $C^*(\dot S^0(\Omega))$ given in Proposition \ref{prop_C*dotS0}.
 It is left to the reader.
\end{proof}

We observe that $\ell_{x_0\pi_0,v_0}$ corresponds to $\gamma (x,\dot\pi) =\delta_{x_0}(x)\otimes\delta_{\dot\pi_0}(\dot\pi)$ and $\Gamma (x_0,\dot\pi_0)= v_0\otimes v_0^*$.


\section{Defect measures}
\label{sec_defect_meas}

In this section, we state and prove our main results, 
that is, the existence of defect measure.
We also give example of such measures
and prove the consistency of our description.

\subsection{Main result}
\label{subsec_main_result}

Before stating our main result, 
let us define the type of object   a microlocal defect measure will be

\begin{definition}
\label{def_traceclass_pos_meas}
Let $\Omega$ be a Borelian set of $G$ with non-empty interior.
\begin{itemize}
\item 
A \emph{positive trace-class-valued measure} on $\Omega\times \Gh/\bR^+$
is a pair $(\gamma,\Gamma)$
where
 $\gamma$ is a positive Radon measure on $\Omega\times \Gh/\bR^+$,
and  $\{\Gamma (x,\dot\pi)\in \sL(\cH_\pi), (x,\dot\pi)\in \Omega\times\Gh/\bR^+\}$
is a measurable field of positive traceclass operators 
  satisfying  for all $C$ compact subset of $\Omega$
$$
 \int_{C \times \Gh/\bR^+} \tr\left( \Gamma (x,\dot\pi)\right) d\gamma(x,\dot\pi)<\infty \, , \quad\mbox{and}\quad
 \gamma(C\times \Gh/\bR^+)<\infty.
 $$
\item 
Two positive trace-class-valued measures $(\gamma,\Gamma)$
and $(\gamma',\Gamma')$ are \emph{equivalent} when
there exists a measurable positive function $f$ on $\Omega\times\Gh/\bR^+$ such that 
$$
d\gamma'(x,\dot \pi)=f(x,\dot \pi) d\gamma(x,\dot \pi)
\quad\mbox{and}\quad
\Gamma'(x,\dot \pi) = \frac1{f(x,\dot \pi)} \Gamma(x,\dot \pi).
$$ 
The equivalence class of $(\gamma,\Gamma)$ is denoted by $\Gamma d\gamma$.
\end{itemize}
\end{definition}

We can now state our main theorem:
 \begin{theorem}
\label{thm_defect_measure}
Let $\Omega$ be a non-empty open set of $G$.
Let $(u_k)$ be a sequence in $L^2(\Omega,loc)$
and $u\in L^2(\Omega,loc)$.
We assume that $u_k \rightharpoonup_{k\to\infty}  u$ a.e. in $L^2(\Omega,loc)$.
Then there exist 
a subsequence $(u_{k(j)})_{j\in \bN}$ of $(u_k)$
and a  positive trace-class-valued measure $(\gamma,\Gamma)$
on $\Omega\times \Gh/\bR^+$
 such that for any $A\in \Psi^0_{cl}(\Omega)$, 
 we have the convergence
$$
\lim_{j\to\infty} 
\left(A (u_{k(j)}-u),(u_{k(j)}-u)\right)_{L^2(\Omega)} 
=
\int_{\Omega \times (\Gh /\bR^{+})}
\tr \left(\princ_0 (A) (x,\dot \pi) \ \Gamma(x,\dot \pi) \right)
d  \gamma(x,\dot\pi) \, ,
$$

Moreover, once the subsequence $(k(j))$ is fixed, 
the  positive trace-class-valued measure
 $(\Gamma,\gamma)$ is unique up to equivalence.
\end{theorem}

\begin{definition}
\label{def_MDM}
We keep the notation of Theorem \ref{thm_defect_measure}.
A sequence $(u_k)$ is \emph{pure} when the subsequence is trivial, i.e. $k(j)=j$.
In this case, 
 the equivalence class  $\Gamma d\gamma$
is called the microlocal defect measure, or \emph{MDM},
of  $(u_{k})$.
\end{definition}

The definition of pure sequence follows the vocabulary set in \cite{gerard_91}; it bears no relation with pure states.

\medskip

Let us turn our attention to the proof of Theorem \ref{thm_defect_measure}.
It suffices to prove the case $u=0$.
Furthermore an argument of diagonal extraction 
over a suitable sequence of bounded open subsets exhausting $\Omega$
shows that it suffices to prove:
 \begin{proposition}
\label{prop_defect_measure}
Let $\Omega$ be a non-empty bounded open set of $G$.
Let $(u_k)$ be a  sequence in $L^2(\Omega)$
converging weakly to 0.
Then there exist 
a subsequence $(u_{k(j)})_{j\in \bN}$ of $(u_k)$
and a  positive trace-class-valued measure $(\gamma,\Gamma)$
 on $\bar \Omega\times \Gh/\bR^+$,
such that for any $A\in \Psi^0_{asymp}$, 
we have the following convergence:
$$
\lim_{j\to\infty} 
\left(A  u_{k(j)},  u_{k(j)} \right)_{L^2(\Omega)} 
=
\int_{\bar \Omega \times (\Gh /\bR^{+})}
\tr \left(\princ_0 (A) (x,\dot \pi) \ \Gamma(x,\dot \pi) \right)
d  \gamma(x,\dot\pi) \, .
$$
Moreover once the subsequence $(k(j))$ is fixed, 
the trace-class-valued positive measure
 $(\Gamma,\gamma)$ is unique up to equivalence.
\end{proposition}

\begin{remark}
\label{rem_prop_defect_measure}
As the sequence
$(u_k)$ 
converges weakly to 0, it is bounded in norm.
If $(u_k)$ converges to 0 for the $L^2$-norm, then
 $\lim_{k\to\infty}(Au_k,u_k)_{L^2}=0$ since $A$ is bounded on $L^2(\Omega)$.
Therefore, in the proof of Proposition \ref{prop_defect_measure},
up to extraction of a subsequence, 
we may assume
\begin{equation}
\label{eq_WMA_prop_inv_defect_measure1}
\sup_{k}\|u_k\|_{L^2(\Omega)} =
\lim_{k\to\infty} \|u_k  \|_{L^2(\Omega)}  =1.
\end{equation}
Furthermore, 
  Rellich's theorem (cf. Theorem \ref{thm_cq_rellich} and its proof)
 shows that it suffices to consider only operators 
 $A\in \Psi^0_{asymp}$ with one term (the 0-homogeneous one) in their homogeneous expansion.
\end{remark}

The proof of Proposition \ref{prop_defect_measure} relies on the following lemma:

\begin{lemma}
\label{lem_limit}
Let $\Omega$ be a non-empty bounded open set of $G$.
Let $(u_k)$ be a  sequence in $L^2(\Omega)$
converging weakly to 0
 and satisfying \eqref{eq_WMA_prop_inv_defect_measure1}.
If $\cR$ is a positive Rockland operator, 
$\psi\in C^\infty(\bR)$ a function
such that $\psi\equiv0$ on a neighbourhood of 0 and $\psi\equiv 1$ on a neighbourhood of $+\infty$, 
and  $\sigma\in \dot S^0_{comp}$,
we set
$$
v_k^{(\sigma)}:=
(\Op(\sigma \psi(\pi(\cR))) u_{k},  u_{k}
\big)_{L^2(\bar \Omega)} \, .
$$

\begin{enumerate}
\item 
The sequence $(v_k^{(\sigma)})_{k\in \bN}$ is bounded.
\item 
We can extract a subsequence $(k_j)_{j\in \bN}$
such that 
the sequence $(v_{k_j}^{(\sigma)})_{j\in \bN}$
has a finite limit in~$\bC$.
\item 
If $(k_j)$ is as in Part (2), then the limit is 
 independent of $\psi$ and $\cR$
and it is the same with 
$\Op(\psi(\pi(\cR)\sigma )$
or with 
$\Op(\psi(\pi(\cR)\sigma \psi(\pi(\cR))$
instead of $\Op(\sigma \psi(\pi(\cR))$.
\item If $(k_j)$ is as in Part (2), then 
the sequence
$(v_{k_j}^{(\sigma^*)})_{j\in \bN}$
 has also a finite limit
 and 
$$
\lim_{j\to \infty} v_{k_j}^{(\sigma^*)}
=
\overline{\lim_{j\to \infty} v_{k_j}^{(\sigma)}}
.
$$

\item 
If $\sigma$ is of the form $\sigma=\tau^* \tau$ 
with $\tau\in \dot S^0(\bar \Omega)$ then 
any limit obtained as in Part (2) is non-negative.

\end{enumerate}
\end{lemma}

\begin{proof}[Proof of Lemma \ref{lem_limit}]
Part (1) follows from $\Op(\sigma \psi(\pi(\cR))) \in \Psi^0$ being bounded on $L^2(G)$. Indeed, we have 
\begin{equation}\label{est_symb}
|v_{k}^{(\sigma)}|
\leq \|\Op(\sigma \psi(\pi(\cR))\|_{\sL(L^2(G))}
\sup_{k'\in\bN} \| u_{k'} \|_{L^2(G)}^2,
\end{equation}
and the supremum above is assumed to be equal to 1.
This proves Part (1) and thus Part~(2).

By Proposition \ref{prop_homogeneous_asymptotic} and its proof
together with the properties of the pseudo-differential calculus,
\begin{eqnarray*}
\Op(\sigma^* \psi(\pi(\cR))) 
&=& 
\Op(\psi(\pi(\cR))\sigma^* ) 
\ \mbox{mod} \Psi^{-\infty}(\Omega_{N_0})
\\
&=& 
\Op(\sigma \psi(\pi(\cR)))^* +E,
\end{eqnarray*}
where $E$ is an error term in $\Psi^{-1}(\Omega_{N_0})$.
Using Rellich's theorem as 
in Theorem \ref{thm_cq_rellich} and its proof shows Part (4).

If  $\sigma=\tau^* \tau$ with $\tau\in \dot S^0(\bar\Omega)$, 
then, extending $\sigma$ and $\tau$ to symbols on $\dot S^0$, by Proposition~\ref{prop_homogeneous_asymptotic} and its proof
together with the properties of the pseudo-differential calculus,
\begin{eqnarray*}
\Op(\sigma \psi(\pi(\cR))) 
&=& 
\Op(\psi(\pi(\cR))\sigma \psi(\pi(\cR))) 
\ \mbox{mod} \Psi^{-\infty}(\Omega_{N_0})
\\
&=& 
\Op(\psi(\pi(\cR))\tau^* )
\Op(\tau \psi(\pi(\cR)))
\ \mbox{mod} \Psi^{-1}(\Omega_{N_0}).
\\
&=& 
\Op(\psi(\pi(\cR))\tau )^*
\Op(\tau \psi(\pi(\cR))) +E,
\end{eqnarray*}
where $E$ is an error term in $\Psi^{-1}(\Omega_{N_0})$.
Thus we have
$$
v_k^{(\sigma)}
=
\|\Op(\psi(\pi(\cR))\tau ) u_k\|_{L^2(\Omega)}^2
+ (E u_k,  u_k)_{L^2( \Omega)}.
$$
The first term of the right-hand side is non-negative for all $k\in \bN$
whereas the second term tends to 0 as $k\to \infty$
by Theorem \ref{thm_cq_rellich} and its proof.
This shows Part (5).
\end{proof}

We can now prove Proposition \ref{prop_defect_measure}.

\begin{proof}[Proof of Proposition \ref{prop_defect_measure}]
Let $(u_k)$ be a  sequence in $L^2(\Omega)$
converging weakly to 0
 and satisfying equation~\eqref{eq_WMA_prop_inv_defect_measure1}.
By Proposition \ref{prop_dotS0_separability}, 
there exists a dense and countable $(\bQ+i\bQ)$-subspace 
$V_0$ of $\dot S^0(\bar \Omega)$.
By diagonal extraction, we may assume that the subsequence 
obtained in Lemma \ref{lem_limit}
is the same for any element of $V_0$
and we set
$$
\ell(\sigma) :=\lim_{j\to \infty} v_{k_j}^{(\sigma)}\, ,
\qquad \sigma\in V_0.
$$
Using the density of $V_0$ and the proof Part (1) of  Lemma \ref{lem_limit}, 
we extend $\ell$ to a continuous linear functional on the Fr\'echet space $\tilde S^0$
 satisfying
for any $\sigma\in\dot S^0(\bar\Omega)$
$$
|\ell(\sigma)|\leq 
\|\sigma \psi(\pi(\cR)) \|_{\sL(L^\infty(\Gh)}
\lesssim
\|\sigma\|_{\dot S^0(\bar\Omega),a,b,c}.
$$
having kept the same notation for $\ell$ and its extension.

Note that we can construct the subspace 
$V_0$ of  $\dot S^0(\Omega)$ as follows.
We consider $V$  a dense and countable $(\bQ+i\bQ)$-subspace of $\tilde S^0$ and $V_1$ a dense and countable $(\bQ+i\bQ)$-subspace of $\cC^\infty( \Omega)$.
Then the tensor product of $V$ and $V_1$ yields 
$V_0$, the set of symbols which are finite linear combinations over $(\bQ+i\bQ)$
of $\phi(x)\sigma(\dot \pi)$ with $\phi\in V_1$, $\sigma\in V_0$.
Then $V_0$ is a dense countable subset of the Fr\'echet space $\dot S^0 (\Omega)$.
The proof of Proposition \ref{prop_dotS0_separability} shows that $V_0$ is also dense in the Banach space $C^*(\dot S^0(\Omega))$ 
whose norm satisfies
$$
\sup_{(x,\dot \pi) \in \bar \Omega \times (\Gh/\bR^+)}
\| \sigma (x,\dot\pi)\|_{\sL(\cH_\pi)}
= \inf
\left\{\sum_i \sup_{\dot\pi\in\Gh/\bR^+}\|\tau_i\|_{\sL(\cH_\pi)}
\sup_{x\in\bar \Omega} |f_i(x)|
: \sigma=\sum_i f_i\tau_i
\right\}.
$$
If the symbol $\tau$ is of the form $f(x)\sigma(\dot \pi)$, 
with $f\in V_1$, $\sigma\in V$, then $\tau\in \dot S^0(\Omega)$ and
\begin{eqnarray*}
v_k^{(\tau)}
&=&
\big(\Op(\sigma \psi(\pi(\cR))) u_{k},\, \bar f  u_{k}
\big)_{L^2(\Omega)}
\\
|v_k^{(\tau)}|
&\leq &
\|\Op(\sigma \psi(\pi(\cR)))\|_{\sL(L^2(\Omega))}
\|  u_{k}\|_{L^2(\Omega)} 
\|\bar f  u_{k}\|_{L^2(\Omega)}
\\
&\leq &
\|\sigma\|_{L^\infty(\Gh)} 
\|\psi\|_{L^\infty(\bR)}
\sup_{k}\| u_{k}\|_{L^2(\Omega)} ^2
\|f \|_{L^\infty(\Omega)} \, ,
\end{eqnarray*}
thus
$$
|\ell(\tau)|\leq  \|\sigma\|_{L^\infty(\Gh)} \|f \|_{L^\infty(G)}
=\sup_{(x,\dot \pi) \in  \bar \Omega \times (\Gh/\bR^+)}
\| \tau (x,\dot\pi)\|_{\sL(\cH_\pi)}.
$$
Hence $\ell$ admits a unique continuous extension to a linear functional of  $C^*(\dot S^0(\Omega))$.
Lemma~\ref{lem_limit} implies
 that $\ell$ is a state of the $C^*$-algebra $C^*(\dot S^0( \Omega))$.
 The application of Proposition \ref{prop_state_C*dotS0Omega} (see also Remark \ref{rem_prop_defect_measure}) concludes the proof of Proposition \ref{prop_defect_measure}, and
so of Theorem \ref{thm_defect_measure}.
\end{proof}

 \subsection{Example: spatial concentration}\label{sec_ex1}
 
In this section, we study the example of a sequence of functions 
whose mass in concentrating in 0:
 
 \begin{proposition}
 \label{prop_spatial_concentration}
 Let $u_1\in \cS(G)$. 
 We define 
$$
u_k(x) =k^{Q\over 2} u_1(k x),\quad x\in G, \ k\in \bN_0\, .
$$
Then $u_k\TendWeak{k}{\infty} 0$ in $L^2(G,loc)$ and this sequence is pure. 
Its MDM is given  
(with the notation of Section~\ref{subsec_polar_dec_Gh}) by
$$
\gamma(x,\pi)= \delta_{x=0} \otimes \varsigma(\pi), 
\quad
\Gamma(\pi)=\int_{r=0}^\infty 
 \hat u_1(r\cdot \pi) \hat u_1(r\cdot \pi)^* r^{Q-1} dr
 \ \mbox{for}\ |\pi|=1.
 $$
\end{proposition}

Recall that $\Sigma_1=\{|\pi|=1\}$ and that Remark \ref{rem_lem_pol_dec_Gh} then implies that the measure $\gamma$ is independent on a choice of quasi-norm. Besides, 
Lemma \ref{lem_Kmap_quotient}
shows that $\gamma$
may be viewed as a measure on $G\times (\Gh/\bR^+ \ \backslash\{1\})$. Note also that $\Gamma(\dot \pi)\geq0$ and 
\begin{eqnarray*}
\int_{\Sigma_1} \tr \Gamma(\dot \pi)\ d\varsigma(\pi)
&=&
 \int_{\Sigma_1}
\int_{r=0}^\infty 
 \|\hat u_1(r\cdot \pi) \|_{HS(\cH_\pi) }^2
\ r^{Q-1} dr \ d\varsigma(\pi)
 \\
&=&
 \int_{\Gh}
 \|\hat u_1(\pi) \|_{HS(\cH_\pi) }^2
  d\mu(\pi)  
  =
  \|u_1\|_{L^2(G)}<\infty.
\end{eqnarray*}
One checks easily that $\Gamma$ on $\Gh$ is a $(-Q)$-homogeneous field of operators.

\begin{proof}[Proof of Proposition \ref{prop_spatial_concentration}] 
By Rellich's theorem (cf. Theorem \ref{thm_cq_rellich} and its proof),
we may assume that $A=\Op(\tilde \sigma_0)$
where $\tilde \sigma_0 =\sigma_0 \psi(\pi(\cR))$.
Using \eqref{eq_int_Gf(rx)dx}, 
the group Fourier transform of $u_k$ is
$$
\widehat u_k(\pi)
=k^{- \frac Q 2} \widehat u_1(k^{-1}\cdot \pi).
$$
Hence we have:
\begin{eqnarray*}
(A u_k,u_k)
&=&
\int_G \int_{\Gh} \tr \left(\pi(x) \tilde\sigma_0(x,\pi) \widehat u_k(\pi)\right) d\mu(\pi) \bar u_k(x) dx\\
&=&
\int_G \int_{\Gh} \tr \left(\pi(x) \tilde\sigma_0(x,\pi) 
\widehat u_1(k^{-1}\cdot \pi)\right)  d\mu(\pi) \bar u_1(kx) dx\\
&=&
\int_G \int_{\Gh} \tr \left(\pi'(x') \tilde \sigma_0(k^{-1}x',k\cdot \pi') 
\widehat u_1(\pi') \right) d\mu(\pi') \bar u_1(x') dx',
\end{eqnarray*}
after the change of variables $x'=kx$ and $\pi'=k^{-1}\cdot\pi$,
using \eqref{eq_int_Gf(rx)dx}
and \eqref{eq_dilation_mu}.
We are going to prove 
that the following expression 
tends to 0 as $k\to\infty$:
\begin{eqnarray*}
&&(A_0u_k,u_k)-\int_{\Gh} \tr \left(\widehat u_1(\pi')^*
\sigma_0(0,\pi') 
\widehat u_1(\pi') \right) d\mu(\pi')
\\ &&\qquad=
\int_G \int_{\Gh} \tr \left(\pi(x)
\left(\tilde \sigma_0(k^{-1}x,k\cdot \pi) -\sigma_0(0,k\cdot\pi)\right)
\widehat u_1(\pi) \right) d\mu(\pi) \bar u_1(x) dx
\\ &&\qquad=
T_1+T_2,
\end{eqnarray*}
where 
\begin{eqnarray*}
T_1&=&
\int_G \int_{\Gh} \tr \left(\pi(x) 
\left(\sigma_0(k^{-1}x,k\cdot \pi) 
- \sigma_0(0,k\cdot \pi) \right) 
\psi(k\cdot \pi(\cR)) \
\widehat u_1(\pi) \right) d\mu(\pi) \bar u_1(x) dx,
\\
T_2&=&
\int_G \int_{\Gh} \tr \left(\pi(x)
\sigma_0(0,\pi)  (1-\psi)(k\cdot \pi(\cR))
\ \widehat u_1(\pi) \right) d\mu(\pi) \bar u_1(x) dx
\\
&=&
 \int_{\Gh} \tr \left(
\sigma_0(0,\pi) (1-\psi)(k\cdot \pi(\cR))
\ \widehat u_1(\pi) \widehat u_1(\pi)^*\right) d\mu(\pi) ,
\end{eqnarray*}
and the function $\psi$ is chosen as usual, $\psi\equiv 0$ on a neighbourhood of~$0$ and $\psi\equiv 1 $ on $(\Lambda,\infty)$ for some $\Lambda>0$.
 
For $T_2$,
 we  use that by \eqref{eq_psi(rpiR)=I}, 
 for each $\pi\in \Gh$, there exists $k_\pi\in \bN$ such that
$(1-\psi)(k\cdot \pi(\cR))=0$  for all $k\geq k_\pi$.
Hence for $k\geq k_\pi$, we also have
$$
\tr \left(
\ (1-\psi)(k\cdot \pi(\cR))\sigma_0(0,\pi) 
\ \widehat u_1(\pi) \widehat u_1(\pi)^*\right)
=0.
$$
Since we have
$$
\left|\tr \left(
\ (1-\psi)(k\cdot \pi(\cR))\sigma_0(0,\pi) 
\ \widehat u_1(\pi) \widehat u_1(\pi)^*\right) \right|
\leq
C_{\psi,\sigma_0}
\| \widehat u_1(\pi) \|_{HS(\cH_\pi)}^2
$$
with 
$C_{\psi,\sigma_0}:=\sup_{\lambda>0}|1-\psi(\lambda)| \ 
\sup_{\pi'\in \Gh}\|\sigma_0(0,\pi') \|_{\sL(\cH_\pi')}\in (0,\infty)$
and
$$
\int_{\Gh} \| \widehat u_1(\pi) \|_{HS(\cH_\pi)}^2 d\mu(\pi) =\|u_1\|_2^2<\infty,$$
the Lebesgue dominated convergence theorem yields
that $T_2$
tends to 0 as $k\to\infty$.

\medskip

Let us now study $T_1$.
The mean value theorem stated in Lemma \ref{lem_mvt}
extends to Banach value functions.
Hence fixing a homogeneous quasi-norm $|\cdot|$, 
there exists a constant $C>0$ such that
for any $\sigma\in \dot S^0$, $x\in G$ and $r>0$,
we have
$$
\sup_{\pi\in \Gh}
\|\sigma_0(rx,\pi) -\sigma_0(0,\pi)\|_{\sL(\cH_\pi)}
\leq 
C
\sum_{j=0}^n |r x|^{\nu_j} 
\sup_{y\in G, \pi\in \Gh} \| X_j  \sigma_0(y,\pi)\|_{\sL(\cH_\pi)}.
$$
We obtain
\begin{eqnarray*}
\left|\tr \left(\pi(x) \
\psi(k\cdot \pi(\cR))\left(\sigma_0(k^{-1}x,k\cdot \pi) 
- \sigma_0(0,k\cdot \pi) \right)
\widehat u_1(\pi) \right) \bar u_1(x)\right|
\\
\leq
C C'_{\psi,\sigma_0} |u_1(x)| 
\sum_{j=0}^n |k^{-1} x|^{\nu_j} 
\tr |\widehat u_1(\pi)|
\end{eqnarray*}
where 
$\displaystyle{
C'_{\psi,\sigma_0}:=\sup_{\lambda>0}|\psi(\lambda)|
\sup_{y\in G,\pi \in \Gh, j=1,\ldots,n} \| X_j  \sigma_0(y,\pi)\|_{\sL(\cH_\pi)} \in (0,\infty).
}$
Since $u_1\in{\mathcal S}(G)$, we have
$$
\int_{G} |u_1(x)| 
\sum_{j=0}^n |x|^{\nu_j} dx <\infty
\quad\mbox{and}\quad
\int_{ \Gh} 
\tr |\widehat u_1(\pi)| \ d\mu(\pi)
 <\infty,
$$
and
the Lebesgue dominated convergence theorem yields again
that $T_1$ tends to 0 as $k\to\infty$.

We have shown that $T_1$ and $T_2$ tend to 0 as $k\to\infty$
and this implies  
$$
(Au_k,u_k)_{L^2(G)}\Tend{k}{\infty}
\int_{\Gh} \tr \left(\widehat u_1(\pi')^*
\sigma_0(0,\pi') 
\widehat u_1(\pi') \right) d\mu(\pi'),
$$
which gives the result by use of the polar decomposition for the Plancherel measure, 
see Section~\ref{subsec_polar_dec_Gh}.
\end{proof}

\subsection{Example: oscillations from square integrable representations}\label{sec_ex2}

We now study another example, which may be viewed as a spectral or dual concentration.
We consider a graded group $G$ which admits a (unitary irreducible) representation $\pi_0$ 
which is square integrable modulo its centre.
We also fix a smooth unit vector $v_0\in \cH_{\pi_0}^\infty$ and consider the associated matrix coefficient:
$$
e_0 (x):=
(\pi_0(x)v_0,v_0)_{\cH_{\pi_0}}, \quad x\in G.
$$
We may assume that the basis $\{X_1, \ldots, X_n\}$ of the Lie algebra $\fg$ 
has been chosen so that 
a subset $\{X_{j_1},\ldots, X_{j_{n_\fz}}\}$, form a basis for the centre $\fz$ of $\fg$.
Therefore we can write any element $x$ as
$$
x= \exp_G (x_1X_1+\ldots+x_nX_n)
=x' x_\fz=x_\fz x' ,
$$
where
$\displaystyle{
x_\fz= \exp_G (x_{j_1}X_{j_1}+\ldots+x_{n_\fz}X_{n_\fz})
\quad\mbox{and}\quad
x'=\exp_G \left(\sum_{j\notin \{j_1,\ldots,j_{n_\fz}\}} x_j X_j\right).
}$
Naturally, we identify the centre of the Lie algebra $\fz$ and the centre of the group $Z:=\exp_G\fz$ with $\bR^{n_{\fz}}$.
Note that we still consider anisotropic dilations in those directions.
The quotient group $G':=G/Z$ is also graded and we denote by $Q'$ its homogeneous dimensions, also given by
$$
Q':=\sum_{j\notin \{j_1,\ldots,j_{n_\fz}\}} \upsilon_j.
$$
Finally, we denote by 
 $d_{\pi_0}$  the formal degree of $\pi_0$
 for which we have for any $v_1,w_1,v_2,w_2\in \cH_{\pi_0}$:
 \begin{equation}
d_\pi \int_{G/Z} 
 (\pi_0(x') v_1,w_1)_{\cH_{\pi_0}}
\overline{ (\pi_0(x') v_2,w_2)_{\cH_{\pi_0}}}
 dx'=
 (v_1,w_1)_{\cH_{\pi_0}}  
\overline{  (v_2,w_2)_{\cH_{\pi_0}}} ,
 \end{equation}
see \cite[p. 169 and theorem 4.5.11]{corwingreenleaf}.

\begin{proposition}
\label{prop_oscillation}
Let $u_0\in \cS(\bR^{n_\fz})$.
For each $k\in \bN$, let  $u_k:G\to \bC$ be the square integrable function given by
$$
u_k(x)=k ^{{Q'}\over 2} e_0(kx) u_0(x_\fz), 
\quad x\in G.
$$
Then $\|u_k\|_{L^2(G)}=\|e_0\|_{L^2(G')} \|  u_0\|_{L^2(Z)}<\infty$
and  $u_k\TendWeak{k}{\infty} 0$ in $L^2(G,loc)$.
This sequence is pure and its MDM  is given by
$$\gamma(x,\dot\pi)= \left(\frac {|u_0(x_\fz)|^2}{d_{\pi_0}} dx_\fz \otimes \delta_{x'=0}\right)  
\otimes \delta_{\dot\pi=\dot\pi_0},$$
with  $\Gamma(\dot \pi_0)=v_0 \otimes v_0^*$ being the orthogonal projection on $\bC v_0$.
\end{proposition}

The Schwartz function on the centre is needed to contain the oscillations,
as in the abelian case. 
Indeed, on the one hand, on the centre $Z$ of the group, 
$\pi_0$ coincides with the character $e^{i\lambda_0\cdot}$, i.e. 
$
\pi(x_\fz)=e^{i\lambda_0x_\fz}
$
where we identify $x_\fz$ with an element of $\bR^{n_\fz}$ 
and
where $\lambda_0x_\fz$ denotes the standard scalar product of the two elements $\lambda_0$ and $x_\fz$ of $\bR^{n_\fz}$.
Thus for any $x=x'x_\fz$ in $G$ we have
$$e_0(x)
= 
(\pi_0(x'x_\fz)v_0,v_0)_{\cH_{\pi_0}} \\
= 
e^{i\lambda_0x_\fz} 
(\pi_0(x')v_0,v_0)_{\cH_{\pi_0}} \\
= 
e^{i\lambda_0x_\fz} 
e_0(x').
$$
On the other hand, 
$
e_0\big|_{G'} \in \cS(G').
$
See again \cite[p. 169 and theorem 4.5.11]{corwingreenleaf}.

\medskip

Before starting the proof, let us describe the more concrete case of the Heisenberg group and the matrix coefficient given by the bounded spherical functions, see e.g. \cite{BFG}. More precisely, 
we realise 
the Heisenberg group as $\bH_1=\{(x,y,t)\in \bR^3\}$ with law
$$
(x,y,t)(x',y',t')=\big(x+x',y+y',t+t'+\frac 12 (xy'-x'y)\big).
$$
Let $\pi_0$ be the representation of $\bH_1$ determined (up to equivalence) by the fact that it coincides with the character $t\mapsto e^{it}$ on the centre of $\bH_1$.
For the smooth vector, we choose the $\ell$-th  Hermite function (with $L^2(\bR)$-normalisation)
if we realise this representation in the Schr\"odinger model, 
or equivalently, the (suitably normalised) monomial of degree $\ell$ in the Bergman-Fock model. 
In this case, the matrix coefficient is given by 
$$
e_0(x,y,t)=e^{it} \cL_\ell (\frac {x^2+y^2}2 ),
$$
where $\cL_\ell$ is the $\ell$-th Laguerre function, that is,
$\cL_\ell(r)= e^{-r\over 2} L_\ell (r)$ and $L_\ell$ is the $\ell$-th Laguere polynomial.
Note that the $e_0$ in this particular case 
is of the form described above.

\begin{proof}[Proof of Peoposition \ref{prop_oscillation}]
First let us show that each function $u_k$ is square integrable:
\begin{align*}
\|u_k\|_{L^2(G)}^2
&=
\int_{G'} \int_{\bR^{n_\fz}}
|e_0(kx' kx_\fz) u_0 (x_z) k^{Q'\over 2}|^2 
dx_\fz dx' =
\int_{G'} |e_0(kx')|^2 k^{Q' }dx'
\int_{\bR^{n_\fz}} | u_0 (x_z) |^2 dx_\fz  \\
&=
\int_{G'} |e_0(x'')|^2 dx''
\int_{\bR^{n_\fz}} | u_0 (x_z) |^2 dx_\fz,
\end{align*}
having used the change of variable $x''=kx'$.
As the functions $e_0$ and $u_0$ are Schwartz on $G'$ and $\bR^{n_\fz}$ respectively, 
the quantity above is finite, and  $u_k\in L^2(G)$ with $\|u_k\|_{L^2(G)}=\|e_0\|_{L^2(G')} \|  u_0\|_{L^2(Z)}$.

For any $\phi_1 \in \cD(\bR^{n_\fz}),\phi_2\in \cD(G')$, we have
$$
\int_{G'} \int_{\bR^{n_\fz}}
 u_k(x_\fz x') \phi_1(x_\fz) \phi_2(x') dx_\fz dx'
=
 \int_{\bR^{n_\fz}}
e^{i\lambda_0 (k x_\fz)} u_0(x') \phi_1(x_\fz) dx_\fz
\int_{G'}e_0(k x')  \phi_2(x') k^{Q'\over 2}  dx'.
$$
By the Riemann-Lebesgue theorem, the integral over $\bR^{n_\fz}$ tends to  zero as $k\to\infty$.
After the change of variable $x''=kx'$, 
the integral over $G'$ becomes
$$
\int_{G'}e_0(k x')  \phi_2(x') k^{Q'\over 2}  dx'
=
k^{-Q'\over 2} \int_{G'}e_0(x'')  \phi_2(k^{-1}x')   dx'
\sim_{k\to\infty} 
k^{-Q'\over 2} \phi_2(0) \int_{G'}e_0(x'')     dx',
$$
thus this integral tends to zero as $k\to\infty$.
Hence $u_k\TendWeak{k}{\infty} 0$ in $L^2(G,loc)$.

Let us now compute the MDM of $(u^k)$.
Let $A=\Op(\sigma)\in \Psi_{cl}^0(G)$.
Let $\chi\in \cD(G)$ be real valued and such that the support of the integral kernel of $A$ is in 
$\{\chi=1\}\times \{\chi=1\}$. 
\begin{eqnarray*}
(Au_k,u_k)_{L^2(G)}
&=&
(A(\chi u_k) , \chi u_k)_{L^2(G)}
\\&=&
\int_{\Gh}\int_{G}
\tr\left( \pi(x) \sigma_0(x,\pi)\psi(\pi_0(\cR))
\pi(\chi u_k)\right)
 (\chi \bar u_k)(x) dx d\mu(\pi);
\end{eqnarray*}
here we understand the double integral over $\Gh$ as in Proposition \ref{prop_FIF2}, that is, as the limit  of the absolutely convergent double integral: 
$$
 \lim_{N\to +\infty}\int_{N\cdot \cC} \int_G \trN 
\trN\left( \pi(x) \sigma_0(x,\pi)\psi(\pi_0(\cR))
\pi(\chi u_k)\right)
 (\chi \bar u_k)(x) dx d\mu(\pi),
$$
where  $\cC$ a compact neighbourhood of $1\in \Gh$ such that $\cup_{N\in \bN} N\cdot \cC=\Gh$, 
and $\trN$ denotes the trace of the operators projected on the subspace spanned by the first $N$ vectors, having fixed a fundamental sequence of vector fields.
Hence   we are led to study:
\begin{eqnarray*}
&&\int_{\Gh}\int_{G}
\tr \left(\pi(x) \sigma(x,\pi)\pi(\chi u_k) \right)
(\chi \bar u_k)(x) dx d\mu(\pi)
\\&&\quad=
\int_{\Gh}\int_{G}\int_G
\tr \left(\pi(x) \sigma(x,\pi)\pi(y)^*\right)
 (\chi \bar u_k)(x) (\chi u_k)(y)dxdyd\mu(\pi),
\end{eqnarray*}
having expanded $\pi (\chi u_k)$. This multiple integral is again convergent. 
Applying the change of variables first $y\mapsto w=y^{-1}x$
and using the properties of the trace,
 the integral above becomes
\begin{eqnarray*}
&&\int_{\Gh}\int_{G}\int_G
\tr \left( \pi(x) \sigma(x,\pi)\pi(wx^{-1})\right) 
(\chi \bar u_k)(x) \
(\chi u_k)(x w^{-1})dxdwd\mu(\pi)
\\&&\qquad=
\int_{\Gh}\int_{G}\int_G
\tr \left(\pi(w) \sigma(x,\pi)\right)
 (\chi \bar u_k)(x) 
 \ 
 (\chi u_k)(x w^{-1})dxdwd\mu(\pi)
\\&&\qquad=
\int_{\Gh}\int_{G}\int_G
\tr \left(\pi'(w') \sigma(x,k\cdot\pi')\right)
 (\chi \bar u_k)(x)
 \  (\chi u_k)(x\  k^{-1}{w'} ^{-1})dxdw'd\mu(\pi'),
\end{eqnarray*}
after the change of variable  $(\pi,w)\mapsto (\pi',w')= (k^{-1}\cdot\pi,kw)$,
 whose Jacobian is 1
 by \eqref{eq_int_Gf(rx)dx} and \eqref{eq_dilation_mu}.
 Let us write
\begin{eqnarray*}
(\chi \bar u_k)(x)\ (\chi u_k)(x \ k^{-1}{w'}^{-1})
&=&
k^{Q'}\chi(x)  \chi(x \ k^{-1}{w'}^{-1}) \
\bar e_0(kx) e_0(kx  \ {w'}^{-1}) \
\bar u_0(x_\fz)
u_0( (x \ k^{-1} {w'}^{-1})_{\fz})
\\&=&
k^{Q'}|\chi|^2(x) |u_0|^2(x_\fz)
\bar e_0(kx) e_0(kx  \ {w'}^{-1})
+ \varepsilon_{k}(x,w').
\end{eqnarray*}

We claim that for any $\tau\in L^\infty (\Gh)$ 
 such that $\cF_G^{-1}\tau $ is a compactly supported distribution on $G$,
we have for any $x\in G$
\begin{equation}
\label{claim_pf_lem_oscillation}
\int_{\Gh }\int_G
\tr \pi(y) \tau(\pi) e_0(xy^{-1}) dy d\mu(\pi)
=(\pi_0(x)\tau(\pi_0) v_0,v_0)_{\cH_{\pi_0}}.
\end{equation}
Indeed by the Fourier inversion formula, 
the limit is equal to 
$$
\int_G \cF_G^{-1}\tau (y)\ e_0(xy^{-1})  dy
$$
interpreted in the sense of a compactly supported distribution 
at a smooth bounded function,
and this is equal to the right hand side of \eqref{claim_pf_lem_oscillation}.
We can apply this to $\tau=\{\sigma(x,\pi),\pi\in \Gh\}$ 
since $\cF^{-1}\sigma(x,\cdot)$ is the convolution kernel of $A$ 
which is compactly supported (as the integral kernel of $A$ is compctly supported).
Hence the claim in \eqref{claim_pf_lem_oscillation} is proved 
and we may apply it to obtain:
$$
(Au_k,u_k)_{L^2(G)} = T(k) +\tilde\varepsilon(k)
$$
where 
\begin{eqnarray*}
T(k)&:=&
\int_G k^{Q'} |\chi|^2(x)|u_0|^2(x_\fz) \bar e_0(kx)
(\pi_0(kx)\sigma(x,k\cdot\pi_0) v_0,v_0)_{\cH_{\pi_0}}
dx,\\
 \tilde\varepsilon(k)&:=&
\lim_{R\to+\infty} 
\int_{(k ^{-1} R)\cdot\cC}\int_{G} 
\int_G
\tr \left(\pi'(w')\sigma(x,k\cdot\pi')
 \varepsilon_{k}(x,w')\right)
dxdw' d\mu(\pi').
\end{eqnarray*}

Let us show that $\tilde\varepsilon(k)$  tends to zero 
as $k\to\infty$.
It is easy to see that in the Sobolev space $L^2_s$ for any $s>Q/2$,
we have the uniform convergence:
$$
\sup_{x\in G}
\|\varepsilon_k(x,\cdot)\|_{L^2_s(G)}\longrightarrow_{k\to\infty} 0.
$$
From Section \ref{subsec_Bessel+FIF}, 
we have
$$
\forall \phi \in L^2_s(G), \ \tau\in L^\infty(\Gh)\qquad
\left|\int_{\Gh} \tr \left(\tau(\pi) \widehat \phi(\pi)\right) d\mu(\pi)\right|
\leq 
C_s \|\tau\|_{L^\infty(\Gh)} \|\phi\|_{L^2_s(G)}.
$$
From the two properties above, we obtain easily 
$$
| \tilde\varepsilon(k)|
\leq \int_G \|\sigma(x,\cdot)\|_{L^\infty(\Gh)}
\|\varepsilon_k (x,\cdot)\|_{L^2_s(G)} dx
\longrightarrow_{k\to\infty} 0,
$$
as the integrand has compact support in $x\in G$.

For $T(k)$, 
as we have 
\begin{eqnarray*}
\bar e_0(kx)
\pi_0(kx)
&=&
e^{-i\lambda_0(kx_\fz)} \bar e_0(kx')
\ e^{i\lambda_0(kx_\fz)} \pi_0(kx')
\\&=&
\bar e_0(kx') \pi_0(kx'),
\end{eqnarray*}
 the change of variable $x''=kx'$
whose Jacobian is $k^{-Q'}$ yields:
\begin{eqnarray*}
T(k)
&=&
\int_{G'}\int_{\bR^{n_\fz}} k^{Q'} |\chi|^2(x)|u_0|^2(x_\fz) \bar e_0(kx')
(\pi_0(kx')\sigma(x_\fz x',k\cdot\pi_0) v_0,v_0)_{\cH_{\pi_0}}
dx_\fz dx'\\
&=&
\int_{G'}\int_{\bR^{n_\fz}} |\chi|^2(x_\fz k^{-1} x'')|u_0|^2(x_\fz) \bar e_0(x'')
(\pi_0(x'')\sigma(x_\fz k^{-1} x'',k\cdot\pi_0) v_0,v_0)_{\cH_{\pi_0}}
dx_\fz dx''.
\end{eqnarray*}
We claim that, by the Lebesgue theorem,
 this  converges towards
\begin{equation}
T(k) \underset{k\to\infty}\longrightarrow
 \int_{G'}\int_{\bR^{n_\fz}} |\chi|^2(x_\fz)|u_0|^2(x_\fz) \bar e_0(x'')
(\pi_0(x'')\sigma_0(x_\fz,\pi_0) v_0,v_0)_{\cH_{\pi_0}}
dx_\fz dx'',
\label{eq_lem_oscillation_pf_cvT1}
\end{equation}
where $\sigma_0=\princ (A)$.
Indeed
we  fix a positive Rockland operator $\cR$ (of homogeneous degree~$\nu$)
and  $\psi\in C^\infty(\bR)$  a smooth function such that 
$\psi\equiv0$ on a neighbourhood of 0 
and $\psi\equiv 1$ on $[\Lambda,\infty)$. 
We know that the symbol
$$
\rho:=\sigma -\sigma_0\psi(\pi(\cR)),
$$
is in $S^{m_1}$ with $m_1<0$.
We may write
$$
\left(\rho (x,k\cdot\pi_0) v_0,v_0\right)_{\cH_{\pi_0} }
=
 \left( \tilde \rho_{k,x} v_k, v_0\right)_{\cH_{\pi_0} }
$$
where
$\displaystyle{
\tilde\rho_{k,x}=
  \rho (x,k\cdot\pi_0) k\cdot\pi_0 (\id+\cR)^{- \frac{m_1}\nu}}$ and
$\displaystyle{
v_k:=k\cdot\pi_0 (\id+\cR)^{\frac{m_1}\nu} v_0.
}$
The operator $\tilde \rho_{k,x}$ is uniformly bounded:
$$
\| \tilde\rho_{k,x}
\|_{\sL(\cH_{\pi_0})}  
\leq
\sup_{x\in G, \pi_1\in \Gh}
\| \rho (x,\pi_1) \pi_1 (\id+\cR)^{- \frac{m_1}\nu}
\|_{\sL(\cH_{\pi_0})}  
\leq
\| \rho\|_{S^{m_1},0,0, m_1},
$$
and 
so is the vector $v_k$:
$$
\|v_k\|_{\cH_0} \leq \left\|k\cdot\pi _0(\id+\cR)^{\frac{m_1}\nu}
\right\|_{\sL(\cH_{\pi_0})}   \|v_0\|_{\cH_0} 
\leq \sup_{\lambda>\lambda_{min}(\pi_0)} (1+k^\nu \lambda)^{\frac{m_1}\nu}\leq 1.
$$
Here $\lambda_{min}(\pi_0)$ is the smallest eigenvalue of $\pi_0(\cR)$, 
see Lemma \ref{lem_normL2_E(epsilon,infty)} (2),
so $\lambda_{min}(\pi_0)\in (0,\infty)$ and $\|v_k\|_{\cH_0}$ tends to 0 as $k\to\infty$.
It is now a routine exercise left to the reader 
to apply the Lebesgue theorem
and obtain the convergence in \eqref{eq_lem_oscillation_pf_cvT1}.

As $A$ is compactly supported in $\{\chi=1\}\times\{\chi=1\}$, 
we may  assume that $\sigma_0$ is compactly supported in $x$, 
and that this support is included in $\{\chi=1\}$.
Hence we have obtain that the $(Au_k,u_k)_{L^2(G)}$ has the same limit as $T(k)$ which can be rewritten as
$$\frac 1{d_{\pi_0}}
 \left( \Big( \int_G \princ_0(A)(x_\fz , \pi_0) \
 |u_0(x_\fz)|^2 dx_\fz \Big) v_0 \, , \, v_0\right)_{\cH_{\pi_0}}.$$
  \end{proof}
  
\subsection{Example: general oscillations}
\label{sec_ex3}

In Section \ref{sec_ex2}, we constructed a pure sequence associated 
with a square integrable representation.
In this section, we generalise the idea to any irreducible representation of $G$. 
If the representation is finite dimensional, then it is of dimension 1 (see~\cite{corwingreenleaf})  and we may proceed as in the Euclidean case. 
Let us consider $\pi_0$ an irreducible representation of infinite dimension.
We will replace the properties of square integrability 
with the general results on representations of nilpotent Lie groups 
due to Pedersen \cite{pedersen}.

Unfortunately, the notation adapted to the presence of dilations is in conflict with the conventional notation for Jordan-H\"older bases. Indeed, our canonical basis $X_1,\ldots, X_n$ of $\fg$, that is, a basis adapted to the gradation (see Section \ref{subsec_G}), is adapted to the  Jordan-H\"older sequence: 
$$
\fg =: \fg_n \subset \fg_{n-1} \subset \ldots \subset \fg_1 \subset \fg_0 :=\{0\},
\quad \mbox{where}\quad 
\fg_k := \bR X_{n-k} \oplus\ldots \oplus \bR X_n, \  k=1,\ldots, n-1,
$$
except for the order of the labels in the basis; for instance $X_k\in \fg_{n-k}$.
We denote by $J$ the set of jump indices of $\pi_0$:
$$
J:=\{1\leq  k \leq n \ : \ d\pi_0 (X_k) \notin d\pi_0 (\fg_{n-k-1})\}.
$$
We observe that the set of jump indices is the same for $r\cdot \pi_0$, 
for any $r>0$.
We set
$$
\fg_J:=\oplus_{k\in J} \bR X_k
\qquad\mbox{and}\qquad
\fg_{J^c}:=\oplus_{k\notin J} \bR X_k.
$$
The natural Haar measures on $\fg_J$ and $\fg_{J^c}$ 
are $\Pi_{k\in J} dx_j$ and $\Pi_{k\notin J} dx_j$ respectively;
we will denote them by $dx$, or any other letter representing the variable of integration.

For any Schwartz function $\phi\in \cS(\fg_J)$ on the vector space $\fg_J$, we define:
$$
\tilde \phi:=
 \int_{x=(x_k)_{k\in J}\in\fg_J} \phi (x)\ 
 \pi_0\left(\exp \left(\sum_{k\in J} x_j X_j)\right)\right)
dx
$$
This is a smooth operator on $\cH_{\pi_0}$, i.e. $\tilde \phi\in \sL(\cH_{\pi_0})_\infty$.
Moreover (cf. \cite{pedersen}), $\phi\mapsto \tilde \phi$ is an isomorphism between the Fr\'echet spaces $\cS(\fg_J)$ and $\sL(\cH_{\pi_0})_\infty$;
its inverse is given by 
\begin{equation}
\label{eq_def_fA}
\sL(\cH_{\pi_0})_\infty\ni A\mapsto {f_A\circ \exp}\big|_{\fg_e},
\quad \mbox{where}\quad
f_A (x) :=\tr 
 \left( \pi_0(x) A\right), \quad x\in G.
\end{equation}

Any element in  $ \sL(\cH_{\pi_0})_\infty$ is trace-class.
An example of an element of $\sL(\cH_{\pi_0})_\infty$ is $v\otimes w^*$ where $v$ and $w$ are two smooth vectors of $\cH_{\pi_0}$.
For any $\phi\in \cS(\fg_J)$ the operator $\tilde \phi$ is traceclass with  \cite{pedersen}
\begin{equation}
\label{eq_tracetildephi}
\tr \tilde \phi =  \frac{1}{d_{\pi_0}} \phi(0).
\end{equation}
Here $d_{\pi_0}>0$ is  a computable constant 
depending on $\pi_0$ and on the choice of Jordan-H\"older basis, but not on $\phi$.

\smallskip

We can now state and prove the following generalisation of Proposition \ref{prop_oscillation}:

\begin{proposition}
\label{prop_oscillation_general}
Let $\pi_0$ be an irreducible representation of $G$ of infinite dimension.
We define its jump set $J$ and the subspaces $\fg_J,\fg_{J^c}$ of $\fg$ 
as above. We set
$$
Q_J:=\sum_{k\in J} \upsilon_k.
$$
Let $u_0\in \cS(\fg_{J^c})$. Let $A\in \sL(\cH_{\pi_0})_\infty$
and define $f_{A}$ as in \eqref{eq_def_fA}.
For each $k\in \bN$, let  $u_k:G\to \bC$ be the  function given by
$$
u_k(x)=k ^{{Q_J}\over 2}  f_{A}(kx)\  u_0(x_{J^c}).
$$
Then $\|u_k\|_{L^2} = d_{\pi_0}^{-1/2} \|A\|_{HS(\cH_{\pi_0})} \|u_0\|_{L^2(\fg_{J^c})}$ and $u_k\TendWeak{k}{\infty} 0$ in $L^2(G,loc)$. This sequence is pure and its MDM is given by 
$$\gamma(x,\dot\pi)= \left(\frac {|u_0(x_{J^c})|^2}{d_{\pi_0}} dx_{J^c} \otimes \delta_{0}(x_{J})\right)  
\otimes \delta_{\dot\pi_0}(\dot\pi),
\qquad \Gamma(\dot \pi_0)=A A^*.
$$
\end{proposition}
In the statement, we have used the following notation:
$$
\mbox{if} \
x=\exp (\sum_{k=1}^n x_j X_j)\in G, 
\ \mbox{then}\ x_{J^c}:=(x_j)_{j\notin J^c}\in \fg_{J^c}, \ 
x_{J}:=(x_j)_{j\in J} \in \fg_{J}.
$$  
In the proof, we will use the following properties:

\begin{lemma}
\label{lem_prop_oscillation_general}
Let $\pi_0$ be an irreducible representation of infinite dimension of $G$.
We define as above its jump set $J$, the subspaces $\fg_J, \fg_{J^c}$.
\begin{enumerate}
\item 
There exists a linear function $F:\fg_J\times \fg_{J^c}\to\fg_J$ such that 
$$
\pi_0(x) = \pi_0\left( \exp \left( X_J + F(X_J,X_{J^c})\right)\right),
$$
where
$\displaystyle{
 x=\exp (\sum_{k=1}^n x_j X_j)\in G}$,
$\displaystyle{
X_J =\sum_{k\in J}x_j X_j \in \fg_J}$,
$\displaystyle{
X_{J^c} =\sum_{k\notin J}x_j X_j \in \fg_{J^c}
}$.\\
Furthermore, for any $X_{J^c}$, the change of variable $X_J\mapsto X_J'=X_J + F(X_J,X_{J^c})$ is a diffeomorphism of $\fg_J$ with determinant 1.

\item 
Let $A\in \sL(\cH_{\pi_0})_\infty$
and define $f_{A}$ as in \eqref{eq_def_fA}.
We have
$$
 \int_{(x_k)_{k\in J}\in\fg_J}  f_A(x) \ 
  \pi_0\left(\exp \left(\sum_{k\in J} x_j X_j\right)\right)^*
dx
=
 \frac{1}{d_{\pi_0}} A.
 $$
\item 
Let $\sigma\in L^\infty(\Gh)$ be such that $\cF_{G}^{-1}\sigma$ is  a compactly supported distribution on $G$.
Then 
$$
\int_G \int_{\Gh}\tr\left(\sigma(\pi) \ \pi(w) \right) f_A(w^{-1}x)d\mu(\pi)
 dw = \tr\left(  \kappa(\pi_0) \pi_0(x)  A\right),
 $$
 interpreting the left hand side as in Proposition \ref{prop_FIF2}, that is, as the limits (in this order) of the absolutely convergent double integral: 
$$
\lim_{R\to \infty} \lim_{N\to +\infty}\int_{N\cdot \cC} \int_G \trN \left(\sigma (\pi)\pi(w)\right) f_A(w^{-1}x) \chi_R(w) \ dw  d\mu(\pi),
$$
where $\chi\in \cD(G)$ with $\chi\equiv 1$ on a neighbourhood of 0
and $\chi_R(x):=\chi (R^{-1}x)$, $\cC$ a compact neighbourhood of $1\in \Gh$ such that $\cup_{N\in \bN} N\cdot \cC=\Gh$, 
and $\trN$ denotes the trace of the operators projected on the subspace spanned by the first $N$ vectors, having fixed a fundamental sequence of vector fields.
\end{enumerate}
\end{lemma}

\begin{proof}
Part (1)  is a simple consequence of  the definition of a jump set, it  is left to the reader. Part (2) is in \cite{pedersen}.
 For Part (3), we apply  Proposition \ref{prop_FIF2} to obtain:
\begin{align*}
&\int_G \int_{\Gh}\tr\left(\sigma(\pi) \ \pi(w) \right) f_A(w^{-1}x)d\mu(\pi)
 dw = \int_G  \cF_{G}^{-1}\sigma(w)\ f_A (w^{-1} x) dw\\
&\qquad =  \int_G \tr \left( \cF_{G}^{-1}\sigma (w) \pi_0 (w)^*\pi_0(x) \ A   \right) dw 
=  \tr\left( \sigma(\pi_0) \pi_0(x)   A\right),
\end{align*}
since $A$ is traceclass and $\sigma(\pi_0)\in \sL(\cH_{\pi_0})$.
\end{proof}

The arguments to show Proposition \ref{prop_oscillation_general} 
follow the ones for Proposition \ref{prop_oscillation}.
The main modifications come from 
replacing  the properties of the centre with Lemma \ref{lem_prop_oscillation_general} Part (1).
We will only outline the ideas, the technical details being very similar to the ones in the proof of Proposition~\ref{prop_oscillation}.

\begin{proof}[Sketch of the proof of Proposition \ref{prop_oscillation_general}]
Let $\chi\in \cD(G)$ and $\sigma\in S^0_{cl}(G)$.
\begin{align*}
&(\Op(\sigma)(\chi u_k) ,\chi u_k)_{L^2(G)}
=
k ^{{Q_J}}  \int_{\Gh} \int_G \int_G 
\tr \left(\sigma(x,\pi) \pi(w) \right)
(\chi u_k) (xw^{-1}) \ \overline{\chi u_k}(x)  dw dx d\mu(\pi)
\\
&\qquad=
k ^{{Q_J}}  \int_{\Gh} \int_G \int_G 
\tr \left(\sigma(x,k\cdot \pi) \pi(w) \right)
(\chi u_k) (x\ k^{-1}w^{-1}) \ \overline{\chi u_k}(x) \  dw dx d\mu(\pi),
\end{align*}
after the change of variable $(\pi,w)\mapsto (k\cdot \pi, k^{-1} w)$.
For $k$ large,
$$ 
(\chi u_k) (x\ k^{-1}w^{-1}) \sim (\chi u_0) (x) \ f_A((kx) w^{-1}).
$$
 Lemma \ref{lem_prop_oscillation_general} Part (3) implies
$$
\int_{\Gh} \int_G 
\tr \left(\sigma(x,k\cdot \pi) \pi(w) \right) f_A((kx) w^{-1})
 \  dwd\mu(\pi) 
 = \tr \left( \pi_0(kx) \sigma(x,k\pi_0) A\right).
$$
Let us define $u_0(x) = u_0((x_k)_{k\in J^c})$ when 
$x=\exp (\sum_{k=1}^n x_j X_j)\in G$.
Therefore
\begin{align*}
&(\Op(\sigma)(\chi u_k) ,\chi u_k)_{L^2(G)}
\sim 
k ^{{Q_J}} \int_G 
\tr \left( \pi_0(kx) \sigma(x,k\pi_0) A\right)
\overline{f_A(kx)} |\chi u_0|^2(x)
 dx\\
 &\quad=
 \int_{\fg_{J^c}}\int_{\fg_J}
\tr \left( \pi_0(e^{X_J + kX_{J^c}}) \sigma(e^{k^{-1}X_J + X_{J^c}},k\pi_0) A\right)
\overline{f_A(e^{X_J + kX_{J^c}})} |\chi u_0|^2(e^{k^{-1}X_J + X_{J^c}})
 dX_J dX_{J^c},
  \end{align*}
having written $x=\exp (X_J + X_{J^c})$ and then performed the change of variable $X_J\mapsto k^{-1} X_J$.
We have $k^{-1}X_J \to 0$, so
\begin{align*}
&
(\Op(\sigma)(\chi u_k) ,\chi u_k)_{L^2(G)}
\\
&\quad\sim 
 \int_{\fg_{J^c}}\int_{\fg_J}
\tr \left( \pi_0(e^{X_J + kX_{J^c}}) \sigma(e^{X_{J^c}},k\pi_0) A\right)
\overline{f_A(e^{X_J + kX_{J^c}})} |\chi u_0|^2(e^{X_{J^c}})
 dX_J dX_{J^c}
 \\
 &\qquad = 
 \int_{\fg_{J^c}}\int_{\fg_J}
\tr \left( \pi_0(e^{X'_J}) \sigma(e^{X_{J^c}},k\pi_0) A\right)
\overline{f_A(e^{X'_J})} |\chi u_0|^2(e^{X_{J^c}})
 dX'_J dX_{J^c},
  \end{align*}
after having used the change of variable  $X_J\mapsto X_J'=X_J + F(X_J,kX_{J^c})$, see Lemma \ref{lem_prop_oscillation_general} Part (1).
Applying 
 Lemma \ref{lem_prop_oscillation_general} Part (2)
 on  the integral over $\fg_{J}$ concludes the (sketch of the) proof.
\end{proof}  

In the next section, we will need the following limits which follow from similar computations to the ones above:

\begin{corollary}
\label{cor_prop_oscillation_general}

\begin{enumerate}
\item Let $\pi_0$, $A$, $u_0$ and $u_k$ as in Proposition \ref{prop_oscillation_general}.
If $x_0\not=x_1$ and $u_0$ has a compact support small enough then 
$$
\lim_{k\to \infty}
(\Op(\sigma)(\chi u_k(x_1 \cdot)) ,\chi u_k(x_0 \cdot))_{L^2(G)}=0.
$$

\item Let $\pi_0$ and $u_0$ as in Proposition \ref{prop_oscillation_general}
and $A$, $B$ be in $\sL(\cH_{\pi_0})_\infty$.
We construct $(u_k)$ and $(v_k)$ as in Proposition \ref{prop_oscillation_general} for  $A$ and $B$ respectively.
If $AB^*=0$ then
$$
\lim_{k\to \infty}
(\Op(\sigma)(\chi u_k) ,\chi v_k)_{L^2(G)}=0.
$$

\item Let $\pi_0$, $A$ and $u_0$ as in Proposition \ref{prop_oscillation_general}
and consider $\dot\pi_1\in \Gh$ with $\dot\pi_1\not=\dot\pi_0$, and
$v_0$ and $B$ in $\sL(\cH_{\pi_1})_\infty$.
We construct $(u_k)$ and $(v_k)$ as in Proposition \ref{prop_oscillation_general} for  $u_0,\pi_0,A$ 
and $v_0,\pi_1, B$ respectively.
Then 
$$
\lim_{k\to \infty}
(\Op(\sigma)(\chi (u_k+v_k)) ,\chi (u_k+v_k))_{L^2}=
\lim_{k\to \infty}(\Op(\sigma)(\chi u_k) ,\chi u_k)_{L^2}
+
\lim_{k\to \infty}(\Op(\sigma)(\chi v_k) ,\chi v_k)_{L^2}.
$$
\end{enumerate}
\end{corollary}

\begin{proof}[Proof of  Part (1)]
An argument of translation shows that it suffices to prove the case $x_1=0$.
Proceeding as in the proof of Proposition \ref{prop_oscillation_general}, 
we obtain:
$$
(\Op(\sigma)(\chi u_k) ,\chi u_k(x_0 \cdot))_{L^2(G)}
\sim 
k ^{{Q_J}} \int_G 
\tr \left( \pi_0(kx) \sigma(x,k\pi_0) A\right)
\overline{f_A(k(x_0x))} (\chi u_0)(x)\ \overline{\chi u_0}(x_0x)
 dx.
 $$
 Now $u_0(x)\bar u_0(x_0x)=0$ for any $x\in G$ 
  when $u_0$ as a support small enough and $x_0\not=0$.
\end{proof}

\begin{proof}[Proof of Part (2)]
Proceeding as in the proof of Proposition \ref{prop_oscillation_general}, 
we obtain:
\begin{align*}
&(\Op(\sigma)(\chi u_k) ,\chi v_k)_{L^2(G)}
\sim 
k ^{{Q_J}} \int_G 
\tr \left( \pi_0(kx) \sigma(x,k\pi_0) A\right)
\overline{f_B(kx)} |\chi u_0|^2(x)
 dx\\
 &\quad \sim 
 \int_{\fg_{J^c}}\int_{\fg_J}
\tr \left( \pi_0(e^{X'_J}) \sigma(e^{X_{J^c}},k\pi_0) A\right)
\overline{f_B(e^{X'_J})} |\chi u_0|^2(e^{X_{J^c}})
 dX'_J dX_{J^c}\\
 &\qquad=\int_{\fg_{J^c}}
\tr \left( \frac 1{d_{\pi_0}} B^* \sigma(e^{X_{J^c}},k\pi_0) A\right)
 |\chi u_0|^2(e^{X_{J^c}})
 dX'_J dX_{J^c}.
  \end{align*}
  Hence this is zero when $AB^*=0$.
\end{proof}

\begin{proof}[Proof of Part (3)]
Proceeding as in the proof of Proposition \ref{prop_oscillation_general}, 
we obtain:
\begin{align*}
&(\Op(\sigma)(\chi u_k) ,\chi v_k)_{L^2(G)}
\sim 
k ^{{Q_J}} \int_G 
\tr \left( \pi_0(kx) \sigma(x,k\pi_0) A\right)
\overline{f_B(kx)} |\chi u_0|^2(x)
 dx\\
 &\quad\sim 
 \int_{\fg_{J^c}}\int_{\fg_J}
\tr \left( \pi_0(e^{X_J + kX_{J^c}}) \sigma(e^{X_{J^c}},k\pi_0) A\right)
\overline{f_B(e^{X_J + kX_{J^c}})} |\chi u_0|^2(e^{X_{J^c}})
 dX_J dX_{J^c},
  \end{align*}
  having used the jump set for $\pi_0$.
  And this is equivalent to the same quantity with $\princ_0(\sigma)$ replacing $\sigma$. 
 So  when $\princ_0(\sigma)$ is zero at $(x,\dot \pi_0)$ for all $x\in G$, we have
  $$
\lim_{k\to \infty}
(\Op(\sigma)(\chi u_k) ,\chi v_k)_{L^2(G)}=0.
$$

Let us fix a continuous real-valued function on $\Gh/\bR^+$ such that $\eta(\dot\pi_0)=0$ and $\eta(\dot\pi_1)=1$.
Considering a general symbol $\sigma$,
 we write
 $\sigma=\sigma \eta + (1-\eta)\sigma$
 so 
 $$
\Re  (\Op(\sigma)(\chi u_k) ,\chi v_k)_{L^2}
  =
\Re  (\Op(\sigma \eta)(\chi u_k) ,\chi v_k)_{L^2}
+
\Re  ((\chi u_k) , (\Op(\sigma (1-\eta))^*  \chi v_k)_{L^2}.
  $$
As $\sigma \eta$ vanishes at $\dot \pi_0$, 
the limit of the first term on the right hand side is zero.
For the second term, we have as in the proof of Lemma \ref{lem_limit} or 
$$
\lim_{k\to \infty}
 ((\chi u_k) , (\Op(\sigma (1-\eta))^*  \chi v_k)_{L^2}
 =
\lim_{k\to \infty}
 ((\chi u_k) , \Op(\sigma^* (1-\eta))  \chi v_k)_{L^2}
 $$
 and it must be zero since 
  $\sigma^* (1-\eta)$ vanishes at $\dot \pi_1$.
\end{proof}

 \subsection{Consistency of the description}
 
Our main result describes MDMs as 
trace-class-valued positive measures, 
see  Section \ref{subsec_main_result}.
In this section, we will show the converse, 
that is, that any trace-class-valued positive measure is 
 a MDM:
 
 \begin{proposition}\label{prop_consistency}
 Let $\Omega$ be a non-empty open set of $G$.
  Let $(\Gamma,\gamma)$ be a trace-class-valued positive measure on $\Omega\times \Gh/\bR^+$.
 Then, there exists a pure sequence $(u_k)$ with $\Gamma d\gamma$ as MDM.
  \end{proposition}
    
As in  Section \ref{subsec_main_result}, 
an argument of diagonal extraction 
over a suitable sequence of bounded open subsets exhausting $\Omega$
shows that it suffices to prove:

     \begin{lemma}\label{lem_prop_consistency}
 Let $\Omega$ be a non-empty bounded open set of $G$.
 We fix a positive Rockland operator  $\cR$,
 and  a function
$\psi\in C^\infty(\bR)$ 
such that $\psi\equiv0$ on a neighbourhood of 0 and $\psi\equiv 1$ on a neighbourhood of $+\infty$.
  For any trace-class-valued positive measure $(\Gamma,\gamma)$  on $\bar\Omega\times \Gh/\bR^+$,  there exists a sequence $(u_k)$  in $L^2(\Omega)$ such that 
   $u_k \rightharpoonup_{k\to\infty}  0$ and 
   \begin{equation}
\label{eq_lem_prop_consistency}
\forall \sigma\in \dot S^0(\Omega),\qquad
\lim_{k\to\infty} 
\left(\Op(\sigma \psi(\pi(\cR))) u_k,  u_k \right)_{L^2(\Omega)} 
=
\int_{\bar \Omega \times (\Gh /\bR^{+})}
\tr \left(\sigma (x,\dot \pi) \ \Gamma(x,\dot \pi) \right)
d  \gamma(x,\dot\pi) \, .
\end{equation}
  \end{lemma}
    
We will need the following vocabulary:
\begin{definition}
\label{def_traceclass_meas}
Let $\Omega$ be a Borelian set of $G$ with non-empty interior.
\begin{itemize}
\item 
A \emph{trace-class-valued measure} on $\Omega\times \Gh/\bR^+$
is a pair $(\gamma,\Gamma)$
where
 $\gamma$ is a complex measure on $\Omega\times \Gh/\bR^+$,
and  $\{\Gamma (x,\dot\pi)\in \sL(\cH_\pi), (x,\dot\pi)\in \Omega\times\Gh/\bR^+\}$
is a measurable field of  operators 
such that $(|\gamma|, \Gamma)$ is a positive trace-class-valued measure,
see Definition \ref{def_traceclass_pos_meas}.
\item 
Two trace-class-valued measures $(\gamma,\Gamma)$
and $(\gamma',\Gamma')$ are \emph{equivalent} when
there exists a measurable positive function $f$ on $\Omega\times\Gh/\bR^+$ such that 
$$
d\gamma'(x,\dot \pi)=f(x,\dot \pi) d\gamma(x,\dot \pi)
\quad\mbox{and}\quad
\Gamma'(x,\dot \pi) = \frac1{f(x,\dot \pi)} \Gamma(x,\dot \pi).
$$ 
The equivalence class of $(\gamma,\Gamma)$ is denoted by $\Gamma d\gamma$. We may sometimes allow ourselves to identify an equivalence class with one of its representatives. 
\item 
A trace-class-valued measure on $\Omega\times \Gh/\bR^+$
$(\gamma,\Gamma)$ is \emph{positive} when $\gamma$ is positive
and in  this case, we may write $\Gamma d\gamma \geq 0$.
\end{itemize}
\end{definition}

The following lemma describes the topological dual of $C^*(\dot S^0(\Omega))$  
when $\Omega$ is bounded:

\begin{lemma}
\label{lem_dual_C*dotS0Omega}
Let $\Omega$ be a non-empty bounded open set of $G$.
We denote by 
$\cM$ the set of trace-class-valued measures on 
$\bar \Omega\times \Gh/\bR^+$
modulo equivalence, 
and by $\cM^+=\{\Gamma d\gamma \geq 0\}$ the set of positive trace-class-valued measures on 
$\bar \Omega\times \Gh/\bR^+$
modulo equivalence. 

\begin{enumerate}
\item 
For any trace-class-valued measure $(\gamma,\Gamma)$ on $\bar \Omega\times \Gh/\bR^+$, 
the linear form $\ell_{(\gamma,\Gamma)}$ given by 
$$
\ell_{(\gamma,\Gamma)}(\sigma):=
\int \tr (\sigma \ \Gamma)\  d\gamma 
=\int_{(x,\dot\pi)\in\bar \Omega\times \Gh/\bR^+} 
\!\!\!\!\!\!\!\!\!\!\!\!\!\!\!\!
\tr (\sigma(x,\dot\pi) \ \Gamma(x,\dot\pi))\  d\gamma (x,\dot\pi), 
\quad \sigma\in C^*(\dot S^0(\Omega)),
$$
is continuous on $C^*(\dot S^0(\Omega))$; 
its norm is 
$\int \tr (\Gamma) d|\gamma|$.

\item 
Conversely, given any continuous form on $C^*(\dot S^0(\Omega))$, 
there exists a trace-class-valued measure $(\gamma,\Gamma)$ on $\bar \Omega\times \Gh/\bR^+$
such that $\ell=\ell_{(\gamma,\Gamma)}$.

\item\label{item_lem_dual_C*dotS0Omega_eq}
If two trace-class-valued measures 
$(\gamma,\Gamma)$ and $(\gamma',\Gamma')$
 yield the same linear form, i.e.
 $\ell_{(\gamma,\Gamma)}=\ell_{(\gamma',\Gamma')}$, 
 then they are equivalent.

\item 
The map $\Phi:\ell =\ell_{(\gamma,\Gamma)}\mapsto \Gamma d\gamma $
is an isomorphism from
the topological dual of the Banach space $C^*(\dot S^0(\Omega))$ 
onto the Banach space $\cM$ which is equipped with the norm given by
$$
\|\Gamma d\gamma\|_{\cM}
:= \int \tr (\Gamma) d|\gamma|.
$$

\item\label{item_lem_dual_C*dotS0Omega_purestate} 
The states of the $C^*$-algebra $C^*(\dot S^0(\Omega))$  
are mapped by $\Phi$ onto the measures $\Gamma d\gamma\geq 0$
with $\int \tr (\Gamma) d\gamma=1$.
The pure states corresponds to 
$(\delta_{x_0}(x)\otimes \delta_{\dot\pi_0}(\pi), v_0\otimes v_0^*)$
where $x_0\in \bar\Omega$, $\dot \pi_0\in \Gh/\bR^+$
and $v_0$ a unit vector in $\cH_{\pi_0}$.

\item 
The positive forms of the $C^*$-algebra $C^*(\dot S^0(\Omega))$  
are mapped by $\Phi$ onto $\cM^+$.
\end{enumerate}
\end{lemma}   
    
  \begin{proof}
  The  states  were characterised in Proposition \ref{prop_state_C*dotS0Omega}.
The properties are easily proved  from well-known facts or routine exercises in functional analysis.
\end{proof}
  
We can now prove Lemma \ref{lem_prop_consistency}.
  
\begin{proof}[Proof of Lemma \ref{lem_prop_consistency}]
We denote by $\tilde \cM$ the subset of $\cM^+$
of $\Gamma d\gamma\geq 0$ 
for which there exists a sequence $(u_k)$  in $L^2(\Omega)$ 
satisfying
   $u_k \rightharpoonup_{k\to\infty}  0$ and 
 \eqref{eq_lem_prop_consistency}.
 We already know that $\tilde \cM$ contains 0, 
 and the examples in Propositions \ref{prop_spatial_concentration}
 and \ref{prop_oscillation_general}.
 
 \medskip

\textit{Claim 1:} Let us show that $\tilde \cM$ is convex.
Indeed, one checks easily that if the sequence $(u_k)$  in $L^2(\Omega)$ 
satisfies
   $u_k \rightharpoonup_{k\to\infty}  0$ and 
 \eqref{eq_lem_prop_consistency} with $\Gamma d\gamma\geq 0$, 
 then for any $r>0$  the sequence $(r u_k)$ satisfies the same property with 
$r^2\Gamma d\gamma\geq 0$.

\medskip

\textit{Claim 2:} 
One checks easily that if $x_0\in G$ and if the sequence $(u_k)$  in $L^2(\Omega)$ 
satisfies
   $u_k \rightharpoonup_{k\to\infty}  0$ and 
 \eqref{eq_lem_prop_consistency} with $\Gamma d\gamma\geq 0$, 
 then  the sequence $(u_k(x_0\, \cdot ))$ satisfies the analogous properties with 
$\Gamma(x_0 x,\dot \pi) d\gamma(x_0 x,\dot \pi)$ when this is in $\tilde \cM$.
In this sense,  $\tilde \cM$ is invariant under spatial translations.

\medskip

Lemma \ref{lem_dual_C*dotS0Omega} allows us to identify $\cM$ with the topological dual of $C^*(\dot S^0(\Omega))$;
we now equip it with the weak-* topology.
By the Krein-Milman Theorem, 
$\cM^+$ is the positive span of the pure states and 0 (i.e. the closure of the set of all non-negative linear combinations of pure states).

\medskip

\textit{Claim 3:} Let us show that $\tilde \cM$ is closed in $\cM^+$.
Indeed, let  $(\Gamma^{(j)} d\gamma^{(j)})_{j\in \bN}$ be a sequence in $\tilde \cM$
converging to $\Gamma d\gamma$ in $\cM^+$.
Considering corresponding sequences  $(u_k^{(j)})_{k\in \bN}$   in $L^2(\Omega)$ 
satisfying
   $u_k^{(j)} \rightharpoonup_{k\to\infty}  0$ and 
 \eqref{eq_lem_prop_consistency} with $\Gamma^{(j)} d\gamma^{(j)}\geq 0$, 
 then we extract a diagonal subsequence  $(u_{k(j)}^{(k(j))})_{j\in \bN}$
 satisfying $u_{k(j)}^{(k(j))} \rightharpoonup_{j\to\infty}  0$.
By Proposition \ref{prop_defect_measure}, 
we may assume that this subsequence satisfies \eqref{eq_lem_prop_consistency} for a certain positive trace-class-valued measure which has to coincide with $\Gamma d\gamma$
by the uniqueness properties in Proposition \ref{prop_defect_measure}
and Lemma \ref{lem_dual_C*dotS0Omega} \eqref{item_lem_dual_C*dotS0Omega_eq}.
Hence the limit $\Gamma d\gamma$ is in $\tilde \cM$
which is thus closed.

\textit{Conclusion:}
Considering a sequence as in Proposition
\ref{prop_oscillation_general} 
with 
$\pi_0\in \Gh$  of infinite dimension,
$A\in \sL(\cH_{\pi_0})_\infty$,
$u_0^{(\epsilon)} = \epsilon^{-Q_{J^c}/2} u_0(\epsilon^{-1} x)$
where $\epsilon>0$, $u_0\in \cD(\fg_{J^c})$ with a support small enough and $Q_{J^c}=\sum_{k\in J^c} \upsilon_k$, 
the proof of Claim 3 shows that 
$(\delta_{x=0}\otimes \delta_{\dot\pi=\pi_0},AA^*) \in \tilde \cM$
when $0\in \Omega$.
Using the invariance under spatial translation (cf. Claim 2), 
$(\delta_{x=x_0}\otimes \delta_{\dot\pi=\dot\pi_0},AA^*) \in \tilde \cM$
for any $x_0\in \bar\Omega$ and $\pi_0\in \Gh$ of infinite dimension.
Note that this membership also holds when $\pi_0$ is of finite dimension (therefore of dimension one, and it suffices to adapt the Euclidean case),  and we  view it as a degenerate case of Proposition
\ref{prop_oscillation_general}.
We choose $A=v_0\otimes v_0^*$ with $v_0\in \cH_{\pi_0}$ smooth and unitary. We can remove the hypothesis `smooth' by considering a sequence of such vectors and Claim 3.
This shows that $\tilde \cM$ contains all the  pure states, see  Lemma \ref{lem_dual_C*dotS0Omega} \eqref{item_lem_dual_C*dotS0Omega_purestate}.
Moreover, they can all be obtained as MDM of sequences obtained by diagonal extractions of suitable sequences constructed in  Proposition
\ref{prop_oscillation_general}. 
This together with 
Corollary \ref{cor_prop_oscillation_general} and Claim 1 show
that $\tilde \cM$ also contains the positive span of the pure states and 0.
Therefore $\tilde \cM=\cM^+$.
\end{proof}

\section{Applications}\label{sec_app}

In this section, we investigate the properties of the MDM of a sequence of functions that satisfy a differential equation. In particular, we are concerned with {\it Div-Curl} type results and, as a consequence, we shall focus on vector-valued sequences. 

\medskip

Let $\Omega$ be an open subset of $G$ and let us consider a vector-valued sequence of functions of $L^2(\Omega)$, $(U_k)_{k\in \bN}=(u_1^k,\cdots,u_N^k)_{k\in\bN}$, $N\in\bN$. We assume that $(U_k)$ converges weakly to some vector valued function $U=(u_1,\cdots,u_N)$ of $L^2(\Omega)^N$, in the sense that for all $j\in \{1,\cdots,N\}$, $u_j^k$ tends weakly to $u_j$ in $L^2(G)$.  In order to study the defects of compactness of a family of the form $(U_k)_{k\in \bN}$, we shall use matrices of symbols in $\dot S^0$. We denote by 
${\mathcal M}_N(\dot S^0)$ the set of such matrices with $N$ lines and $N$ rows. More generally, we denote by 
$\cM_{N}(A)$ the set of matrices with $N$ lines and $N$ rows and with entries in a given algebra $A$, for instance $A=\bC$ or $\dot S^0$ or $C^*(\dot S^0)$.
We shall need basic notions about the $C^*$-algebra ${\mathcal M}_N(A)$ for a general $C^*$-algebra $A$, and this is done in the first subsection. Then, we shall define MDM for vector-valued sequences and discuss localisation property of MDM whenever $(U_k)_{k\in\bN}$ satisfies a system of differential equations. Finally, the last subsection is devoted to compensated compactness results and application to {\it Div-Curl} Lemma.

\subsection{Matrices of a $C^*$-algebra}

Let $A$ be an algebra with unit $1_A$ and let $N\in \bN$.
We have already defined the algebra ${\mathcal M}_N(A)$  of the $N\times N$ matrices with coefficients in $A$. 
We need to set some natural notation.
We denote   by $\id_N\in \cM_N(\bC)$ the identity matrix and by $A \id_N$ the set of diagonal matrices in $M_N(A)$ with the same repeated entry on the diagonal,  by $E_{ij}$ the $N\times N$ complex matrix with 0 in every entry except for the $i$th row and $j$th column where the entry is 1, and by $1_A E_{ij}\in M_N(A)$ the matrix  with 0 in every entry except for the $i$th row and $k$th column where the entry is $1_A$. 
Finally, we denote by $\cM_{P,Q}(A)$
the set of matrices with $P$ lines and $Q$ rows and with entries in $A$. 
If the algebra $A$ is also a normed vector space, 
we set the following norm on $\cM_{N,1}(A)$:
$$ 
\|V\|_{\cM_{N,1}(A)}^2 = \| \sum_{j=1}^N v_jv_j^*\|_A
\quad\mbox{when}\quad
V=\left(\begin{array}{c} 
v_1 \\ \vdots \\v_N\end{array}\right).
$$
The next lemma gives the main properties of $\cM_N(A)$ 
 when $A$ is a $C^*$-algebra.

\begin{lemma}
\label{lem_MN(A)}
Let $A$ be a $C^*$ algebra with unit $1_A$ and let $N\in \bN$.
\begin{itemize}
\item[(i)] 
Equipped with the norm given by 
$$
\|M\|_{\cM_N(A)}=\sup\{\| V_1^* M V_2\|_A : V_1,V_2\in \cM_{N,1}(A), 
\|V_1\|_{\cM_{N,1}(A)}\leq 1,
\|V_2\|_{\cM_{N,1}(A)}\leq 1\},
$$
$\cM_N(A)$ is a $C^*$-algebra with unit $1_A\id_N$.
The sub-$C^*$-algebra $A \id_N$  of $\cM_N(A)$ is isomorphic to the $C^*$-algebra~$A$.

\item[(ii)] 
Let $\pi$ be a representation of the $C^*$-algebra $\cM_N(A)$ .
Let $\xi_1$ be a non-zero vector of this representation
with $\xi_1\in \pi(E_{11})\cH_\pi$. 
We denote by $W$ the closed subspace of $\cH_\pi$ generated by $\xi_1$. 
As Hilbert spaces, $W$ is isomorphic to the orthogonal sum of $N$ copies of $\pi(A \id_N)\xi_1$.
Furthermore the representation $\pi$ of the $C^*$-algebra 
$\cM_N(A)$ on $W$ is completely determined by its restriction  $\pi\big|_{A \id_N}$ to $A \id_N$.
\item[(iii)] 
The spectrum of the $C^*$-algebra 
$\cM_N(A)$ may be identified with the spectrum of the $C^*$-algebra 
$A$ via the homeomorphism which maps an irreducible representation $\pi$ of $\cM_N(A)$ to the irreducible representation of $A$ defined by the restriction $\pi\big|_{A \id_N}$.
\item[(iv)] 
If $\ell$ is a state of $A$ and $V\in \cM_{N,1}(A)$ with $\ell(V^*V)=1$,
then the functional $L=L_{\ell,V}$ defined on $\cM_N(A)$ via
$$
L(M)=\ell(V^*MV), \quad M\in \cM_N(A), 
$$
is a state of $\cM_N(A)$.
The pure states of $\cM_N(A)$ are of the form $L_{\ell,v}$ 
with $\ell$ a  pure state  of $A$ and  $V\in \cM_{N,1}(\bC)1_A$ 
a complex vector satisfying  $\ell(V^*V)=1$.
\end{itemize}
\end{lemma}

\begin{proof}
Part (i) is left to the reader. Let us prove Part (ii). Let $\pi$ be a representation of the $C^*$-algebra 
$\cM_N(A)$. 
For each $j=1, \ldots, N$, 
we set $\cH_\pi^{(j)} :=\pi(1_A E_{jj})\cH_\pi$.
The subspaces $\cH_\pi^{(j)}$ are closed, orthogonal and their sum is 
$\cH_\pi=\oplus^\perp_{1\leq j\leq N}  \cH_\pi^{(j)}$
since $E_{ii}E_{jj}=\delta_{i=j}E_{ii}$ and $\id_N=\sum_{j=1}^N E_{jj}$ in $\cM_N(\bC)$.
Furthermore, since $\pi(m\id_N)$ and $\pi(1_A E_{jj})$ commutes for $j$ fixed and any $m\in A$, 
the algebra $A$ acts on $\cH_\pi^{(j)}$ via $m\mapsto \pi(m\id_N)$.
We also observe that for any $1\leq i,j \leq N$, $\pi(1_A E_{ji})$ maps $\cH_\pi^{(i)}$ to $\cH_\pi^{(j)}$ since $E_{ji}E_{ii} =E_{jj} E_{ji}E_{ii}$. 
In fact, $\pi(1_A E_{ji})$ maps unitarily $\cH_\pi^{(i)}$ onto  $\cH_\pi^{(j)} $ with inverse $\pi(1_A E_{ij})$ since $E_{ij}E_{ji}=E_{ii}$ and $E_{ji}E_{ij}=E_{jj}$.

Let us fix a non-zero vector $\xi^{(1)}\in \cH_\pi^{(1)}$.
Let $\overline{\pi(\cM_N(A)) \xi^{(1)}} $ be
the closed subspace  of $\cH_\pi$ generated by $\xi^{(1)}$ under $\pi$.
Its orthogonal projection on $\cH_\pi^{(j)}$ is
$$
\pi(1_A E_{jj})\overline{\pi(\cM_N(A)) \xi^{(1)}}
=
\overline{\pi(1_A E_{jj})  \oplus_{l,k}\pi(A \id_N)\pi(1_A E_{lk}) \xi^{(1)}}
=
 \overline{\pi(A \id_N)\pi(1_A E_{j1}) \xi^{(1)}},
$$
that is, the closed subspace in $\cH_\pi^{(j)}$ generated by
$\xi^{(j)}:=\pi(1_A E_{j1}) \xi^{(1)}$
 under the action of $A$ given by the restriction of $\pi$ to $A\id_N$.
 All these orthogonal projections are unitarily isomorphic:
$$
\pi(1_A E_{jj})\overline{\pi(\cM_N(A)) \xi^{(1)}}
=
 \overline{\pi(A \id_N) \xi^{(j)}}
 =
 \pi(1_A E_{j1})\overline{\pi(A \id_N) \xi^{(1)}}.
 $$
So, as vector spaces and in terms of actions of $A$ via $\pi|_{A\id_N}$, 
 $\overline{\pi(\cM_N(A)) \xi^{(1)}} $  is isomorphic to $N$ copies of $\overline{\pi(A \id_N)\xi^{(1)}}$.
Furthermore, writing any matrix $M\in \cM_N(A)$ as $M=\sum_{1\leq l,k \leq N} m_{lk} 1_A E_{lk}$ with $m_{l,k}\in A$, we have
\begin{align*}
\pi(M)  \xi^{(j)}
&=
\sum_{1\leq l,k \leq N} 
\pi(m_{lk}1_A E_{lk})
\xi^{(j)} 
=
\sum_{1\leq l,k \leq N} 
\pi(m_{lk} \id_N)\pi(1_A E_{lk} E_{j1}) \xi^{(1)} 
\\&=
\sum_{1\leq l \leq N} 
\pi(m_{lj} \id_N)\pi(1_A E_{l1}) \xi^{(1)}
=\sum_{1\leq l \leq N} 
\pi(m_{lj} \id_N)\xi^{(l)},
\end{align*}
is in $\overline{\pi(\cM_N(A)) \xi^{(1)}} $.
So $\pi$ acts on $\overline{\pi(\cM_N(A)) \xi^{(1)}} $ where it is completely determined  by its restriction to  
$\pi(AI_N)$.
This shows Part (ii). 

Part (iii) follows from Part (ii) and its proof.
Before proving Part (iv), let us observe that the last computations in Part (ii) above imply
$$
(\pi(M)\xi^{(1)},\xi^{(1)})_{\cH_\pi}= (\pi(m_{11}\id_N)\xi^{(1)},\xi^{(1)})_{\cH_\pi}.
$$
A form of converse of this property consists in noticing that if $\ell$ is a state of the $C^*$-algebra $A$, then 
the functional $L$ defined on $\cM_N(A)$ via 
$L(M)=\ell(m_{11})$ is a state of the $C^*$-algebra $\cM_N(A)$.
More generally, if $V\in \cM_{N,1}(A)$ is a fixed vector valued in $A$ and $\ell$ is a state of the $C^*$-algebra $A$, then 
the functional $L$ defined on $\cM_N(A)$ via 
$L(M)=\ell(V^* M V)$ is a state of the $C^*$-algebra $\cM_N(A)$
provided that $\ell (V^* V)=1$.

Let us now prove Part (iv).
Let $L$ be a state of $\cM_N(A)$.
We set $\Gamma_{ji}:=L (1_A E_{ij})$ for $1\leq i,j\leq N$, 
and consider the matrix $\Gamma=(\Gamma_{ij})\in \cM_N(\bC)$.
Since $\overline{L(M)}=L(M^*)$, the matrix $\Gamma=\Gamma^*$ is Hermitian so 
there exists a $N\times N$ unitary matrix  $P$ such that $P^* \Gamma P$ is diagonal. 
We may replace $L$ by $M\mapsto L(P M P^*)$ and assume that $\Gamma$ is diagonal.
Since $L$ is a positive linear functional, so is its restriction $M\mapsto \tr (M \Gamma)$ to $\cM_N(\bC)$ and this implies $\Gamma\geq 0$. Furthermore as $L$ is a state, $\tr \Gamma=1$.
So we may assume that $\Gamma = \diag (\lambda_1,\ldots, \lambda_N)$
with $\lambda_1\geq \lambda_2\geq \ldots \geq \lambda_N \geq 0$ 
and $\lambda_1+  \ldots + \lambda_N=1$.

We now assume that $L$ is pure.
Let $\pi$ be the  irreducible representation of the $C^*$-algebra $\cM_N(A)$
and  $\xi$ the unit vector associated with $L$.
We can decompose 
$$ \xi=\xi_1+\ldots+\xi_N, 
\quad\mbox{where}\quad \xi_j:=\pi(1_AE_{jj})\xi \in \cH_\pi^{(j)}:=\pi(1_A E_{jj})\cH_\pi.
$$
We have $1=\|\xi\|_{\cH_\pi}^2=\|\xi_1\|_{\cH_\pi}^2+\ldots+\|\xi_N\|_{\cH_\pi}^2 = \lambda_1+\ldots+\lambda_N$ and more generally
$$
(\xi_i,\xi_j)_{\cH_\pi}
=(\pi(1_A E_{ij}) \xi,\xi)_{\cH_\pi}
=L(1_A E_{ij}) = \lambda_j \delta_{i=j}.
$$
Necessarily $\xi_1\not=0$.
Naturally, $\xi_i\in \cH^{(i)}_\pi$ is orthogonal to $\xi^{(j)}:=\pi(1_A E_{j1})\xi_1=\pi(1_A E_{j1})\xi\in \cH_\pi^{(j)}$ if $i\not=j$.
And for $j=2,\ldots, N$, $\xi_j$ is orthogonal to $\xi^{(j)}$ since
$$
(\xi_j, \xi^{(j)})_{\cH_\pi}
=
(\pi(1_A E_{jj}) \xi, \pi(1_A E_{j1} \xi )_{\cH_\pi}
= 
(\pi(1_A E_{1j}) \xi,  \xi )_{\cH_\pi}
= L(1_A E_{1j}) =\delta_{1=j}.
$$
This shows that 
if $\xi_2\not=0$ then the space $\overline{\pi(\cM_N(A)) \xi_2}$ generated by $\xi_2$ under the representation $\pi$ of $\cM_N(A)$ will be non-zero and distinct from $\overline{\pi(\cM_N(A)) \xi_1}$,
contradicting the irreducibility of $\pi$. 
Therefore $0=\xi_2=\ldots=\xi_N$ and $\xi=\xi_1$ is unitary. Furthermore  we have:
$$
L(M)
=\sum_{1\leq i,j\leq N} ( \pi(m_{ij}E_{ij})\xi,\xi)
=\sum_{1\leq i,j\leq N} ( \pi(m_{ij}E_{ij})\xi_j,\xi_i)
=( \pi(m_{11}E_{11})\xi_1,\xi_1)
=\ell_1(m_{11}),
$$
where $\ell_1$ is the state of $A$ associated 
with the restriction of $\pi$ to $A$ on $\overline{\pi(A\id_N) \xi_1}$
and the unit vector $\xi_1$.
With the notation of the statement, this shows $L=L_{\ell_1,V}$ with $V=e_1 1_A$ where $e_1$ is the first vector of the canonical basis of $\bC^N$;
 one checks easily $\ell_1 (V^* V)=\ell_1(1_A)=1$.
This concludes the proof of Part (iv) and of Lemma \ref{lem_MN(A)}.
\end{proof}

\subsection{Microlocal defect measures of vector-valued sequences and localisation properties}

Let us now go back to the family  $(U_k)_{k\in \bN}$ in $L^2(\Omega)^N$
where $\Omega$ is an open subset of $G$. We are concerned with the limit of quantities of the form 
$$\left({\rm Op}(\sigma)U_k,U_k\right)_{L^2(\Omega)^N}$$
for $\sigma\in\cM_N(\dot S^0_{cl}(\Omega))$.
In view of the description of the preceding section, a MDM of $(U_k)_{k\in \bN}$ is a pair $(\Gamma,\gamma)$ consisting of a positive Radon measure $\gamma$ and of a $N\times N$ matrix of $\gamma$-integrable fields of non-negative trace-class operators $\Gamma=(\Gamma_{i,j})_{1\leq i,j\leq N}$, such that, up to a subsequence,  for all $\sigma \in \cM_{P,Q}(\dot S^0_{cl}(\Omega))$, 
$$\lim_{k\to\infty} \left({\rm Op}(\sigma)U_k,U_k\right)_{L^2(\Omega)^N}
= \int_{\Omega \times (\Gh /\bR^{+})}
\trN \left(\sigma (x,\dot \pi) \ \Gamma(x,\dot \pi) \right)
d  \gamma(x,\dot\pi) $$
where the trace $\trN$ denotes the trace of operators of  ${\mathcal L}^1(\cH_\pi^N)$: if $\sigma = (\sigma_{i,j})_{1\leq i,j\leq N},\;\Gamma = (\Gamma_{i,j})_{1\leq i,j\leq N} $,
$$\trN \left(\sigma (x,\dot \pi) \ \Gamma(x,\dot \pi) \right)= \sum_{1\leq i,j\leq N} 
\tr \left(\sigma_{i,j} (x,\dot \pi) \ \Gamma_{j,i}(x,\dot \pi) \right).$$
Since $U_k=(u_1^k,\cdots, u_N^k)$, $\Gamma_{i,j}d\gamma$ describes the limit of quantities $\left({\rm Op}(\sigma)u_i^k,u_j^k\right)_{L^2(\Omega)}$ for scalar symbols $\sigma$. The field $\Gamma_{i,j}d\gamma$ is the joint measure of the sequences $(u_i^k)$ and $(u_j^k)$. 

\medskip
 
 Let us now consider a matrix-valued operator  $P$  consisting of $K$ lines and $N$ rows of differential operators of order~$m$ such that 
$(PU_k)_{k\in \bN}$ converges to 0  in $L^2_{-m}(\Omega,loc)^K$ as $k\rightarrow +\infty$.
Recall that $L^2_s(\Omega,s)$ was defined in Definition \ref{def_L2sloc}.
 If the family $(U_k)_{k\in \bN}$ solves a differential equation,  in the sense that $PU_k$ tends to $0$ in $L^2_{-m}(G)^K$, then the MDM  $\Gamma d\gamma$ satisfies the following localization property.

\begin{proposition} \label{prop_localisation}
Let $p(x,\pi)$ be  the principal symbol of the matrix-valued differential operator $P\in \cM_{K,N}(S^m(\Omega))$
where $\Omega$ is an open subset of $G$.
We assume that the family $(U_k)_{k\in \bN}$ in $L^2(\Omega)^N$ is such that $PU_k$ tends to $PU$ in $L^2_{-m}(G)^K$.
Let $\Gamma d\gamma$ be a MDM of $(U_k)_{k\in\bN}$, then
$$
 p_0(x,\dot\pi)\Gamma (x,\dot \pi)p_0(x,\dot\pi)^*=0,\;d\gamma(x,\dot\pi) \; a.e.,
$$ where 
 $p_0(x,\dot\pi):=\pi(\cR)^{-{m\over\nu}} p(x,\dot\pi) $
 for any positive Rockland operator $\cR$.
\end{proposition}

\begin{proof}
We may assume that the sequence $(U_k)$ is pure.
The equation satisfied by $(U_k)_{k\in\bN}$ implies that for any 
$\sigma\in \cM_K(S^0_{cl}(\Omega))$
we have
$$
\lim_{k\to\infty}\left(\Op(\sigma) (\id+\cR)^{-{m\over\nu}}  PU_k, (\id+\cR)^{-{m\over\nu}} PU_k\right)_{L^2(\Omega)^K}=0.$$
By the definition of $\Gamma d\gamma$, we deduce 
$$\int_{\Omega\times (\Gh /\bR^{+})}
\trN \left(p(x,\dot\pi)^*\pi(\cR)^{-{m\over\nu}} \sigma_0 (x,\dot \pi) \pi(\cR)^{-{m\over\nu}} p(x,\dot\pi)\Gamma(x,\dot \pi) \right)
d  \gamma(x,\dot\pi) =0,$$
and this relation holds for any  $\sigma_0\in \cM_K(\dot S^0(\Omega))$, 
and this implies the result.
\end{proof}

\subsection{Compensated compactness}

The issue of compensated compactness result is to pass to the limit in quantities of the form 
$$\int _\Omega \phi(x)(q(x) U_k(x),U_k(x))_{\bC^N}dx,$$
for some compactly supported scalar-valued smooth  function $\phi$ and for smooth bounded matrix-valued function $q\in\cM_N(\cD(\Omega))$. The aim is to find conditions on the matrix $q$ which allow to pass to the limit in terms of weak limits $U$ of $(U_k)_{k\in\bN}$. 
The next proposition is a  compensated compactness result.
Recall that the spaces $L^2_s(\Omega, loc)$ were defined in Definition \ref{def_L2sloc}.

\begin{proposition}\label{prop:comp}
Let $p(x,\pi)$ be  the principal symbol of the matrix-valued differential operator $P\in \cM_{K,N}(S^m(\Omega))$
where $\Omega$ is an open subset of $G$.
Let $(U_k)_{k\in \bN}$ be a sequence in $L^2(\Omega,loc)^N$ which converges to $U$ weakly in $L^2(\Omega,loc)^N$
and such that $(PU_k)_{k\in \bN}$ converges to $PU$ in $L^2_{-m}(\Omega, loc)^K$.
\begin{itemize}
\item[(i)] 
Let $q\in\cM_N(\cC^\infty(\Omega))$ be such that 
for all $x\in \Omega$, $\pi\in \Gh$
and $h\in (\cH_\pi^\infty)^N$,
we have
$$
p(x,\pi)h=0 \Longrightarrow \left( q(x) h,h\right)_{\cH_\pi^N}= 0.
$$
Then the sequence of smooth functions given by $x\mapsto (q(x) U_k(x),U_k(x))_{\bC^N}$ converges to 
 $x\mapsto (q(x) U(x),U(x))_{\bC^N}$ in $\cD'(\Omega)$.
 \item[(ii)] 
Let $q\in\cM_N(\cC^\infty(\Omega))$ be such that $q^*=q$ and satisfying
for all $x\in \Omega$, $\pi\in \Gh$
and $h\in (\cH_\pi^\infty)^N$,
$$
p(x,\pi)h=0 \Longrightarrow \left( q(x) h,h\right)_{\cH_\pi^N}\geq 0.
$$
Then, for any non-negative $\phi\in\cD(\Omega) $,
 $$
 \liminf_{k\to\infty} \int _\Omega 
 \phi(x)(q(x) U_k(x),U_k(x))_{\bC^N}dx 
 \geq  
 \int_\Omega\phi(x) (q(x) U(x),U(x))_{\bC^N}dx.
 $$
\end{itemize}

\end{proposition}

\begin{proof}[Proof of Proposition~\ref{prop:comp}]
The proof follows the lines of~\cite[Theorem~2]{gerard_91}. 
Part (i) follows from Part (ii), using in particular the decomposition of $q=q_1+iq_2$ with $q_1,q_2$ smooth function valued in the space of Hermitian $N\times N$-matrices.
 So we just have to prove Part (ii).
We may assume that the sequence $(U_k)$ is pure and that $U=0$. As a consequence, we know that 
$$
\lim_{k\to \infty} 
\int _\Omega \phi(x)(q(x) U_k(x),U_k(x))_{\bC^N}dx
= 
\int_{\Omega\times \Gh/\bR^+}
\phi(x) {\tr}(q(x) \Gamma(x,\dot\pi)) d\gamma(x,\dot\pi)
$$
and our aim is to show that the right-hand side of the preceding equality is $0$.
 The proof comes from the following observation: 
fixing $\Omega'$  a bounded open subset of $G$ whose closure is included in~$\Omega$, 
 we have
\begin{align}
\forall \eps>0,\quad \exists C_\eps>0,\quad \forall (x,\dot\pi)\in \Omega'\times \Gh/\bR^+,\quad\forall h\in\cH_\pi,\quad \quad\quad\nonumber
\\ (q(x)h,h)_{\cH_\pi^N}+C_\eps\| p_0(x,\dot\pi) h\| ^2_{\cH_\pi^K} \geq -\eps \| h\|^2_{\cH_\pi^N}, 
\label{compactness_bis}
\end{align}
where  $p_0(x,\dot\pi)=\pi(\cR)^{-{m\over\nu}} p(x,\dot\pi) $.
Indeed, this equation yields the positivity of the operator 
$$R_\eps(x,\dot \pi) = q(x) + C_\eps p_0^*(x,\dot\pi) p_0(x,\dot \pi) + \eps {\rm Id}$$
and  we deduce 
$$ 
\int_{\Omega\times \Gh/\bR^+}
\phi(x) \trN(R_\eps(x,\dot\pi) \Gamma(x,\dot\pi)) d\gamma(x,\dot\pi) \geq 0  .
$$
On the other hand, 
$$
\trN(R_\eps(x,\dot\pi) \Gamma(x,\dot\pi))= 
\trN(q(x) \Gamma(x,\dot\pi))
+ C_\eps \trK(p_0(x,\dot\pi) \Gamma(x,\dot\pi)p_0^*(x,\dot\pi)) 
+\eps\, \tr(\Gamma(x,\dot\pi)).
$$
By Proposition \ref{prop_localisation}, we obtain that for all $\phi\in \cD(\Omega')$ and $\eps>0$, 
$$ 
\int_{\Omega\times \Gh/\bR^+}
\phi(x) \trN(q(x)\Gamma(x,\dot\pi)) d\gamma(x,\dot\pi) 
\geq -\eps  
\int_{\Omega\times \Gh/\bR^+}\phi(x) {\tr}\Gamma(x,\dot\pi) \ d\gamma(x,\dot\pi),
$$
which allows us to conclude. 

\medskip 

It remains to prove~(\ref{compactness_bis}). We first note that the injectivity of $\pi(\cR)$ implies that the property on~$q$ with respect to $p$ is also satisfied by $p_0$,  i.e. 
for all $x\in \Omega$, $\pi\in \Gh$
and $h\in \cH_\pi^N$,
$$
 p_0(x,\pi)h=0 \Longrightarrow \left( q(x) h,h\right)_{\cH_\pi^N}\geq  0.
$$
We argue by contradiction: if~(\ref{compactness_bis}) is false, then there exists  $\eps_0>0$
and sequences $(x_n)_{n\in \bN}$ of $\Omega'$
and 
$(\dot\pi_n)_{n\in \bN}$ of $\Gh/\bR^+$, 
together with unit vectors  $h_n\in \cH_{\pi_n}^N$ for each $n\in \bN$
  such that 
$$
(q(x_n)h_n,h_n)_{\cH_{\pi_n}^N}  + n\| p_0(x_n,\dot\pi_n)h_n\|_{\cH_{\pi_n}^N} ^2<-\eps_0.
$$
We interpret these sequences as the data of a sequence of states $(L_n)_{n\in\bN}$ of the $C^*$-algebra $C^*(\cM_N(\dot S^0(\Omega')))$, defined by 
$$\forall \sigma \in C^*(\cM_N(\dot S^0(\Omega'))),\;\;L_n(\sigma) = \left(\sigma(x_n,\dot\pi_n)h_n,h_n\right)_{\cH_\pi^N}.$$
We have 
$$\forall n\in\bN,\;\; L_n(q) + n L_n(p_0^*p_0)<-\eps_0,$$
in particular $L_n(q)<-\eps_0 $ for all $ n\in\bN$.
We extract a weak-$*$ converging sub-sequence from of $(L_n)_{n\in\bN}$
and
we denote its weak limit by $L$. Note that $L$ is a state and it satisfies $L(q)\leq -\eps_0$ and $L(p^*_0p_0)=0$. 
Desintegrating $L$ into pure states \cite[\S 8.8]{Dixmier_C*}
and 
combining  Lemma \ref{lem_MN(A)} with Proposition~\ref{prop_state_C*dotS0Omega},  we obtain $L$ as an integral of states of the form 
$\sigma \mapsto 
(V^*(x_0,\dot \pi_0) \sigma  (x_0,\dot \pi_0)  V(x_0,\dot \pi_0) v_0, v_0)_{\cH_{\pi_0}}$
against a positive measure $\nu$. 
Since $L(p^*_0p_0)=0$, we have:
$$
(V^*(x_0,\dot \pi_0) p_0^*p_0  (x_0,\dot \pi_0)  V(x_0,\dot \pi_0) v_0, v_0)_{\cH_{\pi_0}}
=
\|p_0  (x_0,\dot \pi_0)  V(x_0,\dot \pi_0) v_0\|^2_{\cH_{\pi_0}^K} =0
\quad \nu\mbox{-a.e.}
$$
But our hypothesis implies 
$$
\left( q(x_0) V(x_0,\dot \pi_0) v_0,V(x_0,\dot \pi_0) v_0\right)_{\cH_\pi^N}\geq  0
\quad \nu\mbox{-a.e.},
$$
and therefore $L(q)\geq0$.
This contradicts $L(q)\leq -\eps_0<0$.
Hence (\ref{compactness_bis}) is proved and this concludes the proof of 
Proposition \ref{prop:comp}.
\end{proof}

\subsection{Link with {\it div-curl} results}

Our result below gives a new approach to {\it div-curl lemma}, which had already been considered from a more geometric (sub-Riemannian) perspective in \cite{baldi_franchi,baldi_franchi_tchou_tesi}. 
We assume that $G$ is a stratified Lie  group. We fix a canonical basis $X_1,\ldots, X_{n_1}$ on the first stratum. Then the divergence operator is defined by 
$$
\forall f=(f_1,\cdots,f_{n_1})\in \cS(G)^{n_1},\quad 
\div (f)=X_1 f_1 + \cdots + X_{n_1} f_{n_1}.
$$
We denote by $\pi(\div)$ the symbol of the operator $\div$. This symbol is a vector of $n_1$ symbols of order~$1$, i.e.  $\div \in \cM_{1,n_1}(S_{cl}^1(G))$.
We define the curl-property as follows:
\begin{definition}
\label{def_curl_property}
Let $\Omega$ be an open subset of $G$.
A $n_1\times n_1$ matrix   $\rho(x,\pi)\in \cM_{n_1}(S^m_{cl}(\Omega))$ of symbols of order~$m$ satisfies the \emph{curl}-property when, for all $x\in \Omega$, $\pi\in \Gh$, $h_1,h_2\in (\cH_\pi^\infty)^{n_1}$:
$$
\pi(\div)\cdot h_1=0\;{\rm and} \; \rho(x,\pi) h_2=0 \Longrightarrow (h_1,h_2)_{\cH_{\pi}^{n_1}}=0.
$$
\end{definition}

Recall that the spaces $L^2_s(\Omega, loc)$ were defined in Definition \ref{def_L2sloc}.
We have the following {\it div-curl} type result. 

\begin{proposition}\label{divcurl}
Let $\Omega$ be an open subset of $G$.
Let $(V_k)_{k\in\bN}$ and $(W_k)_{k\in\bN}$ be two bounded pure families of $L^2(\Omega,loc)^{n_1}$  with weak limits $V$ and $W$ respectively.
We assume
that  the sequence of scalar functions $(\div(V_k))_{k\in \bN}$ converges  to $\div(V)$  in $L^2_{-1}(\Omega,loc)$ as $k\to +\infty$
 and 
 that  $(\Op(\rho)W_k)_{k\in \bN}$ converges to $\Op(\rho)W$ in $L^2_{-m}(G,loc)^{n_1}$ 
where $\rho(\pi)\in \cM_{n_1\times n_1}(S^m_{cl}(\Omega))$ satisfies the curl-property, cf Definition \ref{def_curl_property}.
Then the sequence of functions given by $x\mapsto \phi(x) ( V_k(x),W_k(x))_{\bC^{n_1}}$ converges to
$x\mapsto \phi(x)( V(x),W(x))_{\bC^{n_1}}$ in the sense of distribution on $\Omega$ as $k\to\infty$.
\end{proposition}

\begin{proof}
We set:
$$
p(x,\pi):= \begin{pmatrix}
{\mathcal R}(\pi)^{ m-1}\pi(\div) &  0 \\
0 & \rho(x,\pi)
\end{pmatrix} 
\in \cM_{n_1+1,2n_1}(S^m_{cl}(\Omega))
\quad\mbox{and}\quad
 q(x):=\begin{pmatrix} 0 &  0 \\
\id_{n_1} & 0
\end{pmatrix}
\in \cM_{2n_1}(\bC).
$$
Then, for any $h=(h_1,h_2)\in (\cH_\pi^\infty)^{n_1}\times (\cH_\pi^\infty)^{n_1}=(\cH_\pi^\infty)^{2n_1} $, we have 
$$
(q(x)h,h)_{\cH_{\pi}^{2n_1}}=(h_1,h_2)_{\cH_{\pi}^{n_1}},
$$
and
$$
p(x,\pi)h=0
\quad \Longleftrightarrow \quad
\pi(\div)\cdot h_1=0
\;\; {\rm and} \;\; \rho(x,\pi) h_2=0.
$$
The curl-property allows us to apply Proposition~\ref{prop:comp} Part (i) to the sequence $(U_k)_{k\in \bN}$ given by $U_k=(V_k,W_k)\in \bC^{2n_1}$, the operator $P={\rm Op}(p)$ and the matrix-valued function $q(x)$.
The statement follows.
\end{proof}


\bibliographystyle{amsplain}

\end{document}